\documentclass[11pt]{amsart}
\usepackage{hyperref}
\usepackage{amssymb,amsfonts,amsmath,amsopn,amstext,amscd,latexsym}
\usepackage{bbm}
\usepackage[all,cmtip]{xy}
\usepackage{enumerate}
\usepackage[shortlabels]{enumitem}

\theoremstyle{plain}

\newtheorem{theorem}{Theorem}[section]
\newtheorem{thm}[theorem]{Theorem}
\newtheorem{lemma}[theorem]{Lemma}
\newtheorem{lem}[theorem]{Lemma}

\newtheorem{proposition}[theorem]{Proposition}
\newtheorem{prop}[theorem]{Proposition}

\newtheorem{cor}[theorem]{Corollary}

\theoremstyle{remark}
\newtheorem{example}[theorem]{Example}
\newtheorem{remark}[theorem]{Remark}

\DeclareMathOperator{\Hom}{Hom}

\DeclareMathOperator{\Aut}{Aut}

\DeclareMathOperator{\Stab}{Stab}

\DeclareMathOperator{\diag}{diag}
\DeclareMathOperator{\Irr}{Irr}
\DeclareMathOperator{\IBr}{IBr}

\DeclareMathOperator{\Ext}{Ext}
\DeclareMathOperator{\ext}{Ext}

\DeclareMathOperator{\End}{End}
\DeclareMathOperator{\tr}{tr}

\DeclareMathOperator{\Ind}{Ind}

\DeclareMathOperator{\GL}{GL}
\DeclareMathOperator{\SL}{SL}
\DeclareMathOperator{\SU}{SU}
\DeclareMathOperator{\SO}{SO}

\newcommand{\cA}{{\mathcal A}}
\newcommand{\cB}{{\mathcal B}}

\newcommand{\cD}{{\mathcal D}}
\newcommand{\cE}{{\mathcal E}}

\newcommand{\cG}{{\mathcal G}}
\newcommand{\cH}{{\mathcal H}}

\newcommand{\cM}{{\mathcal M}}
\newcommand{\cN}{{\mathcal N}}
\newcommand{\cO}{{\mathcal O}}
\newcommand{\cP}{{\mathcal P}}
\newcommand{\cQ}{{\mathcal Q}}

\newcommand{\cT}{{\mathcal T}}
\newcommand{\cU}{{\mathcal U}}

\newcommand{\cX}{{\mathcal X}}

\newcommand{\cZ}{{\mathcal Z}}

\newcommand{\gtwo}{{\mathrm{G}_2}}
\newcommand{\ff}{{\mathrm{F}_4}}
\newcommand{\df}{{\mathrm{D}_4}}
\newcommand{\btwo}{{\mathrm{B}_2}}
\newcommand{\abs}{{\rm {disc}}}
\newcommand{\rat}{{\rm {rat}}}

\newcommand{\bfC}{{\mathbf C}}

\newcommand{\bbF}{{\mathbb F}}

\newcommand{\bbN}{{\mathbb N}}
\newcommand{\bfN}{{\mathbf N}}

\newcommand{\bfO}{{\mathbf O}}
\newcommand{\bbP}{{\mathbb P}}
\newcommand{\bbQ}{{\mathbb Q}}

\newcommand{\bfZ}{{\mathbf Z}}
\newcommand{\AAA}{{\sf A}}
\newcommand{\SSS}{{\sf S}}
\newcommand{\bbone}{{\mathbbm{1}}}

\newcommand{\rss}{{\rm{ss}}}
\newcommand{\barFp}{{\overline{\bbF}_p}}

\newcommand{\eps}{\epsilon}
\newcommand{\la}{\lambda}

\newcommand{\p}{\partial}

\newcommand{\image}{{\mathrm {Im}}}
\newcommand{\Ker}{{\mathrm {Ker}}}
\newcommand{\Syl}{{\mathrm {Syl}}}
\newcommand{\Sym}{{\mathrm {Sym}}}
\newcommand{\Char}{{\mathrm {char}}}
\newcommand{\GP}{{G^+}}
\DeclareMathOperator{\Sp}{Sp}

\DeclareMathOperator{\soc}{soc}
\DeclareMathOperator{\PSL}{PSL}
\DeclareMathOperator{\PSU}{PSU}
\DeclareMathOperator{\PSp}{PSp}
\DeclareMathOperator{\PGL}{PGL}
\DeclareMathOperator{\GU}{GU}

\DeclareMathOperator{\OO}{O}
\DeclareMathOperator{\PIM}{\cP}

\DeclareMathOperator{\EE}{End}

\theoremstyle{plain}

\renewcommand{\mod}{\bmod \,}

\newcommand{\fpb}{\overline{\bbF}_p}

\newcommand{\tw}[1]{{}^#1}
\numberwithin{equation}{section}
\def\skipa{\vspace{-1.5mm} & \vspace{-1.5mm} & \vspace{-1.5mm}\\}
\newcommand{\dl}{\mathfrak{d}_p}
\newcommand{\codim}{{\mathrm {codim}}}
\newcommand{\St}{{\mathsf {St}}}
\newcommand{\rad}{{\mathrm {rad}}}
\newcommand{\hd}{{\mathrm {head}}}
\newcommand{\Out}{{\mathrm {Out}}}

\begin{document}

\author[R. Guralnick]{Robert Guralnick}
\address{Department of Mathematics, University of Southern California,
Los Angeles, CA 90089-2532, USA}
\email{guralnic@usc.edu}
\author[F. Herzig]{Florian Herzig}
\address{Department of Mathematics, University of Toronto,
40 St. George Street, Room 6290, Toronto, ON M5S 2E4, Canada}
\email{herzig@math.toronto.edu}
\author[P. Tiep]{Pham Huu Tiep}
\address{Department of Mathematics, University of Arizona,
Tucson, AZ 85721-0089, USA}
\email{tiep@math.arizona.edu}

\title{Adequate Subgroups and Indecomposable Modules}
\date{\today}

\thanks{The first author was partially supported by  NSF
  grants DMS-1001962, DMS-1302886 and the Simons Foundation Fellowship 224965.  He
  also thanks the Institute for Advanced Study for its support.
  The second author was partially supported by a Sloan Fellowship and an NSERC grant.
  The third author was partially supported by the NSF grant DMS-1201374 and the Simons Foundation
Fellowship 305247.}
\thanks{We thank Barry Mazur and Jack Thorne for thoughtful discussion of various questions 
considered in the paper, and Frank L\"ubeck and Klaus Lux for help with several computations. 
We also thank the referee for careful reading of the paper.}



\begin{abstract}
The notion of adequate subgroups was introduced by Jack Thorne \cite{T}.
It is a weakening of the notion of big subgroups used by Wiles and Taylor
in proving   automorphy lifting theorems for certain Galois representations.
Using this idea, Thorne was able to strengthen
many automorphy lifting theorems.
 It was shown in \cite{GHTT} and \cite{GHT} that if the dimension is smaller than 
 the characteristic then almost all absolutely irreducible representations are
 adequate.   We extend the results by considering all absolutely
 irreducible modules in characteristic $p$  of dimension $p$.  This relies on a modified definition of 
 adequacy, provided by Thorne in \cite{T2}, which allows $p$ to divide the 
 dimension of the module. We  prove adequacy for
 almost all irreducible representations of 
 $\SL_2(p^a)$ in the natural characteristic and for finite groups of Lie type as long as the field
 of definition  is sufficiently large. 
 We also essentially classify indecomposable modules in characteristic $p$ of dimension
 less than $2p-2$ and answer a question of Serre concerning complete reducibility of 
 subgroups in classical groups of low dimension.   
 \end{abstract}

\maketitle

\tableofcontents

\section{Introduction}

Throughout the paper,
let $k$ be a field of characteristic $p$ and let $V$ be a finite dimensional vector space over $k$.
Let $\rho:G \rightarrow  \GL(V)$ be an  absolutely irreducible representation.
Thorne  \cite{T} called $(G,V)$ is {\it adequate}  if the following conditions hold
(we rephrase the conditions slightly by combining two of the properties into one):

\begin{enumerate}
\item  $p$ does not divide $\dim V$;
\item  $\ext^1_G(V,V)=0$; and
\item  $\EE(V)$ is spanned by the elements $\rho(g)$ with $\rho(g)$ semisimple.
\end{enumerate}

If $G$ is a finite group of order prime to $p$, then it is well known that $(G,V)$ is adequate.
In  this case, condition (iii) is often referred to as Burnside's Lemma, and it is a trivial
consequence of the  Artin-Wedderburn Theorem. Furthermore, if $G$ is a connected 
algebraic group over $k$ and $V$ is a faithful, absolutely irreducible rational $G$-module 
of dimension coprime to $p$, then $(G,V)$ is adequate (cf. \cite[Theorem 1.2]{Gapp} and 
Theorem \ref{thm:asymp}).

These conditions are a weakening of the conditions used by Wiles and Taylor
in studying the automorphic lifts of certain Galois representations.    See
\cite{CHT} for some applications.  Thorne \cite{T}
generalized various results assuming
the weaker hypotheses for $p$ odd.  We refer the reader to \cite{T} for more references and details.
See also \cite{D} for further applications.  Recently Thorne \cite[Corollary 7.3]{T2} has shown that one
can relax the condition that $p \nmid (\dim V)$, still with $p$ odd. 
   So more generally,
we say that an absolutely irreducible
representation $\rho:G \rightarrow  \GL(V)$ is {\it adequate} if:

\begin{enumerate}
 \item  $H^1(G,k)=0$;
 \item  $H^1(G, (V^* \otimes V)/k)=0$;
 \item  $\EE(V)$ is spanned by the elements $\rho(g)$ with $\rho(g)$ semisimple.
\end{enumerate}

Note that we allow the case $p=2$ in the definition. Thorne has used this extended notion of adequacy to prove an 
automorphy lifting theorem for $2$-adic Galois representations of unitary type over imaginary CM fields, see 
\cite[Theorem 5.1]{T2}.

Observe that if $p \nmid (\dim V)$,  $k$ is a direct summand of $V^* \otimes V$.
Thus, $\ext_G^1(V,V)=0$ implies that $H^1(G,k)=0$ in this case.  Also note that,
by the long exact sequence in cohomology, if $H^2(G,k)=0$, then $H^1(G, (V^* \otimes V)/k)=0$
follows from $\ext_G^1(V,V)=0$.   Thus, under the assumption that
either $p \nmid (\dim V)$ or
$H^i(G,k)=0$ for $i=1,2$, adequacy is equivalent to the two conditions:

\begin{enumerate}
 \item  $\ext_G^1(V,V)=0$;
 \item  $\EE(V)$ is spanned by the elements $\rho(g)$ with $\rho(g)$ semisimple.
\end{enumerate}

Following \cite{G2},  we say that the representation $\rho:G \to \GL(V)$, 
respectively the pair $(G,V)$, is {\it weakly adequate} if $\EE(V)$ is
spanned by the elements $\rho(g)$ with $\rho(g)$ semisimple.

It was shown  in \cite[Theorem 9]{GHTT} that:

\begin{theorem} \label{adequate}  Let $k$ be a field of characteristic $p$ and
$G$ a finite group.  Let $V$ be an absolutely irreducible faithful $kG$-module.
Let $G^+$ denote the subgroup generated by the $p$-elements of $G$.  If
$\dim W \leq (p-3)/2$ for an absolutely irreducible $kG^+$-submodule $W$ of $V$,
then   $(G,V)$ is adequate.
\end{theorem}

The example $G=\SL_2(p)$ with $V$ irreducible of dimension $(p-1)/2$ shows
that the previous theorem is best possible.  However, the counterexamples are
rare. In fact, as shown in \cite[Corollary 1.5]{GHT}, if $\dim V < p-3$, then the 
$(p\pm 1)/2$-dimensional representations of $\SL_2(p)$ are the only two 
counterexamples. More precisely, in \cite{GHT} we extend Theorem \ref{adequate} to  
the more general situation that $\dim W < p$ and show that almost always
$(G,V)$ is adequate:

\begin{thm} \label{ght:adequate}
Let $k$ be a field of characteristic $p$ and $G$ a finite group.  Let $V$ be an absolutely irreducible faithful $kG$-module, and let $\GP$ denote the subgroup generated by the $p$-elements of $G$.
Suppose that the dimension $d$ of any irreducible $k\GP$-submodule in $V$ is less than $p$. 
Then the following statements hold.
\begin{enumerate}[\rm(i)]
\item $(G,V)$ is weakly adequate.
\item Let $W$ be an irreducible $\overline{k}\GP$-submodule of $V \otimes_k \overline{k}$. Then
$(G,V)$ is adequate, unless the group
$H < \GL(W)$ induced by the action of $\GP$ on $W$ is as described in one of the
exceptional cases {\rm (a), (b)(i)--(vi)} listed in \cite[Theorem 1.3]{GHT}. In particular, if 
$d < p-3$ and $(G,V)$ is not adequate, then $d = (p \pm 1)/2$ and $H \cong \SL_2(p)$ or 
$\PSL_2(p)$.
\end{enumerate}
\end{thm}

Above the threshold $p-1$ for $\dim W$, there are lots of linear groups that are not adequate. 
Still, if $\dim V = p$, the situation is very much under control. In this paper,  we extend adequacy results
to the case of linear groups of degree $p$ and generalize the asymptotic result
\cite[Theorem 1.2]{Gapp} to disconnected algebraic groups $\cG$
(with $p\nmid [\cG:\cG^0]$) allowing at the same time that $p$ divides the dimension of the
$\cG$-module. Next, we show that, in all cases considered in Theorem \ref{ght:adequate},
under some additional mild condition (say, $G$ is not $p$-solvable if $p$ is a Fermat prime,  
and $p > 5$), one in fact has 
$\dim \Ext^1_G(V,V) \leq 1$ -- a result of interest in the deformation theory.
An outgrowth of our results leads us to prove 
an analogue of the first author's result \cite{Gcr} and answer a question of Serre on
complete reducibility of finite subgroups of orthogonal and symplectic groups of small degree.
In fact, we essentially classify indecomposable modules
in characteristic $p$ of dimension less than $2p-2$. 

Note that if the kernel of $\rho$ has order prime to $p$, then there is no harm
in passing to the quotient.  So we will generally assume that either
$\rho$ is faithful or more generally has kernel  of order prime to $p$.
Also, note that the dimension of cohomology groups and the dimension
of the span of the semisimple elements in $G$ in $\EE(V)$ does not
change under extension of scalars.  Hence, most of the time we will work over
an algebraically closed field $k$.

Our main results are the following. First we show that the condition that $H^1(G,k) =0$ in 
the definition of adequacy is not particularly constraining if
$\dim V$ is small.  In particular, the next result follows fairly easily
from \cite{Gcr} (see \cite[Theorem 4.1]{G2}).   See Theorem \ref{thm:abelian2}
for a slightly more general result.  

\begin{thm} \label{thm:abelian}  Let $G$ be a finite irreducible
subgroup of $\GL_d(k)$ with $k$ algebraically closed of
characteristic $p$.   Assume that $H^1(G,k) \ne 0$
 and $d < 2p-2$.  Then $G$ is solvable,
 $d = p-1$, $p$ or $p+1$ and one of the following holds:
\begin{enumerate}[\rm(i)]
\item  $d=p - 1$, $p = 2^a + 1$ is a Fermat prime, $[G:\bfZ(G)\bfO_2(G)]=p$ and $\bfO_2(G)$
is a group of symplectic type with $\bfO_2(G)/\bfZ(\bfO_2(G))$ (elementary) abelian
of order $2^{2a}$;
\item $d=p$ and $G$ has a normal abelian $p'$-subgroup of index $p$;
\item $d=p + 1$, $p = 2^a - 1$ is a Mersenne prime, and $G$ contains a normal 
abelian $p'$-subgroup $N$ such that $G/N$ is  a Frobenius group of order $dp$ with 
kernel of order $d$.
\end{enumerate}
\end{thm}

The following curious corollary is immediate from Theorem 
\ref{thm:abelian2}. We suspect that there is a proof of this that
does not require the classification of finite simple groups.

\begin{cor} \label{cor:p-cf} 
Let $G$ be a finite irreducible subgroup of $\GL_d(k)$
with $k$ algebraically closed of characteristic $p$ and $d < 2p-2$. Suppose
that $G$ has a composition factor of order $p$. Then $G$ is solvable.
Moreover, either $d = p$, or $d= 2^a$ with $p = d \pm 1$ 
(and so $p$ is either a Mersenne prime or Fermat prime).
\end{cor}

In the situation of Theorem \ref{ght:adequate}, $\Ext^1_G(V,V)$ may be  nonzero and so 
$G$ may fail to be adequate. Nevertheless, we can prove the following two results, which were motivated
by discussions with Mazur and which are of interest
in deformation theory. (Recall \cite[Section 1.2]{Maz}, for instance, that the inequality $\dim \Ext^1_G(V,V) \le n$ implies that the universal                
deformation ring over the ring $\cO$ of integers of a sufficiently large finite extension of $\bbQ_p$ is a                
quotient of $\cO[[x_1,\ldots,x_n]]$.  See also \cite[Theorem 2.4]{Boc}.)

\begin{thm} \label{ext}
Let $k$ be a field of characteristic $p$ and $G$ a finite group.  
Let $V$ be an absolutely irreducible faithful $kG$-module, and let $\GP$ denote the subgroup generated by the $p$-elements of $G$.
Suppose that the dimension $d$ of any irreducible $k\GP$-submodule $W$ in $V$ is less than $p$, and 
let $H$ be the image of $\GP$ in $\GL(W)$.

{\rm (i)} Suppose the following conditions hold:
  \begin{enumerate}[\rm(a)]
    \item If $p$ is a Fermat prime, then $G$ is not $p$-solvable (equivalently, $H$ is not solvable);
    \item If $p = 3$, then $H \not\cong \SL_2(3^a)$ for all $a \geq 2$;
    \item If $p = 5$ and $\dim_k W = 4$, then $H \not \cong \Omega^+_4(5)$.
  \end{enumerate}
Then  $\dim_k\Ext^1_G(V,V) \leq 1$ and $\dim_k\Ext^1_G(V,V^*) \leq 1$. In particular,
$H^1(G,\Sym^2(V))$ and $H^1(G,\wedge^2(V))$ are both at most $1$-dimensional.

{\rm (ii)} In the exceptional cases $(p,\dim_k W,H) = (5,4,\Omega^+_4(5))$ or 
$(p,H) = (3,\SL_2(3^a))$ with $a \geq 2$, $\Ext^1_G(V,V)$ and $\Ext^1_G(V,V^*)$ are at most 
$2$-dimensional.
\end{thm}

Note that one cannot remove the conditions (a)--(c) in Theorem \ref{ext}(i).
In fact, in the case $G$ is 
$p$-solvable of Theorem \ref{ext}, $\Ext^1_G(V,V)$ and $\Ext^1_G(V,V^*)$ can be of 
arbitrarily large dimension. See Example \ref{sol-ext}. On the other hand, if $\dim_kW < (p-1)/2$ 
in Theorem \ref{ext}, then $H^1(G,\Sym^2(V)) = H^1(G,\wedge^2(V)) = 0$, see Corollary 
\ref{small2}.

In fact, we can show that both $\Ext^1_G(V,V)$ and $\Ext^1_G(V,V^*)$ are at most
$1$-dimensional in another situation, without any dimension condition, but instead with a condition on 
Sylow $p$-subgroups.

\begin{thm} \label{ext2}
Let $k$ be a field of characteristic $p$ and $G$ a finite group.  
Let $V$ be an absolutely irreducible faithful $kG$-module, and let $\GP$ denote the subgroup generated by the $p$-elements of $G$. Suppose that the image of $\GP$ in $\GL(W)$ for some 
irreducible $\GP$-submodule $W$ of $V$ has Sylow $p$-subgroups of order $p$, and that
$G$ has no composition factor of order $p$. Then $\dim_k\Ext^1_G(V,V) \leq 1$ and $\dim_k\Ext^1_G(V,V^*) \leq 1$. In particular,
$H^1(G,\Sym^2(V))$ and $H^1(G,\wedge^2(V))$ are both at most $1$-dimensional.
\end{thm}

Next we determine adequacy of linear groups of degree $p$:

\begin{thm} \label{thm:degree-p}   Let $k$ be a field of characteristic $p$ and
$G$ a finite group.  Let $V$ be an absolutely irreducible faithful $kG$-module
with $\dim V = p$.   Then  precisely one of the following holds:
\begin{enumerate}[\rm(i)]
\item  $(G,V)$ is adequate;
\item  $G$ contains a normal abelian subgroup of index $p$;
\item   $p=3$ and the image of $G$ in $\PGL(V)$ is $\PSL_2(9)$. 
\end{enumerate}
\end{thm}

Extending the results of \cite[\S3]{GHT}, we prove in Corollary \ref{sl2q-adequate} that, 
aside from some exceptions with $p = 2,3$ and with $(q,\dim(V)) = (p, (p \pm 1)/2)$, 
all nontrivial irreducible representations of  $\SL_2(q)$ over $\overline\bbF_q$ are adequate. This and other results on weak adequacy and on $\Ext^1$, and the dearth of examples where weak adequacy fails suggest that quite a lot of irreducible representations are indeed weakly adequate. 
(Currently, all but one counterexample to weak adequacy are induced modules, and the only primitive counterexample is given in \cite{G2}.)

\smallskip
Finally, we classify all low dimensional self-dual indecomposable and non-irreducible
$kG$-modules $V$ with $k$ algebraically closed of characteristic $p$
and $G$ a finite 
subgroup of $\GL(V)$ with $\bfO_p(G) = 1$.

First we recall one of the main results of \cite{Gcr} which settled a conjecture
of Serre.

\begin{thm}  \label{thm:smaller}  Let $k$ be a field of positive characteristic $p$.
Let $G$ be a subgroup of $\GL_n(k)=\GL(V)$ with no nontrivial normal unipotent
subgroup and $p \ge n + 2$.
Then $V$  is completely reducible.
\end{thm}

Serre asked for an analogous result for the other classical groups.
The example  $\AAA_p < \SO_p(k)$ shows that one cannot do too much better.
We also see that  there are reducible indecomposable self-dual $\SL_2(p)$-modules
of dimensions $p$ and $p \pm 1$ (contained in $\Sp$ for the dimension $p-1$
and $\SO$ in the other cases).    Building on the methods used in computing
$\ext^1$, we can essentially classify the self-dual reducible indecomposable modules
of dimension less than $2p-2$.

\begin{thm} \label{thm: indecomposable}  Let $k$ be an algebraically closed field
of characteristic $p$.   Let $V$ be a vector space over $k$ with $\dim V \leq 2p-3$.
Suppose that  $G$ is a finite subgroup of $\GL(V)$ such that $\bfO_p(G) = 1$, 
and the $kG$-module $V$ is indecomposable and self-dual
but not irreducible. Then $p > 3$, $\GP$ is quasisimple, $V_\GP$ is uniserial,  and one of the following statements holds for some $U \cong U^* \in \IBr_p(\GP)$.

\begin{enumerate}[\rm(i)]

\item $V_\GP = (k|U|k)$, and $(\GP,p,\dim U)$ is
$(\SL_2(q),q-1,p+1)$, $(\AAA_p,p,p-2)$, $(\SL_n(q),(q^n-1)/(q-1),p-2)$,
$(M_{11},11,9)$, $(M_{23},23,21)$, or $(\PSL_2(p),p,p-2)$.

\item $V_\GP = (U|U)$. Furthermore,
$(\GP,p,\dim U) = (\SL_2(q),q+1,p-2)$, $(2\AAA_7,7,4)$,
$(\PSL_2(p), p \equiv \epsilon (\mod 4), (p+\epsilon)/2)$ or
$(\SL_2(p), p \equiv \epsilon (\mod 4), (p-\epsilon)/2)$ with $\epsilon = \pm 1$.
\end{enumerate}
Moreover, $V$ supports a non-degenerate
$G$-invariant bilinear form that is either symmetric or alternating. Furthermore, all such forms
have the same type, which is symmetric in all cases, except when
$(\GP,p,\dim U) = ({\mathrm {(P)}}\SL_2(p),p,(p-1)/2)$ in which case it is alternating.
Conversely, all the listed cases give rise to reducible self-dual indecomposable modules of
dimension $< 2p-2$.
\end{thm}

In particular, this gives a classification of all finite non-$\cG$-cr subgroups for
$\cG = \Sp(V)$ or $\SO(V)$ with $\dim V < 2p-2$, see Proposition \ref{serre} (recall that the notion
of $\cG$-cr subgroups was introduced by Serre in \cite{Serre2}).  It also yields
the following variant of the main result of \cite{Gcr}:

\begin{cor}\label{clas}
Let $k$ be an algebraically closed field of characteristic $p$ and $V$ a vector space over $k$ with $d := \dim V \leq p-1$.
Suppose $G$ is a finite subgroup of $\GL(V)$ such that $\bfO_p(G) = 1$ and 
the $kG$-module $V$ is self-dual. Then either
the $kG$-module $V$ is completely reducible, or $d = p-1$, 
$\GP = {\mathrm {(P)}}\SL_2(p)$, and any $G$-invariant non-degenerate bilinear form on 
$V$ must be alternating. 
\end{cor}

In Theorem \ref{thm: indecomposable} and Corollary \ref{clas}, the notation
${\mathrm {(P)}}\SL_2(p)$ means $\SL_2(p)$ if $p \equiv 1 (\mod 4)$ and 
$\PSL_2(p)$ if $p \equiv 3 (\mod 4)$.

\medskip
This paper is organized as follows. In \S\ref{sec: 2p},  we
describe the structure of quasisimple linear groups of degree at most $2p$. 
We collect various facts concerning extensions and self-extensions of simple modules in 
\S\ref{sec: ext1} and 
prove Theorem \ref{thm:abelian} in \S\ref{sec: ind}. Theorems \ref{ext} and \ref{ext2} are proved 
in \S\ref{sec: ext2}. Adequacy of linear groups of 
degree $p$ is discussed in \S\ref{sec: dim p}; in particular, we prove
Theorem \ref{thm:degree-p}. In the next \S\ref{sec: pim}, we describe the PIMs for 
various simple modules of simple groups. These data are used in \S\ref{sec:gcr} 
to classify reducible self-dual 
indecomposable modules of dimension at most $2p-3$, cf. Theorem 
\ref{thm: indecomposable}, and to classify the finite non-$\cG$-cr subgroups
of symplectic and orthogonal groups in dimensions at most $2p-3$, cf.
Proposition \ref{serre} and Corollary \ref{clas}. \S\ref{sec: sl2q} is devoted to proving weak adequacy 
of $\SL_2(q)$-representations, cf. Proposition \ref{sl2q}. In \S\ref{SLn adequacy}, we show that almost
always the natural module for $\SL_n(q)$ is adequate.  In \S\ref{sec: asymp},
we prove Theorem \ref{thm:asymp} concerning adequacy of (possibly disconnected) 
reductive algebraic groups and asymptotic adequacy. 
  
\medskip
{\bf Notation.} If $V$ is a $kG$-module and $X \leq G$ is a subgroup, then
$V_X$ denote the restriction of $V$ to $X$. The containments $X \subset Y$ (for
sets) and $X < Y$ (for groups) are strict. Fix a prime $p$ and an
algebraically closed field $k$ of characteristic $p$. Then for any finite group $G$,
$\IBr_p(G)$ is the set of isomorphism classes of irreducible
$kG$-representations (or their Brauer characters, depending on the context), $\dl(G)$ denotes the smallest degree of nontrivial $\varphi \in \IBr_p(G)$,
$\cP(\varphi)$ is the principal indecomposable module (PIM) 
corresponding to $\varphi$, and
$B_0(G)$ denotes the principal $p$-block of $G$. Sometimes we use 
$\bbone$ to denote the principal representation.
$\bfO_p(G)$ is the largest normal $p$-subgroup of $G$, $\bfO^p(G)$ is the
smallest normal subgroup $N$ of $G$ subject to $G/N$ being a $p$-group, and similarly
for $\bfO_{p'}(G)$ and $\bfO^{p'}(G) = \GP$. Furthermore, the {\it Fitting subgroup} $F(G)$ is
the largest nilpotent normal subgroup of $G$, and $E(G)$ is the product of all
subnormal quasisimple subgroups of $G$, so that $F^*(G) = F(G)E(G)$ is the {\it
generalized Fitting subgroup} of $G$.
Given a finite-dimensional $kG$-representation
$\Phi:G \to \GL(V)$, we denote by $\cM$ the $k$-span
$$\langle \Phi(g) \mid \Phi(g) \mbox{ semisimple}\rangle_k.$$
If $M$ is a finite length module over a ring $R$,  then define $\soc_i(M)$
by $\soc_0(M) = 0$ and $\soc_j(M)/\soc_{j-1}(M) = \soc (M/\soc_{j-1}(M))$.
If $M = \soc_j(M)$
with $j$ minimal, we say that $j$ is the {\it socle length} of $M$. If $V$ is a vector space 
endowed with a non-degenerate quadratic form, then $\OO(V)$ denotes the full isometry group
of the form. For a linear algebraic group $\cG$, $\cG^0$ denotes the connected component
containing the identity.

\section{Linear groups of low degree}\label{sec: 2p}

First we recall the description of absolutely irreducible non-solvable linear groups
of degree less than $p = \Char(k)$, relying on the main result of Blau and Zhang \cite{BZ}:

\begin{thm}\label{bz} {\rm \cite[Theorem 2.1]{GHT}.}
Let $W$ be a faithful, absolutely irreducible $kH$-module for a finite group $H$ with
$\bfO ^{p'}(H) = H$. Suppose that $1 < \dim W < p$. Then one of the following cases holds,
where $P \in \Syl_p(H)$.

\smallskip
{\rm (a)} $p$ is a Fermat prime, $|P| = p$, $H = \bfO _{p'}(H)P$ is solvable, $\dim W = p-1$,
and $\bfO_{p'}(H)$ is absolutely irreducible on $W$.

\smallskip
{\rm (b)} $|P| = p$, $\dim W = p-1$, and one of the following conditions holds:

\hskip1pc
{\rm (b1)} $(H,p) = (\SU_n(q), (q^n+1)/(q+1))$, $(\Sp_{2n}(q),(q^n+1)/2)$,
$(2\AAA_7,5)$, $(3J_3,19)$, or $(2Ru,29)$.

\hskip1pc
{\rm (b2)} $p = 7$ and $H = 6_1\cdot \PSL_3(4)$, $6_1\cdot \PSU_4(3)$, $2J_2$,
$3\AAA_7$, or $6\AAA_7$.

\hskip1pc
{\rm (b3)} $p = 11$ and $H = M_{11}$, $2M_{12}$, or $2M_{22}$.

\hskip1pc
{\rm (b4)} $p = 13$ and $H = 6\cdot Suz$ or $2\gtwo(4)$.

\smallskip
{\rm (c)} $|P| = p$, $\dim W = p-2$, and
$(H,p) = (\PSL_n(q), (q^n-1)/(q-1))$, $(\AAA_p,p)$, $(3\AAA_6,5)$,
$(3\AAA_7,5)$, $(M_{11},11)$, or $(M_{23},23)$.

\smallskip
{\rm (d)} $(H,p,\dim W) = (2\AAA_7,7,4)$, $(J_1,11,7)$.

\smallskip
{\rm (e)} Extraspecial case: $|P| = p = 2^n+1 \geq 5$, $\dim W = 2^n$,
$\bfO _{p'}(H) = R\bfZ(H)$, $R = [P,R]\bfZ (R) \in \Syl_2(\bfO_{p'}(H))$, $[P,R]$ is an
extraspecial $2$-group of order $2^{1+2n}$, $V_{[P,R]}$ is absolutely irreducible.
Furthermore,  $S := H/\bfO _{p'}(H)$ is simple non-abelian, and either
$S= \Sp_{2a}(2^b)'$ or $\Omega^-_{2a}(2^b)'$ with
$ab = n$, or $S = \PSL_2(17)$ and $p = 17$.

\smallskip
{\rm (f)} $Lie(p)$-case: $H/\bfZ (H)$ is a direct product of simple groups of Lie type in characteristic $p$.\\
Furthermore, in the cases {\rm (b)--(d)}, $H$ is quasisimple with $\bfZ(H)$ a $p'$-group.
\end{thm}

Now we prove the following result which extends Theorem \ref{bz} for quasisimple groups
and is of independent interest. Note that the {\it complex} analogue of this result is
given by \cite[Theorem 8.1]{TZ1}.

\begin{thm}\label{thm:2p}
Let $p$ be a prime, $H$ a finite quasisimple group of order divisible by $p$. Suppose
$W$ is a faithful, absolutely irreducible $kH$-module of dimension $d$, where
$p \leq d \leq 2p$. Then one of the following statements holds.

\begin{enumerate}[\rm(i)]
\item $H$ is a quasisimple group of Lie type in characteristic $p$.

\item $(H,\dim(W),p)$ is as listed in Tables I, IIa, IIb, III, where the fourth column lists the number of
isomorphism classes of $W$ for each choice of $(H,\dim(W),p)$.
\end{enumerate}
\end{thm}

\begin{proof}
Let $L$ be the universal covering group of $S := H/\bfZ(H)$ and let $\dl(L)$ denote
the smallest degree of nontrivial absolutely irreducible $kL$-representations. Then
\begin{equation}\label{2p}
  2p \geq \dim(W) \geq \dl(L).
\end{equation}
This inequality will allow us to rule out the majority of
the cases. We will assume that $S$ is {\it not} isomorphic to any finite simple group
of Lie type in characteristic $p$.

\smallskip
First, let $S$ be a sporadic group. Then $\dl(L)$ is listed in \cite{Jan}. Furthermore,
$p \leq 71$ and so $\dim W \leq 142$. Now the result follows from inspecting
\cite{HM2} and \cite{JLPW} (and also \cite{ModAt} for the three Conway groups), and
is listed in Table III.

\smallskip
Assume now that $S = \AAA_n$. The cases $5 \leq n \leq 13$ can
be checked by inspecting \cite{JLPW}, and the result is listed in Table I.
If $14 \leq n \leq 16$, then $p \leq 13$, $\dim W \leq 26$, and so the statement follows
by inspecting \cite{HM2}. So we may assume $n \geq 17$. In this case,
$$\dim W \leq 2p \leq 2n < (n^2-5n+2)/2.$$
Hence, by \cite[Lemma 6.1]{GT2}, $W$ is the heart of the natural permutation module
of $G = \AAA_n$, yielding the first row of Table I.

\smallskip
Next suppose that $S$ is an exceptional group of Lie type defined over $\bbF_q$. The
cases $S \in \{\tw2 \btwo(8), \gtwo(3), \gtwo(4), \tw3 \df(2), \tw2 \ff(2)', \ff(2)\}$ can be
checked using \cite{JLPW} and lead to the last six rows of Table IIb. For all other groups,
$\dl(L)$ is bounded below by the Landazuri-Seitz-Zalesskii bounds (see
\cite[Table II]{Tiep} for latest improvements) and one can check that (\ref{2p}) cannot
hold. For instance, if $S = \ff(q)$ with $q \geq 3$, then $p \leq q^4+1$ whereas
$\dl(L) \geq q^8+q^4-2$.

\smallskip
From now on we may assume that $S$ is a finite classical group defined over
$\bbF_q$ and $p \nmid q$. Suppose first that $S=\PSL_2(q)$. Using \cite{JLPW} we may assume  that $q \geq 11$. If $q$ is even, then $L = \SL_2(q)$ and each
$\varphi \in \IBr_p(L)$ has degree $q$ or $q \pm 1$, whereas $p\mid(q \pm 1)$. If
in addition $p \neq q \pm 1$ then $2p \leq 2(q+1)/3 < q-1 \leq \dim(W)$, violating
(\ref{2p}). So $p = q \pm 1$, and inspecting \cite{Burk} we arrive at the first multi-row
of Table IIa. Assume that $q$ is odd. Then again $L = \SL_2(q)$, and we also get
additional possibilities $\varphi(1) = (q\pm 1)/2$ for $\varphi \in \IBr_p(L)$. Note that
$p \neq (q \pm 1)$, $(q \pm 1)/3$ (as $q \geq 11$ is odd) and (\ref{2p}) implies
$p > (q+1)/5$ as $\dl(L) = (q-1)/2$. It follows that $p = (q \pm 1)/4$ or
$p = (q \pm 1)/2$.
A detailed analysis of $\IBr_p(L)$ leads to the 2nd and 3rd multi-rows of Table IIa.

\smallskip
Suppose now that $S = \PSL_n(q)$ with $n \geq 3$. Note that $\SL_4(2) \cong \AAA_8$.
If $(n,q) = (6,3)$, then $p \leq 13$ whereas $\dl(L) \geq 362$ by \cite[Table III]{GT1}.
If $(n,q) = (6,2)$, then $p \leq 31$, so \cite[Table III]{GT1} implies
$(\dim W,p) = (62,31)$, as recorded in Table IIa (the 9th row). The cases
$(n,q) = (3, q \leq 7)$, $(4,3)$ can be checked using \cite{JLPW}. So we may assume that
$(n,q) \neq (3, q \leq 7)$, $(4,2)$, $(4,3)$, $(6,2)$, $(6,3)$. In this case,
$$\dim W \leq 2p \leq \frac{2(q^n-1)}{q-1} < \left\{ \begin{array}{ll}
    (q^2-1)(q-1)/\gcd(3,q-1), & n = 3,\\
    (q^3-1)(q-1)/\gcd(2,q-1), & n = 4,\\
    (q^{n-1}-1)\left(\dfrac{q^{n-2}-q}{q-1}-1\right), & n \geq 5. \end{array} \right.$$
Applying \cite[Theorem 1.1]{GT1}, we see that $W$ is one of the {\it Weil modules} of
$L = \SL_n(q)$, of dimension $(q^n-1)/(q-1)-a$ with $a = 0,1,2$. Note that
$$\frac{q^n-1}{q-1}-2 > \max\left\{2(q+1), 2\cdot \frac{q^{n-2}-1}{q-1},
    \frac{q^{n-1}-1}{q-1}, \frac{2}{3} \cdot \frac{q^n-1}{q-1}\right\}.$$
If $q \geq 3$, then $(q^n-1)/(q-1) -2 > 2(q^{n-1}-1)/(q-1)$.
Recall that $p\mid\prod^{n}_{i=1}(q^i-1)$ and $p \leq \dim W \leq 2p$. So we arrive at one
of the following possibilities:

$\bullet$ $q=2$, $p = 2^{n-1}-1$, whence $W$ is the unique Weil module of dimension
$2p$ by \cite[Theorem 1.1]{GT1}, leading to the 9th row of Table IIa;

$\bullet$ $p = (q^n-1)/(q-1)$, whence $n$ is an odd prime, $S = L$, and $W$ is one
of $q-2$ Weil modules of dimension $p$ by \cite[Theorem 1.1]{GT1}, leading to
the 8th row of Table IIa;

$\bullet$ $2p = (q^n-1)/(q-1)$. Here, $q$ is odd and $n$ must be even, but then
$$\frac{q^n-1}{2(q-1)} = \frac{q^n-1}{q^2-1} \cdot \frac{q+1}{2}$$
cannot be a prime.

\smallskip
Next suppose that $S = \PSU_n(q)$ with $n \geq 3$. The cases
$(n,q) = (3, q \leq 5)$, $(4,q \leq 3)$, $(5,2)$, $(6,2)$ can be checked using \cite{JLPW}. So we may assume that none of these cases occurs. Now observe that
if $n=p=3\mid(q+1)$ then $2p < (q-1)(q^2+3q+2)/6$, and furthermore
$$2p \leq \frac{2(q^n-(-1)^n)}{q+1} < \left\{ \begin{array}{ll}
     (q-1)(q^2-q+1)/3, & n=3, p \neq 3\mid(q+1),\\
    (2q^3-q^2+2q-3)/3, & n = 3, 3 \nmid(q+1)\\
    \dfrac{(q^2+1)(q^2-q+1)}{\gcd(2,q-1)}-1, & n = 4,\\
    \dfrac{(q^n-(-1)^n)(q^{n-1}-q^2)}{(q+1)(q^2-1)}-1, & n \geq 5. \end{array} \right.$$
Applying \cite[Theorem 16]{HM1} and \cite[Theorem 2.7]{GMST},
we conclude that $W$ is one of the {\it Weil modules} of
$L = \SU_n(q)$, of dimension $(q^n-(-1)^n)/(q+1)-b$ with $b = 0,\pm 1$. Note that
$$\frac{q^n-(-1)^n}{q+1}-1 > \max\left\{2(q+1), \frac{2(q^{n-2}-(-1)^n)}{q+1},
    \frac{q^{n-1}+(-1)^n}{q+1}, \frac{2(q^n-(-1)^n)}{3(q+1)}\right\}.$$
If $q \geq 3$, then
$$\frac{q^n-(-1)^n}{q+1} -1 > 2 \cdot \frac{q^{n-1}-(-1)^{n-1}}{q+1}.$$
Recall that $p\mid\prod^{n}_{i=2}(q^i-(-1)^i)$ and $p \leq \dim W \leq 2p$. So we arrive at one
of the following possibilities:

$\bullet$ $q=2$, $p = (2^{n-1}-(-1)^{n-1})/3$; in particular, $n-1 \geq 5$ is a prime. Hence
$W$ is either the unique Weil module of dimension $2p$, or one of the two Weil modules of dimension $2p-1$, leading to the 11th and 12th rows of Table IIa;

$\bullet$ $p = (q^n-(-1)^n)/(q+1)$, whence $n$ is an odd prime, $S = L$, and $W$ is one
of $q$ Weil modules of dimension $p$, yielding the 10th row of Table IIa;

$\bullet$ $2p = (q^n-(-1)^n)/(q+1)$. Here, $q$ is odd and $n$ must be even, but then
$$\frac{q^n-(-1)^n}{2(q+1)} = \frac{q^{n/2}-(-1)^{n/2}}{q+1} \cdot
    \frac{q^{n/2}+(-1)^{n/2}}{2}$$
cannot be a prime.

\smallskip
Now let $S = \PSp_{2n}(q)$ with $n \geq 2$. Note that $\Sp_4(2)' \cong \AAA_6$
and $\PSp_4(3) \cong \SU_4(2)$. Also, the cases
$(n,q) = (2, 4)$, $(3,2)$ can be checked using \cite{JLPW}. So we will assume that
$(n,q) \neq (2, q \leq 4)$, $(3,2)$. In this case,
$$\dim W \leq 2p \leq \frac{2(q^n+1)}{\gcd(2,q-1)} <  \frac{(q^n-1)(q^{n}-q)}{2(q+1)}.$$
Using the Landazuri-Seitz-Zalesskii bound for $\Sp_{2n}(q)$ with $2\mid q$ and applying \cite[Theorem 2.1]{GMST} to $\Sp_{2n}(q)$ with $q$ odd, we now see that
$q$ must be odd and $W$ is one of the {\it Weil modules} of
$L = \Sp_{2n}(q)$, of dimension $(q^n \pm 1)/2$.  So we arrive at the 4th--7th rows
of Table IIa. Note in addition that if $2 < p = (q^n-1)/4$ then $q=5$ and $n$ is an odd prime, and if $p = (q^n+1)/4$ then $q = 3$ and $n$ is again an odd prime. Similarly,
if $p = (q^n-1)/2$ then $q = 3$ and $n$ is an odd prime, and if
$p = (q^n+1)/2$ then $n = 2^m$.

\smallskip
Now we may assume that $S = \Omega_{2n+1}(q)$ with $n \geq 3$, or
$P\Omega^\pm_{2n}(q)$ with $n \geq 4$. Again, the cases of $\Omega_7(3)$ and
$\Omega^\pm_8(2)$ can be checked directly using \cite{JLPW}. Aside from these
cases, one can verify that (\ref{2p}) cannot hold.
\end{proof}

\begin{figure}
\centerline
{{\sc Table} I. Quasisimple linear groups: Alternating groups}
\vspace{0.3cm}
\begin{center}
\begin{tabular}{|c||c|c|c|} \hline \skipa
$H$ & $\dim W$ & $p$ & Class number  \\ \skipa \hline \hline \skipa
$\AAA_n$ & $n-\left\{\begin{array}{ll}2, & p\mid n\\ 1, & p\nmid n \end{array}\right.$ &
    $\dfrac{n-1}{2} \leq p \leq n-1$ & $1$ \\ \hline
$\AAA_{5}$ & $3$ & $3$ & $2$ \\ \hline
$2\AAA_{5}$ & $6$ & $3$ & $1$\\ \hline
$\AAA_{6}$ & $5$, resp. $8$, $10$ & $5$ & $2$, resp. $1$, $1$\\ \hline
$2\AAA_{6}$ & $10$ & $5$ & $2$ \\ \hline
$3\AAA_{6}$ & $6$ & $5$ & $2$ \\ \hline
$6\AAA_{6}$ & $6$ & $5$ & $4$ \\ \hline
$\AAA_{7}$ & $4$ & $2$ & $2$ \\ \hline
$\AAA_{7}$ & $8$, resp. $10$ & $5$ & $1$, resp. $2$ \\ \hline
$\AAA_{7}$ & $10$, resp. $14$ & $7$ & $1$, resp. $2$ \\ \hline
$2\AAA_{7}$ & $4$, resp. $6$ & $3$ & $2$ \\ \hline
$2\AAA_{7}$ & $14$ & $7$ & $2$ \\ \hline
$3\AAA_{7}$ & $6$ & $5$ & $2$ \\ \hline
$3\AAA_{7}$ & $9$ & $7$ & $2$ \\ \hline
$6\AAA_{7}$ & $6$ & $5$ & $4$ \\ \hline
$\AAA_{8}$ & $14$ & $7$ & $1$\\ \hline
$2\AAA_{8}$ & $8$ & $5$, $7$ & $1$\\ \hline
$2\AAA_{9}$ & $8$ & $5$, $7$ & $2$\\ \hline
$2\AAA_{10}$ & $8$ & $5$ & $2$\\ \hline
$2\AAA_{11}$ & $16$ & $11$ & $1$\\ \hline
\end{tabular}
\end{center}
\end{figure}

\begin{figure}
\centerline
{{\sc Table} IIa. Quasisimple linear groups: Groups of Lie type. I}
\vspace{0.3cm}
\begin{center}
\begin{tabular}{|c||c|c|c|} \hline \skipa
$H$ & $\dim W$ & $p$ & Class number  \\ \skipa \hline \hline \skipa
$\begin{array}{c}\SL_2(q)\\2\mid q \end{array}$ &
  $\begin{array}{c}p \\ p \\p+1\end{array}$ &
  $\begin{array}{c}q+1 \\q-1 \\q-1 \end{array}$ &
  $\begin{array}{c} q/2-1 \\ q/2 \\ 1 \end{array}$
  \\ \hline \skipa
$\begin{array}{c}\PSL_2(q)\\2\nmid q \end{array}$ &
  $\begin{array}{c}p\\2p-2\\2p-1\\2p\\2p\end{array}$ &
  $\begin{array}{c}(q \pm 1)/2 \\ (q+1)/2\\(q+1)/4\\(q-1)/2\\(q+1)/2 \end{array}$ &
  $\begin{array}{c}2\\1\\2\\(q-3)/4\\ (q-5)/4 \end{array}$
  \\ \hline \skipa
$\begin{array}{c}\SL_2(q)\\2\nmid q \end{array}$ &
  $\begin{array}{c}p+1\\2p\\2p\\2p \end{array}$ &
  $\begin{array}{c}(q-1)/2\\(q \pm 1)/4\\(q-1)/2\\(q+1)/2 \end{array}$ &
  $\begin{array}{c}2\\2\\(q+1)/4\\ (q-1)/4\end{array}$
  \\ \hline \skipa
$\PSp_{2n}(q)$, $2\nmid q$, $n \geq 2$ & $p$ &
  $\begin{array}{ll}p = (q^n - 1)/2, & q= 3\\p =(q^n+1)/2, & n = 2^m \end{array}$ &
    $2$ \\ \hline \skipa
$\Sp_{2n}(3)$, $n$ odd prime & $p+1$ & $(3^n -1)/2$ & $2$ \\ \hline
\skipa $\PSp_{2n}(3)$, $n$ odd prime & $2p-1$ & $(3^n+1)/4$ & $2$\\ \skipa \hline
\skipa $\Sp_{2n}(q)$, $n$ odd prime & $2p$ &
  $\begin{array}{ll}p = (q^n -1)/4, & q=3 \\p = (q^n+1)/4, & q=5 \end{array}$ & $2$\\
      \skipa \hline
\skipa $\begin{array}{c}\SL_{n}(q)\\ n \mbox{ odd prime} \end{array}$ & $p$ &
  $\dfrac{q^n-1}{q-1}$ & $q-2$ \\ \skipa \hline \skipa
$\begin{array}{c}\SL_{n}(2)\\n-1 \geq 3\mbox{ prime}\end{array}$
  & $2p$ & $2^{n-1}-1$ & $1$\\  \hline
\skipa $\begin{array}{c}\SU_{n}(q)\\ n \mbox{ odd prime} \end{array}$ & $p$ &
  $\dfrac{q^n+1}{q+1}$ & $q$ \\ \skipa \hline
\skipa $\begin{array}{c}\PSU_{n}(2)\\n-1 \geq 5\mbox{ prime}\end{array}$
  & $2p$ & $(2^{n-1}+1)/3$ & $1$\\  \hline
\skipa $\begin{array}{c}\SU_{n}(2)\\n-1 \geq 5 \mbox{ prime}\end{array}$ & $2p-1$ & $(2^{n-1}+1)/3$ & $2$\\  \hline
\end{tabular}
\end{center}
\end{figure}

\begin{figure}
\centerline
{{\sc Table} IIb. Quasisimple linear groups: Groups of Lie type. II}
\vspace{0.3cm}
\begin{center}
\begin{tabular}{|c||c|c|c|} \hline \skipa
$H$ & $\dim W$ & $p$ & Class number  \\ \skipa \hline \hline \skipa
$\SL_3(3)$ & $16$, resp. $26$ & $13$ & $1$, resp. $3$ \\ \hline
\skipa $2 \cdot \PSL_3(4)$ & $6$ & $3$ & $1$\\ \hline
\skipa $2 \cdot \PSL_3(4)$ & $10$ & $5$, resp. $7$ & $2$, resp. $1$ \\ \hline
\skipa $4_1 \cdot \PSL_3(4)$ & $8$ & $5$, resp. $7$ & $2$, resp. $4$ \\ \hline
\skipa $4_2 \cdot \PSL_3(4)$ & $4$ & $3$ & $2$ \\ \hline
\skipa $6 \cdot \PSL_3(4)$ & $6$ & $5$ & $2$ \\ \hline
\skipa $\PSL_4(3)$ & $26$ & $13$ & $2$\\ \hline
\skipa $\SU_3(3)$ & $14$ & $7$ & $1$\\ \hline
 $\SU_4(2)$ & $\begin{array}{c}10\\6 \end{array}$ &
  $5$ & $\begin{array}{c}2 \\ 1 \end{array}$\\ \hline
\skipa $6_1 \cdot \PSU_4(3)$ & $6$ & $5$ & $2$\\ \hline
\skipa $\SU_5(2)$ & $10$ & $5$ & $1$\\ \hline
$\Sp_4(4)$ & $18$, resp. $34$ & $17$ & $1$, resp. $2$ \\ \hline
\skipa $\Sp_6(2)$ & $7$ & $5$, $7$ & $1$\\ \hline
\skipa $2 \cdot \Sp_6(2)$ & $8$ & $5$, $7$ & $1$\\ \hline
\skipa $2 \cdot \Omega^+_8(2)$ & $8$ & $5$, $7$ & $1$\\ \hline
\skipa $\Omega^-_8(2)$ & $34$ & $17$ & $1$\\ \hline
\skipa $\tw2 \btwo(8)$ & $14$ & $7$, $13$ & $2$\\ \hline
$2 \cdot \tw2 \btwo(8)$ & $\begin{array}{c}8\\16, 24 \end{array}$ &
  $\begin{array}{c}5 \\ 13 \end{array}$ & $1$\\ \hline
\skipa $\gtwo(3)$ & $14$ & $7$, $13$ & $1$\\ \hline
\skipa $2 \cdot \gtwo(4)$ & $12$ & $7$ & $1$\\ \hline
\skipa $\tw2 \ff(2)'$ & $26$ & $13$ & $2$\\ \hline
\skipa $\tw3 \df(2)$ & $26$ & $13$ & $1$\\ \hline
\end{tabular}
\end{center}
\end{figure}

\begin{figure}
\centerline
{{\sc Table} III. Quasisimple linear groups: Sporadic groups}
\vspace{0.3cm}
\begin{center}
\begin{tabular}{|c||c|c|c|} \hline \skipa
$H$ & $\dim W$ & $p$ & Class number  \\ \skipa \hline \hline \skipa
$M_{11}$ & $5$ & $3$ & $2$ \\ \hline
$2M_{12}$ & $2p$ & $3$, $5$ & $2$ \\ \hline
$2J_{2}$ & $6$ & $3$, resp. $5$ & $2$, resp. $1$ \\ \hline
$M_{11}$ & $10$ & $5$ & $3$ \\ \hline
$2M_{22}$ & $10$ & $5$, resp. $7$ & $2$, resp. $1$ \\ \hline
$M_{11}$ & $11$, $16$ & $11$ & $1$ \\ \hline
$M_{12}$ & $11$, resp. $16$ & $11$ & $2$, resp. $1$ \\ \hline
$2M_{12}$ & $12$ & $11$ & $1$ \\ \hline
$6Suz$ & $12$ & $7$, $11$, $13$ & $2$ \\ \hline
$J_{1}$ & $14$ & $11$ & $1$ \\ \hline
$J_{1}$ & $22$, $34$ & $19$ & $1$ \\ \hline
$J_{2}$ & $14$ & $7$ & $2$ \\ \hline
$2J_{2}$ & $14$ & $7$ & $1$ \\ \hline
$3J_{3}$ & $18$ & $17$ & $4$ \\ \hline
$M_{22}$ & $20$ & $11$ & $1$ \\ \hline
$3M_{22}$ & $21$ & $11$ & $2$ \\ \hline
$M_{23}$ & $22$ & $11$ & $1$ \\ \hline
$HS$ & $22$ & $11$ & $1$ \\ \hline
$McL$ & $22$ & $11$ & $1$ \\ \hline
$M_{24}$ & $23$, resp. $45$ & $23$ & $1$, resp. $2$ \\ \hline
$M_{23}$ & $45$ & $23$ & $2$ \\ \hline
$Co_3$ & $23$ & $23$ & $1$ \\ \hline
$Co_2$ & $23$ & $23$ & $1$ \\ \hline
$2Co_1$ & $24$ & $13$, $23$ & $1$ \\ \hline
\end{tabular}
\end{center}
\end{figure}

\section{Extensions and self-extensions}\label{sec: ext1}
First we recall a convenient criterion concerning self-extensions in blocks of cyclic defect:

 \begin{lem}  \label{cyclic} {\rm \cite[Lemma 7.1]{GHT}.}
 Suppose that $G$ is a finite group and that $V$ is an irreducible
 $\fpb G$-representation  that belongs to a block of cyclic defect.
 Then $\Ext^1_G (V,V) \ne 0$ if and only if $V$ admits at least two
 non-isomorphic lifts to characteristic zero.
 In this case, $\dim \Ext^1_G(V,V)=1$.
 \end{lem}

The next observation is useful in various situations:

\begin{lem}\label{zero2}
Let $H \leq G$ be a subgroup of index coprime to $p = \Char(k)$.
Suppose that $V$ is a $kG$-module and $V_{H} = V_1 \oplus V_2$ is a direct sum of
two nonzero $H$-submodules, at least one of which is also stabilized by $G$. Then
the $G$-module $V$ is decomposable.
\end{lem}

\begin{proof}
Suppose for instance that $V_1$ is stabilized by $G$, and consider the natural
projection $\pi~:~V \to V_1$ along $V_2$. Write $G = \sqcup^m_{i=1}g_i H$ where
$m := [G:H]$ is coprime to $p$, and let
$$\tilde\pi = \frac{1}{m}\sum^m_{i=1}g_i\pi g_i^{-1}.$$
It is straightforward to check that $\tilde\pi$ is $G$-equivariant, $\tilde\pi^2 = \tilde\pi$,
and $\image(\tilde\pi) = V_1$. Hence the $G$-module $V$ decomposes as
$V_1 \oplus \Ker(\tilde\pi)$.
\end{proof}

From now on we again assume that $k$ is algebraically closed of characteristic $p$. 
First we record the following consequence of the K\"unneth formula,
cf. \cite[3.5.6]{Ben}. 

 \begin{lem} \label{lem:kunneth}  Let $H$ be a finite group.  Assume that
 $H$ is a central product of subgroups $H_i, 1 \le i \le t$, and
that $\bfZ(H)$ is a $p'$-group.
 Let $X$ and $Y$ be irreducible $kH$-modules.  Write
 $X=X_1 \otimes \ldots \otimes X_t$ and $Y=Y_1 \otimes \ldots \otimes Y_t$
 where $X_i$ and $Y_i$ are irreducible $kH_i$-modules.
 \begin{enumerate}[\rm(i)]
 \item If $X_i$ and $Y_i$ are not isomorphic for two distinct $i$, then
 $\ext_H^1(X,Y)=0$.
 \item If $X_1$ and $Y_1$ are not isomorphic but $X_i \cong Y_i$ for $i >1$, then
 $\ext_H^1(X,Y) \cong \ext_{H_1}^1(X_1,Y_1)$.
 \item  If $X_i \cong Y_i$ for all $i$, then
 $\ext_H^1(X,Y) \cong \oplus_i \ext_{H_i}^1(X_i, Y_i)$.
 \end{enumerate}
 \end{lem}

\begin{lem}\label{infl}
Let $G$ be a finite group with a normal subgroup $N = \bfO^{p}(N)$ and $V$ be a $kG$-module.
Suppose that $N$ acts trivially on $V$. Then $H^1(G,V) \cong H^1(G/N,V)$.  
\end{lem}

\begin{proof}
Since $N$ acts trivially on $V$, we have that $H^1(N,V) = \Hom(N,V)$. Furthermore, 
$\Hom(N,V) = 0$ as $\bfO^{p}(N) = N$. Now the inflation-restriction sequence in cohomology 
implies that the sequence 
$$0 \to H^1(G/N,V) \to H^1(G,V) \to 0$$
is exact, whence the claim follows. (Note that if $N$ is a $p'$-group, then the Hochschild-Serre
spectral sequence degenerates, and so  
$H^i(G,V) \cong H^i(G/N,V^N)$ for all $i$. Similarly,
if $G/N$ is a $p'$-group, then $H^i(G, V) = H^i(N, V)^{G/N}$ for all $i$.)
\end{proof}

\begin{lemma}\label{semi1} {\rm \cite[Lemma 7.8]{GHT}.}
Let $V$ be a $kG$-module of finite length.

{\rm (i)} Suppose that $X$ is a composition factor of $V$ such that
$V$ has no indecomposable subquotient of length $2$ with $X$ as a composition
factor. Then $V \cong X\oplus M$ for some submodule $M \subset X$.

{\rm (ii)} Suppose that $\Ext^1_G(X,Y) = 0$ for
any two composition factors $X$, $Y$ of $V$. Then $V$ is semisimple.
\end{lemma}

\begin{lem}\label{mult1}{\rm \cite[Lemma 7.9]{GHT}.}
Let $V$ be a $kG$-module. Suppose that $U$ is a composition factor of $V$ of
multiplicity $1$, and that $U$ occurs both in $\soc(V)$ and $\hd(V)$. Then
$V \cong U \oplus M$ for some submodule $M \subset V$.
\end{lem}

\begin{lemma}\label{semi2}
Let $G$ be group with a normal subgroup
$N$ of index coprime to $p$. Let $k$ be an algebraically closed field
of characteristic $p$, and let $V$ be a $kG$-module of finite length.

{\rm (i)} $V$ is semisimple if and only if $V_N$ is semisimple. In particular,
if $V$ is reducible indecomposable, then $V_N$ cannot be semisimple.

{\rm (ii)}  Suppose $V$ is reducible indecomposable. Then the $N$-module
$V$ has no simple direct summand.
\end{lemma}

\begin{proof}
(i) The ``only if'' part is obvious. For the ``if'' part, suppose $U$ is a $G$-submodule
of $V$. Since $V_N$ is semisimple, $V_N = U \oplus W$ for some $N$-submodule
$W$. As $U$ is $G$-stable, by Lemma \ref{zero2} there is a $G$-submodule
$W'$ such that $V = U \oplus W'$.

\smallskip
(ii) Consider a decomposition
$V_N = \oplus^t_{i=1}U_i$ into indecomposable direct summands, and write
$V= V_1 \oplus  V_2$, where $V_2$ is the sum of those $U_i$'s
which are simple and $V_1$ is the sum of the non-simple $U_i$'s.
Assume that $V_2 \neq 0$.

Note that if $U$ is any reducible indecomposable $N$-module,
then $\soc(U) \subseteq \rad(U)$. Indeed, suppose a maximal submodule $M \subset U$
does not contain $\soc(U)$. Then $\soc(U) = (M \cap \soc(U)) \oplus W$ for some
$N$-submodule $W \neq 0$, and $U = M \oplus W$ is decomposable, a contradiction.
Applying this remark to the summands $U_i$ in $V_1$, we see that
$\soc(V_1) \subseteq \rad(V_1)$. But $V_2$ is semisimple, so
$$\soc(V_1) = \soc(V_1) \cap \rad(V_1) = \soc(V) \cap \rad(V)$$
is $G$-stable. By Lemma \ref{zero2}, there is a $G$-submodule $V'_2 \neq 0$ such
that $\soc(V) = \soc(V_1) \oplus V'_2$.  In this case,
$V_N = V_1 \oplus V'_2$. Since $V'_2$ is $G$-stable, again by Lemma
\ref{zero2} we have that $V = V'_1 \oplus V'_2$ for some $G$-submodule
$V'_1$. As $V$ is indecomposable and $V'_2 \neq 0$, we must have that
$V'_1 = 0$, whence $V_1 = 0$, $V_N = V_2$ is semisimple, contradicting (i).
\end{proof}

\begin{lemma}\label{block2}{\rm \cite[Lemma 7.11]{GHT}.}
Let $V$ be an indecomposable $kG$-module.

{\rm (i)} If the $\GP$-module $V_\GP$ admits a composition factor $L$ of  dimension $1$, then all composition factors of $V_\GP$ belong to $B_0(\GP)$.

{\rm (ii)} Suppose a normal $p'$-subgroup $N$ of  $G$ acts by scalars on a composition factor $L$ of the $G$-module $V$. Then $N$ acts by scalars on $V$. If in addition $V$ is faithful then $N \leq \bfZ(G)$.
\end{lemma}

\begin{cor}\label{block3}
Let $V$ be an indecomposable $kG$-module of dimension $\leq 2p-3$.
Suppose that $\dl(\GP) \geq p-3$. Then one of the following holds:

{\rm (i)} The $\GP$-module $V$ is irreducible.

{\rm (ii)} All composition factors of the $\GP$-module $V$ have dimension
$\leq p$.

{\rm (iii)}  All composition factors of the $\GP$-module $V$ belong to $B_0(\GP)$.
\end{cor}

\begin{proof}
Suppose that $\dim U > p$ for a composition factor $U$ of the $\GP$-module $V$
but $V|_\GP$ is reducible. Since $\dim(V)-\dim(U) \leq p-4 < \dl(\GP)$, the
$\GP$-module $V$ must have a composition factor $L$ of dimension $1$. Hence we
are done by Lemma \ref{block2}.
\end{proof}

Finally, self-dual indecomposable modules of $\SL_2(q)$ (where $q = p^n$) of low dimension
are described in the following statement:

\begin{prop}
\label{prop1} {\rm \cite[Proposition 8.2]{GHT}.}
Suppose that $V$ is a reducible, self-dual, indecomposable representation
of $\SL_2 (\bbF_q)$ over $\barFp$, where $q=p^n$.
If $\dim V < 2p-2$, then $q=p$ and either of the following holds:

\begin{enumerate}[\rm(i)]
\item $\dim V = p$ and $V\cong \PIM (\bbone)$.
\item $\dim V = p+1$ and $V$ is the unique nonsplit self-extension
of $L\left( \frac{p-1}{2}\right)$.
\item $\dim V = p-1$ and $V$ is the unique nonsplit self-extension
of $L\left( \frac{p-3}{2}\right)$.
\end{enumerate}
Conversely, all the listed cases give rise to examples.
\end{prop}

 \section{Finite groups with indecomposable modules of dimension $\leq 2p-2$}\label{sec: ind}

Throughout this section, we assume that $k$ is an algebraically closed field of
characteristic $p > 3$.  First we recall several intermediate results proved in \cite{GHT}:

\begin{lem}\label{eg} {\rm \cite[Lemma 9.1]{GHT}.}
Let $G$ be a finite group, $p > 3$, and $V$ be a faithful $kG$-module of dimension
$< 2p$. Suppose that $\bfO_p(G) = 1$ and $\bfO_{p'}(G) \leq \bfZ(G)$. Then
$F(G) = \bfO_{p'}(G) = \bfZ(G)$, $F^*(G) = E(G)\bfZ(G)$,
and $\GP = E(G)$ is either trivial or a central
product of quasisimple groups of order divisible by $p$. In particular, $G$ has
no composition factor isomorphic to $C_p$ and so $H^1(G,k) = 0$.
\end{lem}

\begin{lem} \label{lem:opprime} {\rm \cite[Lemma 9.3]{GHT}.}
Let $V$ be a faithful indecomposable $kG$-module with two composition factors
$V_1$, $V_2$.  Assume that $\bfO_p(G) =1$ and $\dim V \leq 2p-2$.
If $J:=\bfO_{p'}(\GP) \not\leq \bfZ(\GP)$, then the following hold:
\begin{enumerate}[\rm(i)]
\item $p= 2^a + 1$ is a Fermat prime,
\item $\dim V_1 = \dim V_2 = p-1$,
\item  $J/\bfZ(J)$ is elementary abelian of order $2^{2a}$,
\item  $H^1(\GP,k) \ne 0$.
\end{enumerate}
\end{lem}

\begin{lem}\label{ext-defi} {\rm \cite[Lemma 9.5]{GHT}.}
Let $H$ be a quasisimple finite group of Lie type in char $p > 3$. Assume that
$V_1,V_2 \in \IBr_p(H)$ satisfy
$\dim V_1+\dim V_2 <2p$.

{\rm (i)} If $H \not\cong \SL_2(q),~\PSL_2(q)$, then $\Ext^1_H(V_1,V_2) = 0$.
In particular, there is no reducible indecomposable $kG$-module $V$ with
$\GP \cong H$ and $\dim V < 2p$.

{\rm (ii)} Suppose $H \cong \SL_2(q)$ or $\PSL_2(q)$, $\Ext^1_H(V_1,V_2) \neq 0$,
and $\dim V_1 = \dim V_2$. Then $q = p$ and $V_1 = L((p-3)/2)$ or
$L((p-1)/2)$.
\end{lem}

\begin{prop}\label{indec-str} {\rm \cite[Proposition 9.7]{GHT}.}
Let $p > 3$ and let $G$ be a finite group with a faithful, reducible, indecomposable $kG$-module $V$ of dimension $\leq 2p-3$. Suppose in addition that $\bfO_p(G) = 1$. Then
$\GP = E(\GP)$, $G$ has no composition factor isomorphic to $C_p$, and one of the following holds.

{\rm (i)} $\GP$ is quasisimple.

{\rm (ii)} $\GP$ is a central product of two quasisimple groups and
$\dim V = 2p-3$.  Furthermore, $V$ has one composition factor of dimension $1$,
and either one of dimension $2p-4$ or two of dimension $p-2$. In either case,
$V \not\cong V^*$.
\end{prop}

\begin{cor}\label{indec-gp}
Let $k$ be a field of characteristic $p$ and let $V$ be a faithful reducible
indecomposable $kG$-module of a finite group $G$ with $\bfO_p(G) = 1$.
If $\dim V \leq 2p-3$,  then $V_\GP$ is indecomposable.
\end{cor}

\begin{proof}
Assume the contrary. Then we can pick an indecomposable direct summand $U$ of dimension
$\leq p-2$ of $V_\GP$ and let $H \leq \GL(U)$ be the image of $\GP$ acting on $U$.
By Proposition \ref{indec-str}, $G$ has no composition factors isomorphic to $C_p$. Hence
$\bfO_p(H) = 1$.  Since the $kH$-module $U$ is faithful and indecomposable, $U$ is
simple by  \cite[Theorem A]{Gcr}. But this contradicts Lemma \ref{semi2}(ii).
\end{proof}


Recall that a {\it component} of a finite group is any subnormal quasisimple subgroup. We first note that:

\begin{lem}  \label{lem:p'-component}  Let $G$ be an irreducible
subgroup of $\GL(V) \cong \GL_d(k)$ with $k$ algebraically closed of
characteristic $p$. Assume that $G = \bfO^{p'}(G)$ and that $G$
has a component of order coprime to $p$.  Then $d \ge 2p$.
\end{lem}

\begin{proof} 
Assume the contrary: $d < 2p$. 
Write $E(G) = E_1 * E_2$, where $E_1$ is a central product of all components of $G$ of order 
coprime to $p$ and $E_2$ is the product of the remaining components; in particular,
$1 \neq E_1 \lhd G$. Since $G$ is generated by $p$-elements,  there is a $p$-element $x$ not centralizing $E_1$. Let $W$ be an irreducible constituent of $V_{E_1}$. Since $d <2p$ and 
$\dim W \geq 2$, $xW \cong W$.

Now write $E_1 = Q_1 * \ldots * Q_n$ as a central product of $n$ components and 
$W \cong W_1 \otimes \ldots \otimes W_n$, where $W_i$ is an irreducible $kQ_i$-module. 
Note that $x$ acts on the set $\{Q_1, \ldots ,Q_n\}$. If this action is nontrivial, then 
$\dim W \geq 2p$. (Indeed, we may assume that $x$ permutes $Q_1, \ldots, Q_m$ cyclically for some
$p \leq m \leq n$, and, replacing $W$ by another $E_1$-summand of $V$ if necessary, that $Q_1$ acts nontrivially on $W$, i.e. $\dim W_1 \geq 2$. Since $xW \cong W$, this implies that $\dim W_i \geq 2$
for $1 \leq i \leq m$, whence $\dim W \geq 2^m \geq 2^p \geq 2p$.) Thus we may assume that 
$x$ normalizes each $Q_i$, but does not centralize $Q_1$. It follows that $Q_1$ is a quasisimple
$p'$-group with a nontrivial outer automorphism of $p$-power order; in particular,
$p > 2$. If $p=3$, then 
$Q_1 \cong \tw 2 {\mathrm {B}}_2(2^{2a+1})$ and $\dim W_1\ge 14$. 
So $p > 3$ and, using the description of 
outer automorphisms of finite simple groups \cite[Theorem 2.5.12]{GLS3}, we see that $Q_1$ is a 
quasisimple group of Lie type over a field of size $q^p$ for some prime power $q$.  Now applying \cite{LaS}, we see that $\dim W_1 \geq 2p$.
\end{proof}

The proof of Lemma \ref{lem:p'-component} certainly depends on the 
classification of finite simple groups. We note the following result which 
does not require the classification:

\begin{lem}\label{p-solv} 
Let $k = \bar{k}$ be of characteristic $p$ and let $G$ be a finite irreducible 
$p$-solvable subgroup of $\GL(V) \cong \GL_n(k)$ of order divisible by $p$.
Then $n \geq p-1$.
\end{lem}

\begin{proof}
We may assume that $p > 2$. By a result of Isaacs \cite[Theorem 10.6]{N}, 
$V$ has a $p$-rational lift to characteristic $0$. But then by \cite{Dieder},
the Jordan blocks of any element $g \in G$ of order $p$ acting on $V$ have
sizes $1$, $p-1$, or $p$. So if $n < p-1$, then $g$ acts trivially on $V$,
a contradiction.
\end{proof}

In what follows, we will slightly abuse the language by also 
considering $C_p$ as a Frobenius group with kernel of order $p$.

\begin{lem}  \label{lem:perm}   Let $p > 2$ be a prime and let $G$ be a transitive subgroup of
$\SSS_n$ with $n < 2p$.  Assume that $G$ has a composition factor
of order $p$.  Then one of the following holds.

\begin{enumerate}[\rm(i)]
\item $n = p$ and $G$ is a Frobenius group of order $pe$ for some $e\mid (p-1)$ with kernel of
order $p$.

\item
$p = 2^a  -1$ is a Mersenne prime and $n=2^a = p + 1$.
Moreover,  $\soc(G)$ is a  regular elementary abelian subgroup of order $n$,
$G = \soc(G) \rtimes G_1$,
and  $G_1$ is a Frobenius group of order $pe$ for some $e\mid a$, with kernel of order $p$.  
If $H^1(G,k) \ne 0$, then $|G| = np$.
\end{enumerate}
\end{lem}

\begin{proof}  
Note that $G$ is primitive and contains a $p$-cycle. Hence we can apply \cite{Zie} and see that
either (i) holds, or $n = 2^a=p+1$ and $G=\soc(G) \rtimes G_1$, with $\soc(G) \cong C_2^a$
being regular, and $G_1 \leq \GL_a(2)$ has $C_p$ as a composition factor. Applying \cite{Kan} to
$G_1$, we arrive at (ii). 
\end{proof}

\begin{lem}\label{p-cycle}
Let $S := \Sp_{2a}(2)$ with $p = 2^a \pm 1$
and let $V = \bbF_2^{2a}$ denote the natural module for $S$.
Let $X \leq S$ be a group with $C_p$ as a composition factor and $p  > 3$. Then 
there is a normal elementary abelian $2$-subgroup $E < X$ such that $X/E$ is a 
Frobenius group of order $pe$ with kernel of order $p$, where $e\mid 2a$. 
Furthermore, if $E \neq 1$ then $p = 2^a-1$. Moreover, 
$X$ acts reducibly on $V$ precisely when $p = 2^a-1$ and 
either $E \neq 1$ or $|X|$ is odd, in which case 
$X$ stabilizes a maximal totally isotropic subspace of $V$.
\end{lem}

\begin{proof}
(a) It is easy to see that $Y := \bfO^{p'}(X)$ can be reducible on
$V$ only when $p = 2^a-1$. 
Let $P \in \Syl_p(X)$ and consider the action of $X$ on the natural module 
$V = \bbF_2^{2a}$ for $S$. If $P \lhd X$, then $X$ is contained in 
$\bfN_S(P)$, a Frobenius group of order $2ap$, in which case we set $E = 1$. 
It follows that, if $p = 2^a-1$ in addition, then $X$ is reducible on 
$V$ precisely when $|X|$ is odd. 
So we will assume that $P \ntrianglelefteq X$. It follows that 
$P \ntrianglelefteq Y := \bfO^{p'}(X)$.

Suppose that the $Y$-module $V$ is reducible, and so 
$p=2^a-1$ and $a \geq 3$ is odd. Then 
$Y$ stabilizes a proper subspace $U \neq 0$ of $V$ of dimension $\leq a$. 
Choosing $U$ minimal, we see that $U$ is irreducible over $Y$. If 
$U \cap U^\perp = 0$, then $Y$ is contained in the $p'$-subgroup $\Sp(U) \times \Sp(U^\perp)$, 
a contradiction. Hence $U \subseteq U^\perp$. Now if $\dim U < a$ then $Y$ is contained
in the $p'$-subgroup $\Stab_S(U)$, again a contradiction. Thus $\dim U = a$ and 
$U$ is a maximal totally isotropic subspace.
Setting $Q := \bfO_2(R)$ for $R := \Stab_S(U)$ and $F := Q \cap Y$, we see that 
$F$ is an elementary abelian $2$-subgroup and $Y/F$ is a subgroup of $R/Q \cong \GL_a(2)$ with $C_p$ as a composition factor. By the main result of 
\cite{Kan}, $Y/F$ is a Frobenius group of
order $pb$ for some $b\mid a$. In particular, $|Y/F|$ is odd and so $F = \bfO_2(Y) \lhd X$. 
Note that $F \neq 1$ as otherwise $P \lhd Y$, a contradiction. Also, $Y/F$ acts irreducibly on
$V/U$. Since $Y/F$ acts on $\bfC_V(F) \supseteq U$ and $F \neq 1$, it follows that 
$\bfC_V(F) = U$ and so $U$ is fixed by $X$. Thus $X \leq R$ and we are done by setting 
$E := Q \cap X \geq F$.

\smallskip
(b) From now on we may assume that $V_Y$ is irreducible. Hence $V$ is absolutely irreducible
over $k_0 := \End_{\bbF_2 Y}(V)$. We will consider $V$ as a $b$-dimensional vector space $V'$
over $k_0$ (for some $b\mid 2a$). Thus $W := V' \otimes_{k_0}k$ is an irreducible $kY$-module
for $k := \bar{k}_0$. Observe that $b \leq 2a \leq p-1$. Also, $\bfZ(Y)$ is cyclic by 
Schur's lemma.

Note that if $N \lhd Y$ then the $N$-module $W$ is homogeneous. Indeed, $V_N$ is 
the direct sum of $t \leq b < p$ homogeneous $N$-components $V_i$. Hence any $p$-element
$1 \neq g \in Y$ stabilizes each $V_i$, whence $V_i$ is fixed by $Y = \bfO^{p'}(Y)$ and $t = 1$.

\smallskip
(c) Now we show $E(Y) = 1$. Suppose $E(Y) \neq 1$ and write $E(Y) = L_1 * \ldots * L_n$, 
a central product of $n$ quasisimple groups. Since $|S|_p = p$ and $C_p$ is a composition
factor of $Y$, $E(Y)$ is a $p'$-group. By (b), 
$W_{E(Y)} \cong e(W_1 \otimes \ldots \otimes W_n)$, where $W_i$ is an irreducible $kL_i$-module
of dimension $\geq 2$. Hence $b \geq 2^n$ and so $n < p$. It follows that every $p$-element
$1 \neq g \in Y$ normalizes each $L_i$, and so does $Y$. On the other hand, if $Y$ centralizes
$L_i$, then $L_i \leq \bfZ(Y)$, a contradiction. So some $p$-element $1 \neq g \in Y$ normalizes
but does not centralize $L_1$. As in the proof of Lemma \ref{lem:p'-component}, we 
see that $L_1$ is a quasisimple group of Lie type 
defined over $\bbF_q$ with $q = r^{cp}$ for some 
prime $r$ and some integer $c$, and conclude that $\dim W_1 > p$ when $r \neq 2$. If $r = 2$, then by \cite{Zs}, $|L_1|$ is divisible by 
some prime divisor $\ell$ of $2^p-1$ that does not divide $\prod^{p-1}_{i=1}(2^i-1)$,
whence $\ell \nmid |S|$, again a contradiction. 

Next we observe that every normal abelian subgroup $A$ of $Y$ must be central 
and so cyclic. Indeed, 
$A$ acts by scalars on $W$ by (b), and so $A \leq \bfZ(Y)$.

\smallskip
(d) We have shown that $F^*(Y) = F(Y)$.  Now if $p$ divides $|F(Y)|$, then since 
$|S|_p = p$, $P = \bfO_p(F(Y)) \lhd Y$, a contradiction. Also, if $F(Y) \leq \bfZ(Y)$, then
$Y \leq \bfC_Y(F(Y)) \leq F(Y)$, and so $Y = F(Y)$ is nilpotent, again a contradiction. So 
$F(Y)$ is a $p'$-group and moreover $N := \bfO_r(F(Y))$ is non-central in $Y$ for some 
prime $r \neq p$. By (c), every characteristic abelian subgroup of $N$ is cyclic. Hence by 
Hall's theorem,  $N = F*D$, where $F$ is an extraspecial $r$-group, and either $D$ is cyclic,
or $r = 2$ and $C$ is dihedral, generalized quaternion, or semi-dihedral. Arguing as in part (3) of
the proof of \cite[Theorem 6.7]{GNT}, we can find a characteristic subgroup $L$ of $N$ such 
that $L = \bfZ(L)E$, where $E$ is an extraspecial $r$-group of order $r^{2c+1}$ 
for some $c \geq 1$ and $\bfZ(L)$ is cyclic. Note that $\bfZ(L) \leq \bfZ(Y)$
by (c). It also
follows by (b) that $r^c\mid b$ and $r \neq 2$, whence $a \geq 3$ must be odd (recall that
$p = 2^a \pm 1$ is prime). 
As $L \not\leq \bfZ(Y)$, some $p$-element $1 \neq g \in Y$ normalizes but
does not centralize $L$, and centralizes $\bfZ(L)$. It follows that $p$ 
divides the order of the group $\Out_c(L)$ of outer automorphisms of $L$ that
act trivially on $\bfZ(L)$. Since $\Out_c(L) \hookrightarrow \Sp_{2c}(r)$, we 
get $p\mid (r^d \pm 1)$ for some $d \leq c$. Since $p \geq 2^a-1$ and $r^c\mid a$, we arrive at 
a contradiction.
\end{proof}

Now we can prove Theorem \ref{thm:abelian} in a slightly stronger version.

\begin{thm} \label{thm:abelian2}  
Let $G$ be a finite irreducible
subgroup of $\GL(V) \cong \GL_d(k)$ with $k$ algebraically closed of
characteristic $p > 3$.   Assume that $G$ has a composition
factor of order $p$ and $d < 2p -2 $.
Then $d = p-1$, $p$ or $p+1$, a Sylow $p$-subgroup of $G$
has order $p$, $G$ is solvable and one of the following holds:
\begin{enumerate}[\rm(i)]
\item  $d=p - 1$, $p = 2^a + 1$ is a Fermat prime, $F^*(G) = \bfZ(G)\bfO_2(G)$, 
$G/F^*(G)$ is a Frobenius group of order $pe$ for some $e\mid 2a$ with kernel of order $p$, 
and $\bfO_2(G)$ is a group of symplectic type with $\bfO_2(G)/\bfZ(\bfO_2(G)) \cong C_2^{2a}$;
\item $d=p$ and $G$ has a normal abelian $p'$-subgroup $N = F^*(G)$ such that
$G/N$ is a Frobenius group of order $pe$ for some $e\mid (p-1)$ with kernel of order $p$;
\item $d=p + 1$, $p = 2^a - 1$ is a Mersenne prime and one of
\begin{enumerate}[\rm(a)]
\item $G$ has an abelian normal $p'$-subgroup $N$, where the action of 
$G/N$ on the $d$ distinct eigenspaces of $N$ induces a subgroup of $\SSS_d$ as 
described in Lemma \ref{lem:perm}; or
\item $F^*(G) = \bfZ(G)\bfO_2(G)$, $G/F^*(G)$ is a Frobenius group of 
order $2bp$ for some $b\mid a$ with kernel of order $p$, and $\bfO_2(G)$
is a group of symplectic type with $\bfO_2(G)/\bfZ(\bfO_2(G)) \cong C_2^{2a}$.
\end{enumerate}
\end{enumerate}
Moreover, $H^1(G,k) \neq 0$ if and only if one of
the conclusions of Theorem \ref{thm:abelian} holds.
\end{thm}

\begin{proof}  
(a) First we show that $\GP$ is irreducible on $V$. To this end, 
let $W$ be an irreducible summand of $V_\GP$. Also let 
$K_1$ and $K_2$ denote the kernel of the action of $\GP$ on $W$ and on a $\GP$-invariant 
complement $V'$ to $W$ in $V$. Then $K_1 \cap K_2 = 1$ and so $K_1$ embeds in $\GP/K_2$ as a 
normal subgroup. Since $C_p$ is a composition factor of $\GP$, it follows that it is a 
composition factor of $\GP/K_1$ or of $\GP/K_2$. Replacing $W$ by another irreducible $\GP$-summand in
$V'$ if necessary, we may assume that $\GP/K_1$ has a composition factor of order $p$. 
Then $\dim W \geq p-1$ by Theorem \ref{bz}. This is true for all other irreducible $\GP$-summands
in $V$ and $\dim V < 2p-2$, whence the claim follows. In particular,
$\bfZ(\GP) = \bfZ(G) \cap \GP$. 

\smallskip
(b) By Lemma \ref{eg}, we have that $Q \not\leq \bfZ(\GP)$ and so
$Q \not\leq \bfZ(G)$ for $Q := \bfO_{p'}(\GP) \lhd G$. 
Suppose that $Q$ contains a non-central (in $G$)
abelian subgroup $K \lhd G$. Decompose $V = \oplus^n_{i=1}V_i$ 
into $K$-eigenspaces. By Clifford's theorem, $G$ acts transitively on
$\{V_1, \ldots ,V_n\}$, with kernel $N$, and $\GP$ does as well. 
Note that $n > 1$ as $K \not\leq \bfZ(G)$. Since $\GP$ is generated by 
$p$-elements, $n \geq p$.  But $d < 2p$, so $\dim V_i = 1$ and $N$ is an 
abelian $p'$-group. Now we can apply Lemma \ref{lem:perm} to 
$G/N \hookrightarrow \SSS_d$. If $d=p$, we are in case (ii) and $F^*(G) = N$.  
If $d > p$, then $d = p + 1=2^a$ and $G/N$ 
has the prescribed structure, i.e. case (iii)(a) holds.

\smallskip
(c) Assume now that $Q$ contains no abelian non-central (in $G$) 
subgroup $K \lhd G$.   
Let $N$ be minimal among subgroups of $Q$ 
that are normal but non-central in $G$, so $N$ is non-abelian.
By Lemma \ref{lem:p'-component}, $Q$ contains no components of $G$. Hence  
$E(N) = 1$ and $F^*(N) = F(N)$. If $F(N) < N$, the minimality of $N$ implies 
that $F(N) \leq \bfZ(G)$, but then $F(N) = F^*(N) \geq \bfC_N(F^*(N)) = N$, 
a contradiction. So $N = F(N)$ is nilpotent, and the 
minimality of $N$ again implies that $N$ is an $r$-group for some prime $r \neq p$. 
Let $A$ be any characteristic abelian subgroup of $N$. Then by the assumption, 
$A \leq \bfZ(G)$ and so it is cyclic. Thus every characteristic abelian 
subgroup of $N$ is cyclic (and central in $G$), 
and so Hall's theorem applies to $N$. Arguing as in 
part (d) of the proof of Lemma \ref{p-cycle} and using the minimality of $N$, we
see that $N = \bfZ(N)E$, where $E$ is an extraspecial $r$-group of order 
$r^{2a+1}$ for some $a$ and $\bfZ(N) \leq \bfZ(G)$ is cyclic; 
in particular, $r^a\mid d$. 
Since $N \not\leq \bfZ(\GP)$, there is a
$p$-element $x$ that induces a nontrivial outer automorphism of $N$ acting
trivially on $\bfZ(N)$. As 
$\Out_c(N) \leq \Sp_{2a}(r)$, we see that $p$ divides  $r^{2b} -1$ for some
$1 \leq b \le a$.  On the other hand, $r^a \leq d \leq 2p - 3$.  This implies 
that $r=2$, $p = 2^a \pm 1$ is either a Mersenne or Fermat prime, $d = 2^a$, 
and $N$ acts irreducibly on $V$. The latter then implies that 
$\bfC_G(N) = \bfZ(G)$ and $X := G/\bfZ(G)N$ is a subgroup of 
$\Out_c(N) \leq \Sp_{2a}(2)$ with $C_p$ as a composition factor. Now we can 
apply Lemma \ref{p-cycle} to $X$. Note that if $X$ stabilizes a maximal totally 
isotropic subspace of $N/\bfZ(N)$, then its inverse image in $N$ is an 
{\it abelian} normal non-central subgroup of $G$, contrary to our assumptions. 
Hence either $p = 2^a+1$ and we are in case (i), or $p = 2^a-1$ and we are in 
case (iii)(b).
Also note that $F^*(G) = \bfZ(G)N$ in either case. 

\smallskip
(d) If $G$ satisfies any of the conclusions of Theorem \ref{thm:abelian}
then $H^1(G,k) \neq 0$. Conversely, suppose that $H^1(G,k) \neq 0$. Then 
$G$ possesses a normal subgroup $L$ of index $p$. Thus we can apply 
the above results to $G$. In particular, $|G|_p = p$ and so $L = \bfO_{p'}(G)$. 
Now the description of $G$ in (i)--(iii) shows that $G$ must satisfy one of
the conclusions of Theorem \ref{thm:abelian}.  
\end{proof}

One can also consider an analogue of Theorem \ref{thm:abelian2} for 
$p = 3$. In this case, 
$d \le 3$ and the analogous result is straightforward  by examining subgroups
of $\GL_2$ and $\GL_3$.  

\section{Bounding \texorpdfstring{$\Ext^1_G(V,V)$}{Ext\^{}1\_G(V,V)} and \texorpdfstring{$\Ext^1_G(V,V^*)$}{Ext\^{}1\_G(V,V*)}}\label{sec: ext2}
The following result is well known:

\begin{lem}\label{1-dim}
Let $X$ be a finite group and let $k$ be an algebraically
closed field of characteristic $p$. Let $U$ and $V$ be irreducible $kX$-modules belonging to 
a $kX$-block $B$ with cyclic defect subgroups. Then $\dim_k\Ext^1_X(U,V) \leq 1$. 
\end{lem}

\begin{proof}
By \cite[Lemma 7.1]{GHT}, we may assume $U \not\cong V$. It is known \cite{Pea} that 
$\cP(V)$ has simple head and simple socle, both isomorphic to $V$, and $\rad(\cP(V))/V$ is a direct sum of at most two uniserial submodules. Also, 
note that $\Ext^1(U,V) \cong \Hom_G(U,\cP(V)/V)$. So if $\dim_k\Ext^1_X(U,V) \geq 2$, then at least two edges of the Brauer tree of $B$ correspond to $U$, which is impossible. 
\end{proof}

\begin{lem}\label{str1}
Let $H$ be a finite group with Sylow $p$-subgroups of order $p$. Suppose 
that $H = \bfO^{p'}(H)$ and $H$ has no composition factor of order $p$. Then 
$H/\bfO_{p'}(H)$ is a non-abelian simple group.
\end{lem}

\begin{proof}
Replacing $H$ by $H/\bfO_{p'}(H)$, we have that $\bfO_{p'}(H) = 1$. Together with the condition
that $H$ has no composition factor of order $p$, this implies that $F(H) = 1$, and so
$$F^*(H) = E(H) = S_1 \times \ldots \times S_n$$ 
is a product of non-abelian simple groups. It also follows that $\bfO_{p'}(F^*(H)) = 1$, whence
$n = 1$ and $F^*(H) = S_1$ has order divisible by $p$. Now $H/S_1$ is a $p'$-group and
$H = \bfO^{p'}(H)$.  It follows that $H = S_1$.  
\end{proof}

\begin{prop}\label{str2}
Let $G$ be a finite group with a faithful $kG$-module $V$. Suppose that
$V = W_1 \oplus \ldots \oplus W_t$ is a direct sum of $k\GP$-submodules, and that for each $i$
the subgroup $H_i \leq \GL(W_i)$ induced by the action of $\GP$ on $W_i$ has Sylow $p$-subgroups
of order $p$. Suppose in addition that $G$ has no composition factor of order $p$. Then 
$$\GP/\bfO_{p'}(\GP) \cong S_1 \times \ldots \times S_n$$
is a direct product of non-abelian simple groups $S_i$, each of order divisible by $p$. 
\end{prop}

 
\begin{proof}
By assumption, $H_i$ has no composition factor of order $p$, $|H_i|_p = p$, and 
$\bfO^{p'}(H_i) = H_i$. By Lemma \ref{str1}, $H_i/\bfO_{p'}(H_i)$ is simple non-abelian.  Hence 
the claim follows by \cite[Lemma 2.3]{GHT}. 
\end{proof}

\begin{prop}\label{kernel}
Let $k = \overline{k}$ be of characteristic $p$ and let 
$H$ be a finite group such that $H = \bfO^{p'}(H)$ and 
$$H/J = S_1 \times \ldots \times S_n$$
is a direct product of non-abelian simple groups of order divisible by $p$, where $J := \bfO_{p'}(H)$.
Suppose that $W_1$ and $W_2$ are irreducible $kH$-modules such that the image of $H$ in
$\GL(W_i)$ has Sylow $p$-subgroups of order $p$ for $i = 1,2$, and that 
$\Ext^1_H(W_1,W_2) \neq 0$. Then the actions of $H$ on $W_1$ and $W_2$ have the same kernel.
\end{prop}

\begin{proof}
Let $K_i$ denote the kernel of the action of $H$ on $W_i$, so that $|H/K_i|_p = p$. Note that $H$,
and so $K_1 \cap K_2$ as well, has no composition factor of order $p$, whence 
$K_1 \cap K_2 = \bfO^p(K_1 \cap K_2)$. Hence by Lemma \ref{infl} there is no loss to assume 
that 
\begin{equation}\label{k12}
  K_1 \cap K_2 = 1.
\end{equation}  
We aim to show in this case that $K_1 = K_2 = 1$. Note that the condition $H=\bfO^{p'}(H)$ implies
that $n \geq 1$.

\smallskip
(i) Suppose for instance that 
$J_1 := J \cap K_1 \neq 1$. This implies by (\ref{k12}) that $J_1 \not\leq K_2$, i.e. $J_1$ does not
act trivially on $W_2$. Since $J_1 \lhd H$, we see that $(W_2)_{J_1}$ is a direct sum of
nontrivial $kJ_1$-modules. On the other hand, the $p'$-group $J_1$ acts trivially on $W_1$ and
on $W_1^*$. Setting $M := W_1^* \otimes_k W_2$, we then have that $M^{J_1} = 0$ and so
$$\Ext^1_H(W_1,W_2) \cong H^1(H,M) \cong H^1(H/J_1,M^{J_1}) = 0,$$
a contradiction.   

\smallskip
(ii) We have shown that $J \cap K_1 = J \cap K_2 = 1$. Hence 
$$K_1 \cong K_1J/J \lhd H/J = S_1 \times \ldots \times S_n,$$
and so $K_1$ is isomorphic to the direct product $\prod_{i \in I}S_i$ for some subset
$I \subseteq \{1,2, \ldots ,n\}$. As $|H/K_1|_p = p$ and $p\mid |S_i|$ for all $i$, we 
may assume that 
\begin{equation}\label{k1}
  K_1 \cong JK_1/J = S_2 \times \ldots \times S_n.
\end{equation}  
In particular, if $n = 1$ then $K_1 = 1$ and similarly $K_2 = 1$, whence we are done. 

\smallskip
(iii) Now we assume that $n \geq 2$. Consider $X := JK_2 \cap K_1$. Then 
$X \cap K_2 \leq K_1 \cap K_2 = 1$ and so 
$$X \cong XK_2/K_2 \lhd JK_2/K_2 \cong J$$
is a $p'$-group. On the other hand, $X \lhd K_1$ and so by (\ref{k1}) we again have that
$X \cong \prod_{i \in I'}S_i$ for some subset $I' \subseteq \{2, \ldots ,n\}$. As $p\mid |S_i|$ for all
$i$, we conclude that $X = 1$. Similarly $JK_1 \cap K_2 = 1$. Together with (\ref{k1}), this implies that
$$K_2  \hookrightarrow H/JK_1 \cong (H/J)/(JK_1/J) \cong S_1.$$
Furthermore, as shown in (ii), 
$JK_2/J \cong K_2 \cong \prod_{i \in I''}S_i$ for some subset $I'' \subseteq \{1,2, \ldots ,n\}$ of 
cardinality $n-1$. It follows that $n = 2$, $K_2 \cong S_1$, 
$K_1 \cong S_2$, and 
$$H = JK_2K_1 \cong JK_2 \times K_1 = (J \times K_2) \times K_1 \cong J \times K_1 \times K_2.$$ 
Now we can write 
$$W_1 \cong A_1 \otimes_k k \otimes_k B_2,~~W_2 \cong A_2 \otimes_k B_1 \otimes_k k,$$
where $A_1,A_2 \in \IBr_p(J)$ and $B_i \in \IBr_p(K_i)$ for $i = 1,2$. In this case, if
$B_1 \not\cong k$ and $B_2 \not\cong k$, then $\Ext^1_H(W_1,W_2) = 0$ by 
Lemma \ref{lem:kunneth}(i), a contradiction. So we may assume that $B_1 \cong k$, i.e. 
$K_1$ acts trivially on $W_2$. It then follows that $K_1 \leq K_1 \cap K_2  = 1$, contradicting 
(\ref{k1}) and the equality $n=2$.
\end{proof}

\begin{proof}[Proof of Theorem \ref{ext2}]
We take the convention
that $V^\eps$ is $V$ for $\eps = +$ and $V^*$ if $\eps = -$, and the same holds for other 
modules. Assume that $\Ext^1_G(V,V^\eps) \neq 0$ for some $\eps = \pm$. Decompose 
$$V_\GP = e\bigoplus^t_{i=1}W_i,~~V^\eps_\GP = e\bigoplus^t_{i=1}W_i^\eps$$
where $W_1, \ldots ,W_t$ are pairwise non-isomorphic and 
$G$-conjugate irreducible $k\GP$-modules. By assumption, the image of $\GP$ in each 
$\GL(W_i)$ has Sylow $p$-subgroups
of order $p$, and $\GP$ has no composition factor of order $p$. By Proposition \ref{str2},
\begin{equation}\label{for-gp}
  \GP/\bfO_{p'}(\GP) = S_1 \times \ldots \times S_n
\end{equation}  
is a direct product of non-abelian simple groups of order divisible by $p$.

\smallskip
(i) First we consider the case $k = \overline{k}$. Recall that 
$\GP = \bfO^{p'}(\GP)$. So by Proposition \ref{kernel} we have that 
$\Ext^1_\GP(W_i,W_j^\eps) = 0$, unless $\GP$ has the same kernel on $W_i$ and $W_j^\eps$.

Let $K_i$ denote the kernel of $\GP$ on $W_i$, and on $W_i^\eps$ as well. Relabeling 
$W_1, \ldots, W_t$, we may assume that $K_1, \ldots ,K_s$ are pairwise distinct, with
$s := |\{K_1, \ldots, K_t\}|$. Defining
$$V_i := e\bigoplus_{j:K_j = K_i}W_j$$
for $1 \leq i \leq s$, we then have that 
$$V = V_1 \oplus \ldots \oplus V_s,~~V^\eps = V_1^\eps \oplus \ldots \oplus V_s^\eps.$$   
Certainly, $G$ acts transitively on $\{W_1, \ldots ,W_t\}$, $\{V_1, \ldots,V_s\}$, and on 
$\{K_1, \ldots,K_s\}$ via conjugation. Also, $H := \bfN_G(K_1) \rhd \GP$ stabilizes $V_1$ and has 
index $s$ in $G$. It follows that $H = \Stab_G(V_1)$, $V \cong \Ind^G_H(V_1)$, and so 
$V_1$ is an irreducible $kH$-module. 

By the definition of $V_i$, we have that $\Ext^1_\GP(V_1,V_i^\eps) = 0$ for all $i > 1$. As 
$H/\GP$ is a $p'$-group, it  follows that $\Ext^1_H(V_1,\bigoplus_{i>1}V_i^\eps) = 0$. Now by 
Frobenius' reciprocity,
$$\Ext^1_G(V,V^\eps) = \Ext^1_G(\Ind^G_H(V_1),V^\eps) \cong 
    \Ext^1_H(V_1,(V^\eps)_H) = \Ext^1_H(V_1,V_1^\eps).$$
Recall that $K_1$ acts trivially on $V_1$ and $V_1^\eps$, and $|H/K_1|_p = |\GP/K_1|_p = p$. 
Hence $\dim_k\Ext^1_{H/K_1}(V_1,V_1^\eps) \leq 1$ by Lemma \ref{1-dim}. Finally,
$K_1$ has no composition factor of order $p$ by (\ref{for-gp}). So 
$\Ext^1_H(V_1,V^\eps_1) \cong \Ext^1_{H/K_1}(V_1,V_1^\eps)$ by Lemma \ref{infl}, and 
so we are done.

\smallskip
(ii) Now we consider the general case. By \cite[Theorem 9.21]{Is}, $W_1 \otimes_k \overline{k}$ is
a direct sum of irreducible $\overline{k}\GP$-modules $W_{11}, \ldots ,W_{1m}$, which form a Galois conjugacy class over $k$. By assumption, $|\GP/K_1|_p = p$, where $K_1$ is the kernel of 
$\GP$ on $W_1$. Certainly, $K_1$ is contained in the kernel $K_{11}$ of the action of
$\GP$ on $W_{11}$, whence $|\GP/K_{11}|_p \leq p$. If $|\GP/K_{11}|_p < p$, then
the equality $\GP = \bfO^{p'}(\GP)$ implies that $K_{11} = \GP$, whence
$\GP$ acts trivially on $W_{11}$ and so on $W_1$ and on $V$ as well, contradicting the 
faithfulness of $V$ and the assumption $\Ext^1_G(V,V^\eps) \neq 0$. Hence we must have 
that $|\GP/K_{11}|_p = p$. Since the dimension of $\Ext^1_G(V,V^\eps)$ does not change under
field extensions, we are done by replacing $V$ by $V \otimes_k\overline{k}$ and applying
the result of (i).
\end{proof}

A key ingredient of the proof of Theorem \ref{ext} is the following statement:

\begin{prop}\label{dim-ext1}
Let $X$ be a finite group with a normal subgroup $Y \geq \bfO^{p'}(X)$. Let $A$, $B$, $W$, $W'$ be 
$kX$-modules, where $A$ and $B$ are absolutely irreducible and  $Y$ acts via scalars on both $A$ and $B$. Suppose in addition that  $A_Y \cong B_Y$.  
Then 
$$\dim_k\Ext^1_X(A \otimes_k W, B \otimes_k W') \leq \dim_k\Ext^1_Y(W_Y,W'_Y).$$ 
\end{prop}

\begin{proof}
Since the dimensions of $\Ext^1$-spaces do not change under field extensions, we may 
assume that $k$ is algebraically closed. By assumption, $X/Y$ is a $p'$-group and $Y$ acts trivially on $A^* \otimes_kB$. Without loss we may assume that $\dim_k B \leq \dim_k A$. 
Denoting $$H := H^1(Y,(W_Y)^* \otimes_k W'_Y) \cong \Ext^1_Y(W_Y,W'_Y)$$
we then have $\dim_k B \otimes_k H \leq (\dim_kA)(\dim_kH)$ and so 
$$\dim_k\Hom_{kX}(A, B \otimes_k H) \leq \dim_k H = \dim_k \Ext^1_Y(W_Y,W'_Y).$$
Applying the inflation-restriction spectral sequence, we obtain
$$\begin{array}{ll}\dim_k\Ext^1_X(A \otimes_k W, B \otimes_k W')
     & = \dim_k(H^1(Y,A^* \otimes_k B \otimes_k W^* \otimes_k W'))^{X/Y}\\
     & =  \dim_k (A ^*\otimes_k B \otimes _k H)^{X/Y}\\
     & =  \dim_k \Hom_{kX}(A, B \otimes_k H)\\
     & \leq \dim_k\Ext^1_Y(W_Y,W'_Y).
     \end{array}$$   
\end{proof}

\begin{prop}\label{dim-ext2}
Let $k$ be an algebraically closed field of characteristic $p$.
Assume the hypothesis of Theorem \ref{ght:adequate} and write 
$V_\GP = \oplus^t_{i=1}V_i$, where $V_i \cong eW_i$ and $W_1, \ldots,W_t$ are 
pairwise non-isomorphic irreducible $k\GP$-modules. Suppose that there is a unique 
$j \geq 1$ such that $\Ext^1_\GP(W_1,W_j) \neq 0$.
Then 
$$\dim_k\Ext^1_G(V,V) \leq \dim_k\Ext^1_\GP(W_1,W_j).$$
\end{prop}

\begin{proof}
Let $G_1 := \Stab_G(V_1)$ be the inertia group of $W_1$ in $G$. Since $\GP \lhd G_1$, 
the uniqueness of $j$ implies that $G_1 = \Stab_G(V_j)$ as well. Next, 
$V \cong \Ind^G_{G_1}((V_1)_{G_1})$, and so
$$\begin{array}{ll}\Ext^1_G(V,V) & = \Ext^1_G(\Ind^G_{G_1}((V_1)_{G_1}),V) \cong 
    \Ext^1_{G_1}((V_1)_{G_1},V_{G_1}) \smallskip \\
    & \cong \Ext^1_{G_1}((V_1)_{G_1},(V_j)_{G_1})
    \oplus \Ext^1_{G_1}((V_1)_{G_1},\bigoplus_{i \neq j}(V_i)_{G_1}).\end{array}$$
Since $\GP$ contains a Sylow $p$-subgroup of $G_1$, 
$\Ext^1_{G_1}((V_1)_{G_1},\bigoplus_{i \neq j}(V_i)_{G_1})$ injects in
$$\Ext^1_{\GP}((V_1)_{\GP},\bigoplus_{i \neq j}(V_i)_{\GP}) \cong 
    e^2 \bigoplus_{i \neq j}\Ext^1_{\GP}(W_1,W_i) = 0$$
and so it is zero. 

\smallskip
It remains therefore to show that 
$$\dim_k\Ext^1_{G_1}((V_1)_{G_1},(V_j)_{G_1}) \leq \dim_k\Ext^1_\GP(W_1,W_j).$$
Let $X$ denote a universal $p'$-cover of $G_1$ (so that $G_1 \cong X/Z$ for some 
$p'$-subgroup $Z \leq \bfZ(X) \cap [X,X]$), and let $Y := \bfO^{p'}(X)$.
Now we view $V_1$ as an irreducible $kX$-module by inflation and note that 
$$\dim_k\Ext^1_{G_1}((V_1)_{G_1},(V_j)_{G_1}) = \dim_k\Ext^1_{X}((V_1)_{X},(V_j)_{X})$$
as $Z$ is a $p'$-group. 
Since $Z$ acts trivially on $V_1$, we also have that
$(V_1)_Y \cong e(W_1)_Y$ and also $YZ/Z \cong \GP$. Hence $(W_1)_Y$ is irreducible, and similarly
for $W_j$. Also,
$$\dim_k\Ext^1_{Y}((W_1)_{Y},(W_j)_{Y}) = \dim_k\Ext^1_{\GP}(W_1,W_j).$$

Fix a basis of $W_1$ and the corresponding representation $\Phi$ of $Y$ on $W_1$ in this basis.
By the Clifford theory, we 
can decompose the irreducible representation $\Theta$ of $X$ on $V_1$ as a tensor product of 
an irreducible projective representation $\Lambda$ of $X/Y$ (of degree $e$)
and an irreducible projective representation $\Psi$ of $X$, with 
$$\Psi(y) = \Phi(y)$$
for all $y \in Y$. Since $X$ is $p'$-centrally closed, there is a function $f:X \to k^\times$
such that 
$$\Psi':x \mapsto f(x)\Psi(x)$$ 
is a linear representation. Note that $f_Y \in \Hom(Y,k^\times)$ since $\Psi_Y = \Phi$ is a linear
representation, and so $f_Y = 1_Y$ as 
$Y = \bfO^{p'}(Y)$. In particular, $\Psi'(y) = \Phi(y)$ for all $y \in Y$.
Now we inflate $\Lambda$ to a projective representation of $X$ and define 
$$\Lambda':x \mapsto f(x)^{-1}\Lambda(x)$$
so that $\Theta(x) = \Lambda'(x) \otimes \Psi'(x)$ for all $x \in X$.
Then $\Lambda'$ is also a linear representation of $X$ and furthermore $\Lambda'_Y$ is 
trivial (since $f_Y = 1_Y)$. 
Thus we can decompose
$$(V_1)_X = A \otimes_k W,$$
where the $kX$-modules $A$ and $W$ are irreducible, $Y$ acts trivially on $A$, 
and $W_Y \cong (W_1)_Y$. Similarly,
$$(V_j)_X = B \otimes_k W',$$
where the $kX$-modules $B$ and $W'$ are irreducible, $Y$ acts trivially on $B$, 
and $W'_Y \cong (W_j)_Y$. Now our statement follows by applying Proposition \ref{dim-ext1}.
\end{proof}

The same proof as above yields:

\begin{prop}\label{dim-ext3}
Let $k$ be an algebraically closed field of characteristic $p$.
Assume the hypothesis of Theorem \ref{ght:adequate} and write 
$V_\GP = \oplus^t_{i=1}V_i$, where $V_i \cong eW_i$ and $W_1, \ldots,W_t$ are 
pairwise non-isomorphic irreducible $k\GP$-modules. Suppose that there is a unique 
$j \geq 1$ such that $\Ext^1_\GP(W_1,W^*_j) \neq 0$.
Then 
$$\dim_k\Ext^1_G(V,V^*) \leq \dim_k\Ext^1_\GP(W_1,W^*_j).$$
\hfill $\Box$
\end{prop}

\begin{lem}\label{extra}
Given the assumption of Theorem \ref{ext}, suppose that $H$ is as in the extraspecial case (e) of Theorem \ref{bz}. Then $$\Ext^1_G(V,V^*) = 0.$$
\end{lem}

\begin{proof}
Write $V_\GP = e\bigoplus^t_{i=1}W_i$ as usual. It suffices to show
that  $\Ext^1_\GP(W_i,W^*_j) = 0$ for all $i,j$. 
Recall that $J := \bfO_{p'}(G)$ acts irreducibly on $W_i$ and $W_j^*$ by \cite[Theorem 2.4(ii)]{GHT}.
Since $J$ is a $p'$-group, we have $M = \bfC_M(J) \oplus [M,J]$ for 
$M := W^*_i \otimes W^*_j$. As $J$ has no fixed point on $[M,J]$, 
$H^1(\GP,[M,J]) = 0$. Also, 
$$\bfC_M(J) \cong \Hom_J(W_i,W_j^*)$$
is either $0$ or $k$. Hence
$$\Ext^1_\GP(W_i,W_j^*) \cong H^1(\GP,M) \cong H^1(\GP, \bfC_M(J))
    \hookrightarrow H^1(\GP,k)$$  
As $\GP$ is perfect by \cite[Theorem 2.4]{GHT}, $H^1(\GP,k) = 0$, and so we are done.
\end{proof}

\begin{proof}[Proof of Theorem \ref{ext}]
(i) Assume that $(G,V)$ satisfies all the hypotheses of Theorem \ref{ext}. We take the convention
that $V^\eps$ is $V$ for $\eps = +$ and $V^*$ if $\eps = -$, and the same holds for other 
modules. Since the dimension of $\Ext^1_G(V,V^\eps)$ does not change under field extensions, we will assume that $k = \overline{k}$. Assume that
$\Ext^1_G(V,V^\eps) \neq 0$ for some $\eps = \pm$. It suffices to show that 
$G$ then fulfills the conditions of Propositions \ref{dim-ext2} and \ref{dim-ext3} (with
$\Ext^1_\GP(W_1,W_j^\eps) \cong k$ for the index $j$ indicated in these propositions). 
By \cite[Lemma 7.2]{GHT}, there is some $j$ such that 
\begin{equation}\label{eq:ext3}
  \Ext^1_\GP(W_1,W_j^\eps) \neq 0. 
\end{equation}  
Note that $\Ext^1_G(V,V^\eps) = 0$ in the extra special case (e) of 
Theorem \ref{bz}, by \cite[Proposition 10.4]{GHT} and Lemma \ref{extra}. So we may assume 
that the image $H$ of $\GP$ in $\GL(W)$ is a central product of quasisimple groups, whence, 
by \cite[Theorem 2.4]{GHT}, 
$$\GP = L_1 * \ldots * L_n$$
is also a central product of quasisimple groups $L_i$.  Moreover, 
if some $L_i$ is not a quasisimple group of Lie type in characteristic $p$, then by 
\cite[Theorem 2.4]{GHT}, the image of $\GP$ in each $\GL(W_i)$ has Sylow $p$-subgroups 
of order $p$, and so Theorem \ref{ext2} applies. So in what follows we may assume that 
all $L_i$ are quasisimple groups of Lie type in characteristic $p$.
Correspondingly, we can decompose
$$W_1 = A_1 \otimes \ldots \otimes A_n,~~~W_j^\eps = B_1 \otimes \ldots \otimes B_n,$$
where $A_i$ and $B_i$ are irreducible $kL_i$-modules and $L_{i'}$ acts trivially on $A_i$ and $B_i$
whenever $i' \neq i$.  By Lemma \ref{lem:kunneth} and (\ref{eq:ext3}), we may assume that 
$$A_i \cong B_i$$
for $i > 1$, and furthermore $\Ext^1_{L_1}(A_1,B_1) \neq 0$. 
Since $\dim_kW_1 = \dim_kW_j$, it follows that $\dim_k A_i = \dim_k B_i$ for all $i$. 

Note that if $\dim_k A_i = 1$, then $A_i \cong k$ as $L_i$ is perfect, and similarly 
$B_i \cong k$, whence $\Ext^1_{L_i}(A_i,B_i) = 0$. In fact, $\Ext^1_{L_i}(A_i,B_i) = 0$ if 
$\dim_kA_i \leq (p-3)/2$ by the main result of \cite{Gcr}. It follows that $\dim_k A_1 \geq (p-1)/2$.
Since $\dim_k W_1 \leq p-1$, we arrive at two possible cases:

(a) $\dim_k A_i = 1$ (and so $A_i \cong B_i \cong k$) for all $i > 1$; or

(b) $p \geq 5$, $\dim_k A_i = 1$ (and so $A_i \cong B_i \cong k$) for all $i > 2$, 
and $\{\dim_kA_1,\dim_kA_2\} = \{(p-1)/2,2\}$.

\smallskip
(ii) Suppose we are in case (b). Then the quasisimple group $L_m$ (for some $m \in \{1,2\}$) is 
acting irreducibly on $A_m\cong k^2$. As $L_m$ is 
a Lie-type group in characteristic $p$, we have that $L_m \cong \SL_2(p^a)$ for some 
$a \geq 1$. 
By Lemma \ref{ext-defi}, $L_1 \cong \SL_2(p)$ (modulo a central subgroup), 
$\dim A_1 = (p-1)/2$, $\Ext^1_{L_1}(A_1,B_1) \cong k$, and $A_1 \cong B_1$. 
We have shown that $A_i \cong B_i$ for all $i$; in particular $W_j \cong W_1^\eps$. 
Now we have that $m = 2$, and $\dim_k\Ext^1_{L_2}(A_2,B_2)$ 
equals $0$ if $p^a > 5$ and $1$ if $p^a = 5$, see 
\cite[Lemma 8.1]{GHT}. Again by Lemma \ref{lem:kunneth},
$$\dim_k\Ext^1_\GP(W_1,W_j^\eps) = 1+\dim_k\Ext^1_{L_2}(A_2,B_2).$$
In the case $p^a = 5$, we have $(\dim W,H) = (4,\Omega^+_4(5))$ and 
conclude by Propositions \ref{dim-ext2} and \ref{dim-ext3} that $\Ext^1_G(V,V)$ and 
$\Ext^1_G(V,V^*)$ are at most $2$-dimensional. Moreover, Example \ref{sol-ext}(i) shows that 
the upper bound $2$ can indeed be attained.
If $p^a > 5$, then $\Ext^1_\GP(W_1,W_j^\eps) \cong k$ and $W_j \cong W_1^\eps$ 
for any $j$ satisfying (\ref{eq:ext3}).

\smallskip
(iii) Now we consider the case (a). Then
\begin{equation}\label{eq:ext5}
  \Ext^1_{L_1}(A_1,B_1) \cong \Ext^1_\GP(W_1,W_j^\eps) \neq 0
\end{equation}   
by Lemma \ref{lem:kunneth}. 

Suppose first that $p = 3$. Then $L_1 \cong \SL_2(3^a)$ for some $a \geq 2$, and 
$$W_1 = A_1 \otimes_k k \otimes_k \ldots \otimes_k k,
    W_j = B_1^\eps \otimes_k k \otimes_k \ldots \otimes_k k.$$ 
We may also assume that $A_1$ is the natural $kL_1$-module. 
If $a = 2$, then by (\ref{eq:ext5}) and \cite[Corollary 4.5]{AJL} we have that $B_1$ is isomorphic to the Frobenius twist $A_1^{(3)}$ of $A_1$, and 
$\Ext^1_{L_1}(A_1,B_1) \cong k^2$.
Thus $W_j$ is uniquely determined, and so $\dim_k\Ext^1_G(V,V^\eps) \leq 2$ by 
Propositions \ref{dim-ext2} and \ref{dim-ext3}. Suppose now that $a > 2$. Since 
$G_1 := \Stab_G(V_1)$ stabilizes 
the isomorphism class of $W_1$, we see that $G_1$ normalizes each of $L_1$ and 
$L_2 * \ldots * L_n$, and induces an inner-diagonal automorphism of $L_1$.  Next, 
by (\ref{eq:ext5}) and \cite[Corollary 4.5]{AJL} we have that $B_1$ is isomorphic to one of the Frobenius twists $A_1^{(3)}$, $A_1^{(3^{a-1})}$ of $A_1$, and $\Ext^1_{L_1}(A_1,B_1) \cong k$. Thus there are 
at most two possibilities for $W_j$, each stabilized by $G_1$. If only one of them occurs among
the submodules $W_i$, then we have $\dim_k\Ext^1_G(V,V^\eps) \leq 1$ by 
Propositions \ref{dim-ext2} and \ref{dim-ext3}. Suppose that both of them occur, say for $j_1$ and 
$j_2$. It follows that $G_1 = \Stab_G(V_{j_1}) = \Stab_G(V_{j_2})$ and furthermore 
both $V_{j_1}$ and $V_{j_2}$ are irreducible over $G_1$. 
Then, arguing as in the proof of Proposition \ref{dim-ext2} we have 
$$\begin{array}{ll}\Ext^1_G(V,V^\eps) & = \Ext^1_G(\Ind^G_{G_1}((V_1)_{G_1}),V^\eps) \cong 
    \Ext^1_{G_1}((V_1)_{G_1},V^\eps_{G_1}) \smallskip \\
    & \cong \Ext^1_{G_1}((V_1)_{G_1},(V^\eps_{j_1})_{G_1})
    \oplus \Ext^1_{G_1}((V_1)_{G_1},(V^\eps_{j_2})_{G_1})\end{array}$$
has dimension at most $2$. In fact, Example \ref{sol-ext} shows that the upper bound $2$ can
indeed be attained.

Suppose now that $p > 3$. Then by Lemma \ref{ext-defi}, $L_1 = \SL_2(p)$ (modulo a central 
subgroup), $A_1 \cong B_1$, $\Ext^1_{L_1}(A_1,B_1) \cong k$, $W_j \cong W_1^\eps$, and 
$\Ext^1_\GP(W_1,W_j^\eps) \cong  k$.


\smallskip
(iv) We have shown that in the case of Theorem \ref{ext}(i), there is a unique $j$ such that 
$\Ext^1_{\GP}(W_1,W_j^\eps) \neq 0$, in which case it has dimension $1$. Hence we are done
by Propositions \ref{dim-ext2} and \ref{dim-ext3}.
\end{proof}

\begin{example}\label{sol-ext}
(i) Let $p = 5$ and let $S = L_1 \times L_2$, with $L_i \cong \SL_2(5)$, be acting on 
$V = W_1 \otimes W_2$, where $W_i \cong k^2$ is an irreducible $kL_i$-module and
$L_i$ acts trivially on $W_{3-i}$. Note that the kernel of this action is the diagonal cyclic subgroup 
$Z \cong C_2$ of $\bfZ(L_1) \times \bfZ(L_2)$. Now $G = \GP := S/Z \cong \Omega^+_4(5)$ acts faithfully and irreducibly on $V$, and $\dim_k\Ext^1_G(V,V) = 2$ by Lemma \ref{lem:kunneth}. 
Also, $V \cong V^*$.

\smallskip
(ii) Let $p = 3$, $S = \SL_2(3^a)$ for some $a\geq 2$ coprime to $3$, 
$W_1 = k^2$ be the natural $kS$-module, and 
let $W_{i+1}$ denote the Frobenius $(W_1)^{(3^i)}$ twist of $W_1$ for $1 \leq i \leq a-1$. Then 
$G = S \rtimes \langle \sigma \rangle$ (with $\sigma$ being the field automorphism of $S$, of order $a$) acts irreducibly and faithfully on 
$V = W_1 \oplus \ldots \oplus W_a$, $\GP = S$, and 
\begin{equation}\label{eq:ext4}
  \Ext^1_G(V,V) \cong \bigoplus^{a}_{i=1}\Ext^1_S(W_1,W_i) \cong k^2
\end{equation}  
by \cite[Corollary 4.5]{AJL}. (Indeed, if $a = 2$ then $\Ext^1_S(W_1,W_2) \cong k^2$.
If $a \geq 3$, then $\Ext^1_S(W_1,W_2) \cong \Ext^1_S(W_1,W_{a}) \cong k$. All other
summands in the middle term of (\ref{eq:ext4}) are zero.) Also, $V \cong V^*$.

\smallskip
(iii) Let $p = 2^f+1$ be a Fermat prime and let $H = \bfO_{p'}(H)P$ (with $P \cong C_p$ and 
$\bfO_{p'}(H) \cong 2^{1+2f}_{-}$)  acting faithfully and 
absolutely irreducibly on $W_1 = k^{p-1}$ as in case (i) of Theorem \ref{bz}. Note that 
the $kH$-module $W_1$ is self-dual.  Let $n$ be coprime to $p$ and let 
$$G = H_1 \wr C_n = (H_1 \times \ldots \times H_n) \rtimes C_n$$
with $H_i \cong H_1 = H$, so that $\GP = H_1 \times \ldots \times H_n$. Inflate $W_1$ to 
a $k\GP$-module and consider $V := \Ind^G_\GP(W_1)$. Note that $J := \bfO_{p'}(\GP)$ acts absolutely irreducibly on $W_1$, and 
$$W_1^* \otimes W_1 = \bfC_{W^*_1 \otimes W_1}(J) \oplus [W_1^* \otimes W_1,J]$$
with $\bfC_{W^*_1 \otimes W_1}(J) \cong k$. Since $\GP/J \cong C_p^n$, it now follows that 
$$\Ext^1_\GP(W_1,W_1) \cong H^1(\GP,W_1^* \otimes W_1) \cong H^1(\GP/J,k) \cong k^n.$$
On the other hand, the actions of $J$ on $W_1$ and $W_j$ have different kernels for any $j > 1$,
and so $\Ext^1_\GP(W_1,W_j) = 0$. 
Hence $V \cong V^*$ and 
$$\Ext^1_G(V,V) \cong \Ext^1_\GP(W_1,V_\GP) \cong \Ext^1_\GP(W_1,W_1) \cong k^n.$$
\end{example}

Next we strengthen Theorem \ref{ext} in the case $\dim W$ is small.

\begin{thm}\label{small1}
Let $k$ be a field of characteristic $p$ and let $V$ and $V'$ be absolutely irreducible faithful
$kG$-modules. Suppose that $\dim_kW + \dim_kW' \leq p -2$, where $W$ and $W'$ are irreducible 
$k\GP$-submodules of $V$ and $V'$, respectively. Then
$H^1(G,M) = 0$ for any subquotient $M$ of the $G$-module $V \otimes V'$.
\end{thm}

\begin{proof} It suffices to prove $H^1(\GP,M) = 0$. Note 
that $V_\GP = \oplus^t_{i=1}W_i$ and $V'_\GP = \oplus^{s}_{j=1}W'_j$
with $W_i, W'_j \in \IBr_p(\GP)$, and $\bfO_p(\GP) \leq \bfO_p(G) = 1$. Since 
\begin{equation}\label{eq:small}
  \dim_kW_i + \dim_kW'_j \leq p-2,
\end{equation} by the main result of \cite{Gcr} we have $\Ext^1_\GP(W^*_i,W_j') = 0$.
It follows that $\Ext^1_\GP(V^*_\GP,V'_\GP)=0$, i.e. $H^1(\GP,(V \otimes V')_\GP) = 0$.
By Corollary 1 to \cite[Theorem 1]{Serre1},
(\ref{eq:small}) also implies that the $\GP$-module $V \otimes V'$ is semisimple. Thus 
$M$ is isomorphic to a direct summand of $(V \otimes V')_\GP$, 
whence $H^1(\GP,M) = 0$, as desired. 
\end{proof}

\begin{cor}\label{small2}
Let $k$ be a field of characteristic $p$ and let $V$ be an absolutely irreducible faithful
$kG$-modules. Suppose that $\dim_kW < (p-1)/2$ for any irreducible $k\GP$-submodule of $V$.
Then $H^1(G, \Sym^2(V))=H^1(G, \wedge^2(V))=0$.
\hfill $\Box$
\end{cor}

\section{Modules of dimension $p$}  \label{sec: dim p}

Let $p$ be a prime and let $k$ be algebraically closed of characteristic $p$.
The aim of this section is to show that if $G$ is an irreducible subgroup of
$\GL_p(k)=\GL(V)$, then almost always $(G,V)$ is adequate (using Thorne's
new definition). We begin with some observations.  

\begin{remark}\label{g-gp}
Suppose that $G \leq \GL(V)$ is a finite irreducible subgroup.
Note that, to show $(G,V)$ is adequate it suffices to show that $\GP$
is adequate on $V$. Indeed, any subgroup being weakly adequate implies that
the spanning condition holds for $G$. Next, adequacy for any subgroup containing
a Sylow $p$-subgroup of $G$ implies that necessary vanishing of $H^1$ for
$G$.
\end{remark}

\begin{lemma}\label{p-div}
Let $G$ be a finite group with a Sylow $p$-subgroup $P$ of order $p$ and let
$V \in \IBr_p(G)$ be such that $p\mid (\dim V)$. Then $V$ is projective.
\end{lemma}

\begin{proof}
Assume that $V$ is non-projective and set $N := \bfN_G(P)$.
By the Green correspondence
\cite[Lemma 4.1.1]{HL}, in this case we have
$V_N = W \oplus M$, where $W$ is a non-projective indecomposable
$N$-module and $M$ is a projective $N$-module (or zero). Now $W$ belongs
to an $N$-block $b$ of defect $1$. By \cite[Lemma 4.2.14]{HL}, $W$ is a
uniserial (non-projective) quotient of $\cP(U)$ where
$U := \hd(W) \in \IBr_p(N)$. By \cite[Lemma 4.2.13]{HL}, $\cP(U)$ has length
$p$, so $W$ has length $l < p$. According to \cite[Remark 4.2.11]{HL}, all
simple $kN$-modules in $b$ are of the same dimension $d$ and have
$P$ in their kernel. It follows that $d$ divides $|N/P|$ and so $d$ is coprime
to $p$. Hence $p \nmid dl = \dim W$ and so $p \nmid (\dim V)$ (as $p\mid (\dim M$)), a contradiction.
\end{proof}

\begin{lemma}\label{h1h2}
Let $G$ be a finite group with a cyclic Sylow $p$-subgroup $P$ and $p = \Char(k)$.
Suppose that $G = \bfO^p(G)$.
Then $H^1(G,k) = H^2(G,k) = 0$.
\end{lemma}

\begin{proof}
The vanishing of $H^1(G,k)$ is obvious. Suppose that $H^2(G,k) \neq 0$. Since the dimension of $H^2$ does not change under extension of scalars, we may assume that $H^2(G,C_p) \neq 0$. As moreover $H^1(G,C_p) = 0$, it follows that
$p$ divides the order of the Schur multiplier of $G$. It is well known that the latter then implies
that Sylow $p$-subgroups of $G$ are non-cyclic (see e.g. \cite[Corollary (11.21)]{Is}).
\end{proof}

Next we give an example showing that for modules of dimension
$2p$, we can satisfy all conditions aside from the spanning condition.  

\begin{example}\label{ex:2p}   
Assume that $p > 2$. Let $C$ be a nontrivial cyclic group of order coprime to $p$, with 
a faithful character $\lambda:C \to k^\times$,
and let $G = C \wr D$ where $D$ is a dihedral group of order $2p$.  Let $V$ be
the irreducible $kG$-module of dimension $2p$  induced from the $1$-dimensional representation with character 
$$\lambda \otimes 1_{C} \otimes \ldots \otimes 1_{C}$$ of the abelian normal subgroup
$A = C \times C \times \ldots \times C \cong C^{2p}$ of $G$. Let $E$ be the unique subgroup of $D$ of order $p$. Note that $V = V_1 \oplus V_2$, where the $V_i$ are irreducible $AE$-submodules of $V$ 
(of dimension $p$). Then the following statements hold:
\begin{enumerate}
\item $H^1(G,k)=H^2(G,k)=0$ by Lemma \ref{h1h2};
\item $\ext_G^1(V,V)=0$ (indeed, $V$ is projective by Lemma \ref{p-div});
\item  The span $\cM$ of the $p'$-elements of $G$ in $\EE(V)$ is precisely
$\cA \oplus \Hom(V_1, V_2) \oplus \Hom(V_2, V_1)$, where $\cA$ is 
the image of $kA$ in $\EE(V_1) \oplus \EE(V_2)$.
\end{enumerate}
\end{example}

Now we describe all irreducible linear groups of degree $p$:

\begin{prop}\label{degree-p}
Let $k$ be an algebraically closed field of characteristic $p$ and let
$G < \GL_p(k)$ be a finite irreducible subgroup. Then one of the following holds:

\begin{enumerate}[\rm(i)]
\item $G$ is imprimitive on $W := k^p$, $G < \GL_1(k) \wr \SSS_p$, and
furthermore $A := G \cap \GL_1(k)^p$ is non-central in $G$;
\item $G$ is almost quasisimple. Furthermore, $H := G^{(\infty)}$ is quasisimple of
order divisible by $p$ acting irreducibly on $W$, and so $(H,W)$ is as
described in Theorem \ref{thm:2p}.
\end{enumerate}
\end{prop}

\begin{proof}
By the hypothesis, $G$ acts irreducibly on $W = k^p$. Suppose the action is imprimitive.
Then $G$ permutes transitively the $p$ summands of a decomposition
$W = W_1 \oplus \ldots \oplus W_p$, with kernel say $A$. If $A \not\leq \bfZ(G)$,
we arrive at (i). Assume that $A \leq \bfZ(G)$. Note that $S:= G/A$ is a transitive subgroup
of $\SSS_p$, and so we can apply the main result of \cite{Zie} to $S$. In particular,
if $S$ is solvable, then $S = P:C$ with $P \cong C_p$ and $C \leq C_{p-1}$.
Then $AP$ is a normal abelian subgroup of $G$, whence by Ito's theorem the
degree of any $\chi \in \Irr(G)$ divides $[G:AP]\mid (p-1)$. On the other hand, $G$ is solvable,
and so by the Fong-Swan theorem, $W$ lifts to an irreducible complex module of
dimension $p$, a contradiction. Thus $S$ is non-solvable, which implies by \cite{Zie}
that $S$ is almost simple, $G$ is almost quasisimple, and $H:=G^{(\infty)}$ is a normal subgroup of index coprime to $p$. Since $\dim W = p$, the last condition also
implies that $H$ is irreducible on $W$ and so we arrive at (ii).

We may now assume that the $G$-module $W$ is primitive.
Since $\dim(W) = p$ is prime, this module cannot be tensor decomposable
nor tensor induced. Now we can apply Aschbacher's theorem in the version given
in \cite[Proposition 2.8]{GT3} to $(G,W)$ to conclude that $G$ is almost quasisimple:
$S \lhd G/\bfZ(G) \leq \Aut(S)$ for some non-abelian simple group $S$. In
particular, $H = G^{(\infty)} \lhd G$ is quasisimple, and moreover irreducible
on $W$ by \cite[Lemma 2.5]{GT3}. Hence we can apply Theorem \ref{thm:2p} to
$(H,W)$.
\end{proof}

\subsection{Imprimitive case}  \label{subsection:imprimitive}


\begin{proposition}\label{imp-p}
Suppose we are in case {\rm (i)} of Proposition \ref{degree-p}. Then
$(G,W)$ is adequate if and only if $|G/A| \neq p$.
\end{proposition}

\begin{proof} 
Let $P \in \Syl_p(G)$, so that $|P| = p$ (if $P = 1$ then $G$ cannot
act irreducibly on $W$). If $|G/A| = p$, then $G = AP$,
$A = \bfO_{p'}(G)$ contains all the $p'$-elements of $G$ but does not
act irreducibly on $W$ (as $A$ is a $p'$-group), whence $G$ is not 
weakly adequate.

\smallskip
Now assume that $|G/A| \neq p$; in particular, $p > 2$. 
Suppose that $G$ has a normal 
$p$-complement $K$. Then $K = \bfO_{p'}(G) > A$ (as otherwise $G = AP$ and
so $|G/A| = p$), and $H := G/A \leq \SSS_p$ has a normal $p$-complement 
$K/A \neq 1$. Thus $H$ is a transitive subgroup of $\SSS_p$ with $C_p$ as 
a composition factor. But then $\bfO_{p'}(H) = 1$ by Lemma 
\ref{lem:perm}, a contradiction. 
Thus $G$ cannot have a normal $p$-complement, and so $H^i(G,k)=0$ for $i=1,2$
by Lemma \ref{h1h2}. Also, $W$ is projective as a $G$-module
by Lemma \ref{p-div}, whence $\ext_G^1(W,W)=0$.

Since $A \not\leq \bfZ(G)$, $A$ has $p$ distinct eigenspaces 
$W_1, \ldots ,W_p$ on $W$ permuted transitively by $P$.
Thus, it remains only to prove that  
$$\EE(W) \cong \oplus_{1 \leq i,j \leq p}\Hom(W_i,W_j)$$ 
is spanned by the images of the $p'$-elements of $G$. 
Given $1 \leq i \neq j \leq p$, we claim that there exists a $p'$-element 
$x \in G$ with $xW_i = W_j$. Since $P$ is transitive on  
$\{W_1, \ldots, W_p\}$, we can choose $y \in P$ with $yW_i=W_j$. 
Note that $N$ acts on this set
as the Frobenius subgroup $C_p \rtimes C_s$ of $\SSS_p$, with 
kernel $A \cap N$, and all the elements of $(C_p \rtimes C_s) \setminus C_p$
are $p'$-elements. So, since $s > 1$, we 
can find $z \in N \setminus{AP}$ such that $zW_j=W_j$ and set $x:=zy$.  Then
$xW_i=W_j$ and $x \in N \setminus {AP}$, whence $x$ is a $p'$-element.

Now $B :=\langle A, x \rangle$ is a $p'$-group. Note that 
$W_A = \oplus^p_{a=1}W_a$ is a direct sum of $p$ non-isomorphic 
simple $A$-submodules. Hence $W_B = \oplus^t_{b=1}U_b$ is a direct sum of 
$t \geq 1$ non-isomorphic simple $B$-modules, with 
$U_1 \supseteq W_i \oplus W_j$. By the Artin-Wedderburn theorem, the image of 
$kB$ in $\EE(W)$ is just $\oplus^t_{b=1}\EE(U_b)$, and so 
it contains 
$$\EE(U_1) \supseteq \EE(W_i) \oplus \EE(W_j) \oplus \Hom(W_i,W_j) \oplus 
  \Hom(W_j,W_i),$$ 
and the result follows.
\end{proof}

\subsection{Chevalley groups in characteristic $p$} \label{subsection:natural}

We first point out the following:

\begin{prop} \label{prime-natural1}   Let $H$ be a quasisimple finite group of Lie
type in characteristic $p$.  Let $k$ be an algebraically closed field of characteristic $p$.
Let $W$ be a faithful irreducible $kH$-module of prime dimension $r \le p$.   Then one of the following statements holds:
\begin{enumerate}[\rm(i)]
\item  $H=\SL_2(p^a)$ for $r = 2$ and $H = \PSL_2(p^a)$ for $r > 2$;
\item  $H=\SL_r(p^a)$ or $\SU_r(p^a)$, and $r > 2$;
\item  $H=\Omega_r(p^a)$ and $r \ge 5$; or
\item  $r =7$ and $H=\gtwo(p^a)$.
\end{enumerate}
\end{prop}

\begin{proof}
Since $\Char(k) = p$ and $V$ is faithful, $\bfO_p(H) = 1$. Hence there is a simple
simply connected algebraic group $\cG$ in characteristic $p$ and a Frobenius
endomorphism $F:\cG \to \cG$ such that $H \cong G/Z$ for $G := \cG^F$ and
$Z \leq \bfZ(G)$. Inflate $W$ to a $kG$-module.
Since $r = \dim W$ is prime,  $W$ is tensor indecomposable and
in particular is a twist of a restricted representation.  So we may assume
that $W$ is restricted and extend $W$ to a $k\cG$-module.  By \cite{jantzenlow}, it follows
that $W=L(\lambda)$ where $\lambda$ is a dominant weight, and
$\dim W$ equals the dimension of the Weyl module $V(\lambda$) labeled by $\lambda$.
Thus, we can apply the same result
for characteristic $0$ which was proved by Gabber \cite[1.6]{katz}.
\end{proof}

\begin{prop} \label{prime-natural2}
Suppose we are in case {\rm (ii)} of Proposition \ref{degree-p}. Assume in addition
that $H = G^{(\infty)}$ is a quasisimple group of Lie type in characteristic $p > 3$.
Then $(G,W)$ is adequate.
\end{prop}

\begin{proof}
By the hypothesis, we have that $H$ is quasisimple and
irreducible on $W$. So we can apply
Proposition \ref{prime-natural1} to $H$; in particular we have that $W_H = L(\lambda)$ is
a restricted module (up to a Frobenius twist; in what follows we will ignore this
twist). In the case $H = \PSL_2(p^a)$,
we have that $\lambda = (p-1)\varpi_1$, where $\varpi_1$ is the fundamental weight.
Since $W_H$ is $G$-invariant, we see that $G$ cannot induce nontrivial field automorphisms
on $H$; in particular, $\GP = H$. In other cases, applying Propositions 5.4.11 and 5.4.12 of
\cite{KL}, we see that $W_H \cong \cN$ or $\cN^*$ where $\cN$ is the natural $kH$-module of dimension $p$ (with highest weight $\varpi_1$), and again $\GP = H$.

By Remark \ref{g-gp}, without loss we may now assume $G = H$.
Note that all the {\it classical}
groups given in the previous proposition when $r=p$ contain
an irreducible subgroup $L \cong \PSL_2(p)$. Indeed, the irreducible
$kL$-representation of degree $p$ embeds $L$ in $M \cong \Omega_p(p)$. In turn,
$M$ embeds in $\SL_p(q)$ and $\SU_p(q)$ for any $q = p^a$.
The same is true for $\gtwo(p)$ with $p = 7$: $\gtwo(7) > \gtwo(2) > \PSL_2(7)$. (It is well known,
see e.g. \cite{Kl-g2} that $H = \gtwo(7)$ contains a maximal subgroup $X \cong \gtwo(2)$ which
acts irreducibly on the minimal $7$-dimensional $H$-module $W$. Next,
$X$ contains a maximal subgroup $Y \cong \PSL_2(7)$, cf. \cite{Atlas}. Using \cite{JLPW} one can
check that $Y$ is irreducible on $W$.)
Thus weak adequacy follows by \cite[Proposition 3.1]{GHT}.

It is well known that $H^1(G,k)=H^2(G,k)=0$ (since $p > 3$).  Thus, it
suffices to show that $\ext_G^1(W,W)=0$.   If $G=\PSL_2(p^a)$, the result
follows by \cite{AJL}.   If $G=\Omega_5(5)$, one computes directly that
$\ext_G^1(W,W)=0$ (this was done by Klaus Lux).
In all other cases, $\ext_G^1(W,W)=0$ by the main result of
\cite{Mc}.
\end{proof}

\subsection{Remaining cases}   \label{subsection:cross}

\begin{lem}\label{alt-p}
Let $k = \bar{k}$, $H = \AAA_{p+1}$ with $p \geq 5$, and let $W$ be an
irreducible $kH$-module of dimension $p$. Then $(H,W)$ is weakly adequate.
\end{lem}

\begin{proof}
Note that $W$ is irreducible over a subgroup $L \cong \PSL_2(p)$ of $H$. Hence the
claim follows by \cite[Proposition 3.1]{GHT}.
\end{proof}

We record the following useful observation:

\begin{lem}\label{mult-free}
Let $X$ be a finite $p'$-subgroup of $G < \GL(W)$ where $W$ is a finite dimensional
vector space over $k$. Suppose that $W_X$ is multiplicity-free. Then
$(\End(W)/\cM)^X = 0$.
\end{lem}

\begin{proof}
Note that the $X$-module $\End(W)$ is semisimple. Furthermore, the
multiplicity-free assumption implies $\cM \supseteq \End(W)^X$ by
the Artin-Wedderburn theorem. Hence the claim follows.
\end{proof}

\begin{prop}\label{weil-sp}
Let $k = \bar{k}$, $H = \PSp_{2n}(q)$ with $2 < p = (q^n \pm 1)/2$, and let $W$ be an
irreducible $kH$-module of dimension $p$. Then $(H,W)$ is weakly adequate.
\end{prop}

\begin{proof}
(a) Note that $W$ is a Weil module and restricts irreducibly to a subgroup
$\PSL_2(q^n)$ of $H$. So without loss we may assume $n = 1$. We will inflate $W$ to a $kL$-module for $L := \SL_{2}(q)$.
Note that $W$ is obtained by reducing modulo $p$ one of the four complex Weil
modules of $L$, with characters $\eta_i$ of degree $(q-1)/2$ and $\xi_i$ of degree
$(q+1)/2$, where $i = 1,2$ and $\xi_i+\eta_i$ is a {\it reducible Weil
character} of $L$, see e.g. \cite{GMST} and \cite{TZ2}. Let $\tau$ denote the
permutation character of $L$ acting on the set of all vectors of the natural module
$\cN := \bbF_q^{2}$. Using the character table of $L$ as given on \cite[p. 155]{DM}, we see that
\begin{equation}\label{tau1}
  (\xi_i+\eta_i)(\bar\xi_i+\bar\eta_i) = \tau.
\end{equation}
Let $P :=\Stab_L(\langle v \rangle_{\bbF_q})$ for some $0 \neq v \in \cN$
with normal subgroup $Q :=\Stab_L(v)$ of order $q$, and let
$\varphi$ denote the Brauer character of $W$. Assume the contrary:
$\cM \neq \End(W)$, and let $\vartheta$ denote the Brauer character of
$\cQ := \End(W)/\cM$.

\smallskip
(b) Consider the case $p = (q+1)/2$, whence $\varphi = \xi_i^\circ$ and
$P$ is a $p'$-group.
Inspecting the values of $\varphi_P$, we see that
$W_P = W_1 \oplus W_2$ with $W_1, W_2 \in \Irr(P)$ of dimension $1$ and
$(q-1)/2 > 1$. Moreover, $(W_1)_Q$ is trivial, and $W_2^Q = 0$.
By the Artin-Wedderburn theorem applied to $P$,
$\cM \supseteq \End(W_1) \oplus \End(W_2)$; in particular,
$\dim \cM^P \geq 2 = \dim \End(W)^P$.
Hence we conclude that for any composition factor $Y$ of $\End(W)/\cM$,
$Y^Q = 0$ and $\dim Y \leq q-1$.

Let $\rho$ denote the permutation character of $L$ acting on the $1$-spaces of $\cN$.
Then $\rho = 1_L + \St$, where $\St$ is the Steinberg character of $L$. Moreover,
all irreducible constituents of $\tau-\rho-1_L$ have degree $q+1$ or $(q+1)/2$ and thus
have $p$-defect $0$. Note that $\St^\circ = 1_L + \psi$ with $\psi \in \IBr_p(L)$ (see \cite{Burk}),
and $\rho$ is the character of the PIM $\cP(1_L)$ of $1_L$.
Since $\varphi = \xi_i^\circ$ has degree $p$ and (the projective module)
$\End(W)$ contains a trivial simple
submodule, we see that $\End(W)$ is the direct sum of $\cP(1_L)$  and some
$p$-defect $0$ modules of dimension $q+1$ or $(q+1)/2$. In particular, since
$\psi$ is the Brauer character of the heart of $\cP(1_L)$, it cannot be afforded by
a quotient of $\End(W)$, and so $\vartheta \neq \psi$. Since
$\vartheta(1) \leq q-1$, it follows that all irreducible constituents of $\vartheta$ are
of degree $1$ (with multiplicity $\leq 2$) and $(q+1)/2$ (with multiplicity $\leq 1$).
But the principal character and the Weil
characters of degree $(q+1)/2$ of $L$ all contain $1_Q$ when restricted to $Q$,
a contradiction.

\smallskip
(c) Now assume that $p = (q-1)/2$; in particular, $\varphi = \eta_i^\circ$ and
$q \equiv 3 (\mod 4)$. Consider a cyclic subgroup
 $C \cong C_{(q+1)/2}$ of $H$.
It is straightforward to check that for any $\chi \in \Irr(H)$, either
$[\chi_Q,1_Q]_Q \neq 0$ or $[\chi_C,1_C]_C \neq 0$. Since irreducible $p$-Brauer characters of $H$ lift to complex characters (cf. \cite{Burk}), it follows that for any $U \in \IBr_p(H)$,
either $U^Q \neq 0$, or $U^C \neq 0$.

Now we may assume $\varphi = \eta_1^\circ$
and observe that both $\varphi_Q$ and $\varphi_C$
are multiplicity-free. Hence, each of $Q$, $C$ has no nonzero fixed points on
$\End(W)/\cM$ by Lemma \ref{mult-free}. Consequently, $\cM = \End(W)$.
\end{proof}


\begin{prop}\label{weil-sl}
Let $k = \bar{k}$ and let $H = \SL_{n}(q)$, where either $3 < p = (q^n-1)/(q-1)$ or
$(n,p) = (2,q-1)$.  Let $W$ be an
irreducible $kH$-module of dimension $p$. Then $(H,W)$ is weakly adequate.
\end{prop}

\begin{proof}
Let $\cN = \langle e_1, \ldots ,e_n \rangle_{\bbF_q}$ denote the natural
$\bbF_q H$-module, and let $P :=\Stab_H(\langle e_1 \rangle_{\bbF_q})$.
Since $\SL_2(4) \cong \PSL_2(5)$ and
$\SL_3(2) \cong \PSL_2(7)$, we may assume $(n,q) \neq (2,4)$, $(3,2)$. Also, let
$\varphi$ denote the Brauer character of $W$.

\smallskip
(a) First we consider the case $p = (q^n-1)/(q-1)$. In this case, $W$ is induced from a
one-dimensional $kP$-module with character say $\lambda$. So we can
write $W = \oplus_{\omega \in \bbP\cN}W_{\omega}$ as a direct sum of
one-dimensional subspaces $W_{\omega}$ permuted transitively by $H$, where
$\bbP\cN$ is the set of $1$-spaces in $\cN$.

\smallskip
(a1) Assume in addition that $n \geq 3$. It suffices to show that, for any
two distinct $\omega_1 = \langle e\rangle_{\bbF_q}$,
$\omega_2 = \langle f\rangle_{\bbF_q} \in \bbP\cN$,
\begin{equation}\label{2points}
  \cM \supseteq \End(W_{\omega_1}) \oplus \Hom(W_{\omega_1},W_{\omega_2})
    \oplus  \Hom(W_{\omega_2},W_{\omega_1}).
\end{equation}
Since $H$ acts transitively on those pairs $(\omega_1,\omega_2)$, we may assume
that $e = e_1$ and $e=e_2$.

Consider an opposite parabolic subgroup
$R :=\Stab_H(\cN_1)$ for $\cN_1 := \langle e_1, \ldots ,e_{n-1}\rangle_{\bbF_q}$, which is a $p'$-subgroup. Then $R$ stabilizes the subspace
$W_1 := \oplus_{\omega \in \bbP\cN_1}W_\omega$ of dimension $(q^{n-1}-1)/(q-1)$.
Note that the unipotent radical $Q$ of $R$ acts trivially on $W_1$. Indeed, $q \geq 3$
since we are assuming $n \geq 3$ and $(n,q) \neq (3,2)$. Now the Levi subgroup
$L :=\Stab_H(\cN_1,\langle e_n \rangle_{\bbF_q})$ of $R$ acts transitively on
$q^{n-1}-1 > \dim W_1$ nontrivial linear characters of $Q$, whence the claim follows.
Now we identify $L$ with $\GL_{n-1}(q)$ via
$\diag(X, \det(X)^{-1}) \mapsto X$. Then the $L$-character of $W_1$ is just induced
from the character $\lambda_{P \cap L} \neq 1_{P \cap L}$ of the maximal  parabolic
subgroup $P \cap L$ (of index $(q^{n-1}-1)/(q-1)$) of $L$. Hence $W_1$ is an
{\it irreducible} Weil module of dimension $(q^{n-1}-1)/(q-1)$ for $L$, and it is irreducible
over $R$. Note that $\lambda$ is a linear character of order dividing $q-1 = |P/P'|$ and so it takes
value $1$ on any unipotent element of $P$. Using this, one can check that
$\varphi(t) = (q^{n-1}-1)/(q-1)$ for any $1 \neq t \in Q$. In particular,
$$\varphi_Q = \frac{q^{n-1}-1}{q-1}\cdot 1_Q + \sum_{\nu \in \Irr(Q)}\nu =
    (\dim W_1 +1)\cdot 1_Q + \sum_{1_Q \neq \nu \in \Irr(Q)}\nu.$$
It now follows (by Clifford's theorem) that $W_R = W_1 \oplus W_2 \oplus W_3$,
where $W_2$ has dimension $1$, $W^Q = W_1 \oplus W_2$, and $W_3$ is
irreducible of dimension $q^{n-1}-1$.  Applying
the Artin-Wedderburn theorem, we see that
$\cM \supseteq \End(W_1)$. Since $e_1,e_2 \in \cN_1$, (\ref{2points}) follows.

\smallskip
(a2) Assume now that $n = 2$, and so $p = q+1 \geq 17$. In this case, $\varphi$ is
real, and so $W$ is self-dual and supports a non-degenerate $H$-invariant
symmetric bilinear form $(\cdot,\cdot)$.  Write
$P=QT$ where $Q$ is elementary abelian of order $q$ and $T  \cong C_{q-1}$.
We also consider another parabolic subgroup
$P^\sharp = Q^\sharp T :=\Stab_H(\langle e_2 \rangle_{\bbF_q})$, with
$T = P \cap P^\sharp$. Letting $\rho$ denote the regular character of $T$ and
$\nu := \lambda_T$, we see that
\begin{equation}\label{for-tq}
  \varphi_T = \rho + \nu + \nu^{-1},~~
      \varphi_Q = 1_Q + \sum_{\alpha \in \Irr(Q)}\alpha.
\end{equation}
Next, using (\ref{for-tq}) one can see that
$W_P = \bfC_W( Q) \perp [W,Q]$, a direct orthogonal sum of two
$P$-submodules. Here, $C := \bfC_W(Q)$  is of dimension $2$ and
affords the $T$-character $\nu+\nu^{-1}$,  $[W,Q]$ is of dimension $q-1$ and
affords the $Q$-character $\sum_{1_Q \neq \alpha \in \Irr(Q)}\alpha$ and
the $T$-character $\rho$ (as $T$ permutes cyclically and transitively
the $q-1$ non-principal irreducible characters of $Q$). It also follows that these
two subspaces are non-degenerate and self-dual $P$-submodules.
Next, we can further decompose:
$$[W,Q] = A \perp B$$
as an orthogonal sum of two self-dual $T$-modules, where $A$ affords the $T$-character
$\rho-\nu-\nu^{-1}$, and $B$ affords the $T$-character $\nu + \nu^{-1}$.
Summarizing, we have that
$$W_P = A \perp B \perp C,$$
where $A \perp B$ is an irreducible $P$-module of dimension $q-1$, and $C$ is a
sum of two irreducible $P$-modules of dimension $1$. Applying the Artin-Wedderburn
theorem to $(P,W)$, we obtain
\begin{equation}\label{for-p1}
  \cM \supset \End(A \oplus B) := \{ f \in \End(W) \mid
        f(A \oplus B) \subseteq A \oplus B, f(C) = 0\}.
\end{equation}

Repeating the above argument for $P^\sharp$ instead of $P$,
we have that
$$W_{P^\sharp} = A^\sharp \perp B^\sharp \perp C^\sharp,$$
where $A^\sharp$ affords the $T$-character $\rho-\nu-\nu^{-1}$,
$B^\sharp$ affords the $T$-character $\nu+\nu^{-1}$, and
$C^\sharp = \bfC_W(Q^\sharp)$ affords the $T$-character $\nu+\nu^{-1}$, and
\begin{equation}\label{for-p2}
  \cM \supset \End(A^\sharp \oplus B^\sharp) :=
   \{ f \in \End(W) \mid
        f(A^\sharp \oplus B^\sharp) \subseteq A^\sharp \oplus B^\sharp, f(C^\sharp) = 0\}.
\end{equation}
Comparing with (\ref{for-tq}), we see that $A^\sharp = A$. Furthermore,
$C \cap C^\sharp$ is centralized by $\langle Q,Q^\sharp \rangle = H$, so
$C \cap C^\sharp = 0$. But both $C$ and $C^\sharp$ are of dimension $2$ and orthogonal to $A = A^\sharp$, whence $C \oplus C^\sharp = A^\perp$. Next,
$B \cap B^\sharp$ is a subspace of the non-degenerate subspace $A^\perp$
which is orthogonal to both $C$ and $C^\sharp$, so $B \cap B^\sharp = 0$. Since
$\dim B + \dim B^\sharp = 4 = \dim A^\perp$, we have shown that
$$A^\sharp = A,~~A^\perp = B \oplus B^\sharp = C \oplus C^\sharp,~~
    W = A \perp (B \oplus B^\sharp).$$
Suppose now that $f \in \End(W)$ belongs to both $\End(A \oplus B)$ and
$\End(A \oplus B^\sharp)$ as identified in (\ref{for-p1}) and (\ref{for-p2}).  Then
$f = 0$ on $C \oplus C^\sharp = A^\perp$, i.e. $f(A^\perp) = 0$. Next,
$$f(A) \subseteq (A \oplus B)  \cap (A \oplus B^\sharp) = A.$$
It follows that
$$\End(A \oplus B) \cap \End(A \oplus B^\sharp) \subseteq
    \End(A) := \{ f \in \End(W) \mid f(A) \subseteq A, f(A^\perp) = 0\},$$
and so by (\ref{for-p1}), (\ref{for-p2}) we have
$$\dim \cM \geq 2(q-1)^2 - (q-3)^2 = q^2+2q-7,$$
i.e. $\codim_{\End(W)}\cM \leq 8$.

On the other hand, all non-principal
$\psi \in \IBr_p(H)$ have degree $\geq q-1 \geq 15$. So, assuming
$\cM \neq \End(W)$, we see that all composition factors of the $H$-module
$\cQ := \End(W)/\cM$ are trivial. Since $H$ is perfect, it follows that
$H$ acts trivially on $\cQ$. But $\dim \Hom_{kH}(\End(W),1_H) = 1$, so
$\cM$ is contained in the unique submodule
$\cE_0 := \{ f \in \End(W) \mid \tr(f) = 0\}$ of codimension $1$ in $\End(W)$. But this
is a contradiction, since by (\ref{for-p1}), $\cM$ contains the map $g$ which acts as
identity on $A \oplus B$ and as $0$ on $C$, with $\tr(g) = q-1 = p-2 \neq 0$.

\smallskip
(b) Now we handle the case $p = q-1$ (so $2\mid q \geq 8$). Then for the unipotent radical
$Q$ of $P$ we have
$$W_Q = \bigoplus_{1_Q \neq \alpha \in \Irr(Q)}W_{\alpha}$$
with $W_\alpha$ affording the $Q$-character $\alpha$. Next, let
$S \cong C_{q+1}$ be a non-split torus in $H$. Then there is some
$1_S \neq \gamma \in \Irr(S)$ such that
$$W_S = \bigoplus_{\beta \in \Irr(S),~\beta \neq \gamma,\gamma^{-1}}W_\beta$$
with $W_\beta$ affording the $S$-character $\beta$. Thus both $Q$ and $S$ are multiplicity-free
on $W$. By Lemma \ref{mult-free}, we see
that for any composition factor $U$ of $\End(W)/\cM$, $U^Q = U^S = 0$.  On the
other hand, inspecting the (Brauer) character table of $H$ (see
\cite{Burk}), one sees that $U^Q \neq 0$ if $\dim U = 1,q,q+1$ and
$U^S \neq 0$ if $\dim U = q-1$, for any $U \in \IBr_p(H)$. Hence we conclude that
$\cM = \End(W)$.
\end{proof}

\begin{prop}\label{weil-su}
Let $k = \bar{k}$ and let $H = \SU_{n}(q)$ with $3 < p = (q^n+1)/(q+1)$ and $n \geq 3$.  Let $W$ be an
irreducible $kH$-module of dimension $p$. Then $(H,W)$ is weakly adequate.
\end{prop}

\begin{proof}
Let $\varphi$ denote the Brauer character of $W$ and let $\cN := \bbF_{q^2}^n$ denote the natural $\bbF_{q^2}H$-module. Recall that $H$ possesses the
so-called {\it reducible Weil character}
\begin{equation}\label{for-z}
  \zeta_{n,q}~:~g \mapsto (-1)^n (-q)^{\dim_{\bbF_{q^2}}\Ker(g-1)}
\end{equation}
for all $g \in H$, which decomposes as the sum of $q+1$ distinct {\it irreducible Weil characters}
$$\zeta_{n,q} = \sum^{q}_{i=0}\zeta^i_n$$
(of degree $(q^n-q)/(q+1)$ for $i = 0$ and $(q^n+1)/(q+1)$ for $i > 0$), see \cite{TZ2}. Then $\varphi$ can be obtained by restricting some $\zeta^j_n$ with
$j > 0$ to $p'$-elements of $H$.  We also let $P :=\Stab_H(U)$ for a
maximal totally singular subspace $U$ of $\cN$, with unipotent radical $Q$ (so
$P$ is a $p'$-group), and let $\rho$ denote the permutation character of $H$ acting on the set $\Omega$ of singular $1$-spaces of $\cN$ .

\smallskip
(a) First we show that

$\bullet$ the only irreducible constituents of
$\zeta^j_n\bar\zeta^j_n$ that are {\it not} of $p$-defect $0$ are $1_H$ and (possibly) another one, $\sigma$, of degree
\begin{equation}\label{for-s}
  \sigma(1) = \frac{(q^n-q)(q^n+q^2)}{(q^2-1)(q+1)};
\end{equation}

$\bullet$ all $p$-defect $0$ constituents of $\zeta^j_n\bar\zeta^j_n$ have degree $>2p$ if $n > 3$, and

$\bullet$ $\sigma$ is the Steinberg character $\St$ of $H$ if
$n = 3$ and it is a constituent of $\rho$ if $n > 3$.\\
Indeed,  (\ref{for-z})
implies that $(\zeta_{n,q})^2$ is just the permutation character of $H$ acting on
the point set of $\cN$, and at the same time it equals the restriction to $H$ of the reducible Weil character $\zeta_{2n,q}$ of $\SU_{2n}(q)$, if we embed
$H$ diagonally into $\SU_{2n}(q)$: $X \mapsto \diag(X,X)$. Assume $n > 3$. Then
all irreducible constituents of the latter restriction are described by
\cite[Proposition 6.3]{LBST}, and their degrees are listed in
\cite[Table III]{LBST}. It follows that $(\zeta_{n,q})^2$ has exactly two
non-$p$-defect $0$ irreducible constituents, namely $1_H$ (with
multiplicity $q+1$) and another one
$\sigma$ of indicated degree (with multiplicity $q$). Certainly, the permutation representation of $H$ on
(the point set of) $\cN$ contains the permutation representation of $G$ on
$\Omega$ as a subrepresentation (no matter if $n > 3$ or not). On the other hand,
$\rho$ contains an irreducible constituent of degree as listed in (\ref{for-s}) (see
\cite[Table 2]{ST}), so $\rho = 1_H + \sigma + \psi$ (and $\psi \in \Irr(H)$ has $p$-defect $0$). One also easily checks all defect $0$ constituents of $(\zeta_{n,q})^2$ have
degree $>2p$.

Suppose that $n=3$; in particular $3\nmid (q+1)$ and $q > 2$.
Inspecting the character table of $H = \SU_3(q)$ as given in
\cite{Geck}, we see that the only non-$p$-defect $0$ irreducible characters of
$H$ are $1_H$, the Weil character $\zeta^0_3$ of degree $q^2-q$, the Steinberg character $\St$ of degree $q^3$ matching (\ref{for-s}), and $(q^2-q)/3$ characters
$\chi^{(u)}_{(q+1)^2(q-1)}$. Direct calculations show that
$[\zeta^j_3\bar\zeta^j_3,\chi^{(u)}_{(q+1)^2(q-1)}]_H = 0$. Next, observe that
$$(\zeta_{3,q})^2 = 1_H + 1^H_Q + (q-1)1^H_L,$$
where $Q =\Stab_H(u)$ is the unipotent radical of $P$ as above (if
$U = \langle u \rangle_{\bbF_{q^2}}$) and $L =\Stab_H(v) \cong \SU_2(q)$ for
some non-singular $v \in \cN$. Furthermore,
$$[(\zeta^0_3)_Q,1_Q]_Q = 0,~~(\zeta^0_3)_L = \sum^{q}_{i=1}\zeta^i_2.$$
The former relation implies $[\zeta^0_3,1^H_Q]_H = 0$. On the other hand,
by \cite[Lemma 4.7(ii)]{TZ2}, each $\zeta^i_2$ in the latter relation is obtained
by restricting an irreducible character of degree $q-1 \geq 2$ of $\GU_2(q) \rhd L$ to $L$.
It follows by Clifford's theorem that $[\zeta^0_3,1^H_L]_H = 0$.  Thus we have
shown that $\zeta^0_3$ is {\it not} a constituent of $(\zeta_{3,q})^2$, as claimed.

\smallskip
(b) Now we show that, if $\beta$ is an irreducible constituent of $\varphi\bar\varphi$ and
$\beta \neq 1_H$,  then either $\beta(1) \geq 2p$, or $n = 3$, $\beta(1) = p$, and
$[\beta_Q,1_Q]_Q > 0$. Indeed, suppose that $\beta(1) < 2p$. Suppose for the moment
that $n > 3$. Then by the results of (a), $\beta$ is a constituent of the Brauer character
$\sigma^\circ$. But according to \cite{L1}, $\sigma^\circ-1_H \in \IBr_p(H)$, so
$\beta(1) = \sigma(1)-1 > 2p$ by (\ref{for-s}), a contradiction. Thus $n = 3$. If
moreover $\beta$ is in a block of $p$-defect $0$, then using \cite[Table 3.1]{Geck} we see
that $\beta(1) = p$, $\beta = (\zeta^i_3)^\circ$ for some $i > 0$ and so
$\beta_Q$ contains $1_Q$. Otherwise, by the results of (a), $\beta$ is a constituent
of $\St^\circ$. In this case, according to \cite[Theorem 4.2]{Geck},
$\St^\circ -1_H \in \IBr_p(H)$ and
so $\beta(1) = \St(1)-1 = q^3-1 > 2p$, again a contradiction.

\smallskip
(c) When $n \geq 5$, according to Lemmas  12.5 and 12.6 of \cite{GMST},
$\varphi_{\bfZ(Q)}$ contains a non-principal linear character $\lambda$, whose $P$-orbit $\cO$ has length $(q^{n-1}-1)/(q+1)$; moreover, any irreducible character of $Q$ above
$\lambda$ has degree $q$.
Since $\varphi(1) = (q^n+1)/(q+1)$, it follows
that $W_P = A \oplus B$, where $B := \bfC_W(\bfZ(Q))$ has dimension $1$,
$A := [W,\bfZ(Q)] \in \Irr(P)$ has dimension $(q^n-q)/(q+1)=p-1$ and affords the
$\bfZ(Q)$-character $q\sum_{\alpha \in \cO}\alpha$.
The same is also true for
$n = 3$, see Tables 2.1 and 3.1 of \cite{Geck}. Applying the Artin-Wedderburn theorem
to $(P,W)$, we see that
$$\cM \supseteq \End(A) \oplus \End(B).$$
In particular, if $\cM \neq \End(W)$, then any composition factor $X$ of the $H$-module
$\End(W)/\cM$ has dimension $\leq 2p-2$ and moreover
$X^Q \subseteq X^{\bfZ(Q)} = 0$. But this is impossible by the results of (b).
\end{proof}

\begin{lem}\label{non-serial}
Let $k = \bar{k}$ and let $W$ be an irreducible $kH$-module of dimension $p \geq 3$,
where $H$ is quasisimple and $(H,W)$ is one of the non-serial examples listed in
Tables I,IIa, IIb, or III. Then $(H,W)$ is weakly adequate.
\end{lem}

\begin{proof}
Let $\varphi$ denote the Brauer character of $W$.
Note that the cases $(H,p) = (\AAA_5,3)$, $(\AAA_6,5)$ are covered by Proposition \ref{weil-sp} since $\AAA_5 \cong \PSp_2(5)$ and $\AAA_6 \cong \PSL_2(9)$.

Suppose that $(H,p) = (\Sp_6(2),7)$. Then $H > L \cong \SL_2(8)$,
and $\varphi_L$ is irreducible (see \cite{JLPW}), so we are done by Proposition
\ref{weil-sl}.

Assume that $(H,p) = (M_{11},11)$. Then $H$ contains a $p'$-subgroup
$L = M_{10} \cong \AAA_6 \cdot 2_3$, and using \cite{GAP} we can check that
$\varphi_L = \lambda + \psi$, where $\lambda,\psi \in \Irr(L)$ are rational
of degree $1$ and $10$
(and $\lambda \neq 1_L$). It follows that $W_L = A \oplus B$, where $A$ affords the
character $\lambda$ and $B$ affords the character $\psi$. Applying the Artin-Wedderburn
theorem to $(L,W)$ we see that $\cM \supseteq \End(A) \oplus \End(B)$. In particular,
if $\cM \neq \End(M)$, then any composition factor $U$ of the $H$-module
$\End(W)/\cM$ has dimension $\leq 20$ and moreover all composition factors of
$U_L$ afford the character $\lambda\psi = \psi$. The latter condition also implies
that $\dim U = 10$ or $20$. On the other hand, using \cite{JLPW} and \cite{GAP} we see that any such $U$ must be of dimension $10$ and its character restricted to $L$ yields
an irreducible non-rational character, different from $\psi$. Hence $\cM = \End(W)$.

Assume that $(H,p) = (M_{12},11)$. Then $H$ contains a maximal subgroup
$L \cong \PSL_2(11)$, and using
\cite{GAP} we can check that $\varphi_L$ is irreducible. So we are done by \cite[Proposition 3.1]{GHT}.

Assume that $(H,p) = (M_{24},23)$. Then $H$ contains a maximal subgroup
$L \cong \PSL_2(23)$, and using
\cite{GAP} we can check that $\varphi_L$ is irreducible. So we are done by \cite[Proposition 3.1]{GHT}.

Assume that $(H,p) = (Co_{2},23)$ or $(Co_3,23)$. Then $H$ contains a $p'$-subgroup
$L \cong McL$, and using \cite{Atlas} we can check that $\varphi_L = 1_L + \psi$,
with $\psi \in \Irr(L)$. It follows that $W_L = A \oplus B$, where $A := \bfC_W(L)$ has
dimension $1$ and $B$ affords the character $\psi$. Applying the Artin-Wedderburn
theorem to $(L,W)$ we see that $\cM \supseteq \End(A) \oplus \End(B)$. In particular,
if $\cM \neq \End(M)$, then any composition factor $U$ of the $H$-module
$\End(W)/\cM$ has dimension $\leq 44$ and moreover
$U^L = 0$. On the other hand, using \cite{ModAt} we see that the only
irreducible $kH$-modules of dimension $\leq 44$ are $k$ and $W$, and both have
nonzero $L$-fixed points. Hence we conclude that $\cM = \End(W)$.
\end{proof}

Now we can prove  Theorem \ref{thm:degree-p} which we restate:

\begin{thm}\label{main-prime}
Let $k = \bar{k}$ be of characteristic $p$ and let $G$ be a finite group with a faithful
irreducible $kG$-module $V$ of dimension $p$. Then precisely one of the following
holds:  
\begin{enumerate}[\rm(i)]
\item  $G$ is adequate on $V$;
\item  $G$ contains an abelian normal subgroup $A$ of index $p$ (and
$G$ permutes $p$ one-dimensional summands of $V$ with kernel $A$); or
\item  $p=3$ and the image of $G$ in $\PGL(V)$ is $\PSL_2(9)$.
\end{enumerate}
\end{thm}

\begin{proof}  First assume that $p  > 3$. 
Apply Proposition \ref{degree-p} to $G$. In the case (i) of the proposition,
we are done by Proposition \ref{imp-p}. So we may assume that $G$ is almost
quasisimple, $H := G^{(\infty)}$ is quasisimple, with simple quotient $S$,
and $H$ is irreducible on $V$. If $S$ is of Lie type in
characteristic $p$, we can apply Proposition \ref{prime-natural2}. Assume we
are in the remaining cases. In all these cases, observe that the outer
automorphism group ${\mathrm {Out}}(S)$ is a $p'$-group and the Schur multiplier
${\mathrm {Mult}}(S)$ is a $p'$-group as well (as $p > 3$). The first condition implies
that $\GP = H$, whence by Remark \ref{g-gp} without loss we may assume $G = H$
and so $H^1(G,k) = 0$. The second condition implies that $H^2(G,k) = 0$. Furthermore,
in all cases $V$ lifts to a complex module of $p$-defect $0$, whence $V$ is
projective and so $\Ext^1_G(V,V) = 0$. Finally,
$H$ is weakly adequate on $V$ by
Propositions \ref{weil-sp}, \ref{weil-sl}, \ref{weil-su}, and Lemmas \ref{alt-p},
\ref{non-serial}.

Now consider the cases when $p=2$ or $3$.  If $V$ is imprimitive
the result follows as above.  So assume that $V$ is primitive.  Set $H=G^{(\infty)}$. 

Suppose that  $G=H=\SL_2(p^a) $
or $\PSL_2(p^a)$.  Then  $a \ge2$ and the result follows by
Corollary \ref{sl2q-adequate}.   If $G > H$ then
as $V^g \cong V$ as $H$-modules for all $g \in G$, $G/H$ is a $p'$-group
and $G$ is adequate on $V$ whenever $H$ is.   Thus the last case to consider
is $H=\PSL_2(9) \cong \AAA_6$.   The normalizer of $H$ in $\PGL(V)$ is
$\PGL_2(9)$ (the normalizer is just the subgroup of the automorphism group
which fixes the isomorphism class of $V$).    If the image of $G$ in $\PGL(V)$ is $\PGL_2(9)$,    $H^1(G,k)
=H^2(G,k)=0$ (see the proof of Corollary \ref{gl2q-adequate})
and since $\ext^1_G(V,V)=0$, $V$ is adequate in this case.

By  Proposition \ref{degree-p} and Theorem \ref{thm:2p}, the remaining 
cases to consider are   $G$ almost quasisimple, $p=3$  and 
$H \in  \{\AAA_5, \PSL_2(7), \SL_3(3^a), \SU_3(3^a)\}$.
In the first two cases,  the order of $G$ is not divisible by $9$, whence
$V$ is projective and so $\ext_G^1(V,V)=0$.   Note also that
$H^1(G,k)=H^2(G,k)=0$.
In the first case, 
 $V \otimes V^*$
is a direct sum of the projective cover of $k$ and a $3$-dimensional module.
Elements of order $5$ have nonzero trace and $3$-dimensional fixed space.
Since elements of order $5$ have only a $2$-dimensional fixed space on
the projective cover of $k$, it follows that the span of the elements of order $5$
generate $V \otimes V^*$.  In the second case, $V \otimes V^*$ is the projective
cover of $k$ and since the trace of an irreducible character cannot be identically
$0$, it follows that $H$ is weakly adequate on $V$ in this case as well.   Thus,
$(G,V)$ is adequate.  
   
   In the last two cases,  weak adequacy follows from the fact
   that $V \otimes V^*$  is a uniserial module with trivial socle
   and head.  It follows
   by the main result of \cite{Mc}  that $\ext^1_G(V,V)=0$ for $a > 2$.  
 One computes directly 
that $\ext_G^1 (V,V)=0$ in all other cases (Klaus Lux did the computation;  also see \cite{JP} for the case of
$\SL_3(3^a)$).  Since $H^1(G,k)=H^2(G,k)=0$, the result follows. 
\end{proof}

\section{Certain PIMs for simple groups}\label{sec: pim}

For a finite group $X$ and a fixed prime $p$, let $B_0(X)$ denote the principal
$p$-block of $X$. We will sometimes use the same notation for an irreducible
$kX$-module and its Brauer character.

First we describe the submodule structure of the PIMs for some
non-projective modular representations of
simple groups $H$ described in Theorems \ref{bz} and \ref{thm:2p}.

\medskip
Assume that $H$ has a Sylow $p$-subgroup $P$ of order $p$ and furthermore that
$P = \bfC_H(P)$. In this case,
$P$ has a unique block $b$ with defect group $P$ and canonical character $1_P$,
see \cite[Theorem 4.6.12]{LP}. According to
Brauer's theorem \cite[Theorem 4.12.1]{LP}, $H$ has a unique $p$-block
$B$ of defect $d > 0$ (hence $d = 1$), and $B = b^G$. In particular, $B = B_0(H)$.
Note that the number of exceptional characters in $B$ equals $(p-1)/|\bfN_H(P)/P|$
in this situation, and all of them are $p$-conjugate (and so non-rational if
$|\bfN_H(P)/P| < p-1$), cf. Theorem 4.12.1 and
Corollary 4.12.2 of \cite{LP}.
We will use \cite[Theorem 4.12.1]{LP} to find PIMs $\cP(\varphi)$ for some
$\varphi \in \IBr_p(B)$.

\subsection{The case \texorpdfstring{$H = \PSL_n(q)$}{H = PSL\_n(q)} with \texorpdfstring{$p = (q^n-1)/(q-1)$}{p = (q\^{}n-1)/(q-1)} and \texorpdfstring{$n \geq 2$}{n >= 2}}
\label{pim-sl}
First suppose that $n \geq 3$. Then
$B$ contains unipotent characters $\chi_0 = 1_H$, $\chi_1$, and $\chi_2$ labeled by the partitions $(n)$, $(n-1,1)$, and $(n-2,1^2)$, and Brauer characters
$\varphi_0 = 1_H$, $\varphi_1$ of degree $p-2$ (afforded by $\cD$), and
$\varphi_2$ of degree
$$\frac{(q^n-2q^2+1)(q^n-q)}{(q^2-1)(q-1)}+1$$
(afforded by a $kH$-module say $\cU$)
among others (see e.g. \cite[Proposition 3.1]{GT1}). More precisely,
\begin{equation}\label{dec-sl}
  \chi_0^\circ = \varphi_0,~~\chi_1^\circ = \varphi_0+\varphi_1,~~
  \chi_2^\circ = \varphi_1 + \varphi_2.
\end{equation}
Note that
$$\dim(\cU) = \varphi_2(1) > 2p$$
unless $(n,q) = (3, \leq 3)$.
Since $\chi_i$ are rational for $i = 0,1,2$, they are all non-exceptional. Next, the
character of any PIM in $B$ is of the form $\psi_1 + \psi_2$, where $\psi_1 \in \Irr(B)$
is non-exceptional, and either $\psi_2 \in \Irr(B)$, or $\psi_2$ is the sum of
all $(p-1)/n$ exceptional characters in $B$. Hence the relations (\ref{dec-sl}) show
that

$\bullet$ $\cP(\varphi_0)$ affords the character
$(\chi_0+\chi_1)^\circ = 2\varphi_0 + \varphi_1$. In fact, one can see that
$\cP(k)$ is the uniserial module $(k|\cD|k)$. In particular, this shows that
$\Ext^1_H(k,\cU) = \Ext^1_H(\cU,k) = 0$;

$\bullet$ $\cP(\varphi_1)$ affords the character
$(\chi_1+\chi_2)^\circ = \varphi_0 + 2\varphi_1+\varphi_2$.
By \cite[Corollary 4.12.5]{LP},  the module $\cP(\varphi_1) = \cP(\cD)$
has socle series $(\cD|k \oplus \cU|\cD)$. Since $\varphi_1$ is real, $\cP(\cD)$
is self-dual. Furthermore, the only nonzero proper submodules of $\cP(\cD)$ are
$$\cD = \soc(\cP(\cD)),~~(\cD|k),~~(\cD|\cU),~~
    (\cD|k \oplus \cU) = \rad(\cP(\cD)),$$
(cf. \cite[Figure 4.3]{LP}), and so none is of the form $(\cD|k|\cD)$.

\smallskip
Next, let $H = \SL_2(q)$ and $p = q+1$. The decomposition numbers for $B_0(H)$
are given in \cite{Burk}. In particular, $\Irr(B)$ consists of $\chi_0 = 1_H$,
$\chi_1 = \St$ of degree $q$, and $q/2$ exceptional characters $\theta_i$,
$1 \leq i \leq q/2$, of degree $q-1$, and $\IBr_p(B) = \{1_H,\varphi_1\}$,
with $\varphi_1$ afforded by $\cD$.
So as before, we have that $\cP(k) = (k|\cD|k)$ is uniserial. Next,
$\cP(\cD) = \cP(\varphi_1)$ affords the character
$$(\St + \sum^{q/2}_{i=1}\theta_i)^\circ = 1_H + (\frac{q}{2}+1)\varphi_1.$$
If we let $\cD_j = (\cD|\cD| \ldots |\cD)$ denote a uniserial module with Brauer character
$j\varphi_1$,
then $\cP(\cD) = (\cD|k \oplus \cD_{(q-2)/2}|\cD)$ (cf. \cite[Corollary 4.12.5]{LP}).
Furthermore, the only nonzero proper submodules of $\cP(\cD)$ are
$\cD_j$ or $(\cD|k \oplus \cD_{j-1})$ with $1 \leq j \leq q/2$.

\subsection{The case \texorpdfstring{$H = \PSU_n(q)$}{H = PSU\_n(q)} with \texorpdfstring{$p = (q^n+1)/(q+1)$}{p = (q\^{}n+1)/(q+1)}, \texorpdfstring{$n \geq 3$}{n >= 3},
\texorpdfstring{$(n,q) \neq (3,2)$}{(n,q) <> (3,2)}}\label{pim-su}
Note that in this case $\tilde{H} \cong \bfZ(\tilde{H}) \times H$ for
$\tilde{H} := \GU_n(q)$; moreover, the unipotent characters of $\tilde{H}$ as well as
the characters in $B_0(\tilde{H})$ are all trivial at $\bfZ(\tilde{H})$.
Hence without loss we may assume $H = \GU_n(q)$. We will consider the unipotent
characters $\chi_0 = 1_H$, $\chi_{1,2,3}$, labeled by the partitions
$(n)$, $(n-1,1)$, $(n-2,1^2)$, and $(n-3,1^3)$ (the latter being considered only
when $n > 3$). Using the description of Brauer trees for $H$ given in \cite[\S6]{FS}, we see that, when $n \geq 5$, there exist  Brauer characters
$\varphi_0 = 1_H$, $\varphi_{1,2,3}$, such that
\begin{equation}\label{dec-su}
  \chi_0^\circ = \varphi_0,~~\chi_1^\circ = \varphi_1,~~
  \chi_2^\circ = \varphi_0 + \varphi_2,~~\chi_3^\circ = \varphi_1 + \varphi_3.
\end{equation}
In particular, $\varphi_0 = 1_H$, $\varphi_1$ is a Weil character of degree $p-1$,
$$\varphi_2(1) = p\frac{q^n+q^2-q-1}{q^2-1}-2 > 8p,~~
    \varphi_3(1) = p\frac{(q^n-q)(q^n-q^3+q^2-1)}{(q^2-1)(q^3-1)}-2p+2 > 28p.$$
Since $\chi_i$ are all rational, they are all non-exceptional, so as in
\S\ref{pim-sl}, the relations (\ref{dec-su}) and \cite[Corollary 4.12.5]{LP} show that both
$\cP(\varphi_{0,1})$ are uniserial:
$$\cP(\varphi_0) = (\varphi_0|\varphi_2|\varphi_0),
    ~~\cP(\varphi_1) = (\varphi_1|\varphi_3|\varphi_1),$$
where we have used the same notation for the module and its Brauer character.

Suppose now that $n=3$ and $q \geq 3$. Then we still have
$\cP(\varphi_0) = (\varphi_0|\varphi_2|\varphi_0)$ is uniserial for $\varphi_0 = 1_H$
and $\varphi_2(1) > 2p$. For the Weil character $\varphi_1$ of degree $p-1$,
now $\cP(\varphi_1)$ affords the character
$$(\chi_1+\sum^{(p-1)/3}_{i=1}\theta_i)^\circ =
    \frac{p+2}{3}\varphi_1 + \frac{p-1}{3}\varphi_2,$$
where $\theta_i$ are exceptional characters in $B$,
of degree $(q^2-1)(q+1) > 2p$, for $1 \leq i \leq (p-1)/3$. We claim that
$$\cP(\varphi_1) = (\varphi_1|\varphi_2|\varphi_1| \ldots |\varphi_2|\varphi_1)$$
is self-dual, uniserial of length $(2p+1)/3$, with the composition factors $\varphi_1$
and $\varphi_2$ alternating. The first statement follows since $\varphi_1$ is real. The second
one holds because in this case, both the socle $\soc(\cP(\varphi_1))$ and head
$\cP(\varphi_1)/\rad(\cP(\varphi_1))$ are simple and
$\rad(\cP(\varphi_1))/\soc(\cP(\varphi_1))$ is uniserial. By \cite[Theorem 1.1]{TZ1},
$H$ has a unique complex character of degree equal to $\varphi_0(1)$ or $\varphi_1(1)$. Hence the last statement holds by Lemma \ref{cyclic} since $\varphi_{0,1}$ each has a unique lift to characteristic $0$.

\subsection{The case \texorpdfstring{$H = \SL_2(q)$}{H = SL\_2(q)} and \texorpdfstring{$p = q-1 \geq 3$}{p = q-1 >= 3}}\label{pim-sl2}
The decomposition numbers for $B_0(H)$ are given in \cite{Burk}. In particular, $\Irr(B)$ consists of $\chi_0 = 1_H$,
$\chi_1 = \St$ of degree $q$, and $(q-2)/2$ exceptional characters $\theta_i$,
$1 \leq i \leq q/2$, of degree $q+1$, and $\IBr_p(B) = \{\varphi_0=1_H,\varphi_1\}$,
with $\varphi_1 = \St^\circ$ afforded by $\cD$. Clearly, both
$\varphi_{0,1}$ have a unique complex lift, so by Lemma \ref{cyclic} they have
no self-extensions. Arguing as in \S\ref{pim-su} and using
\cite[Corollary 4.12.5]{LP}, we see that both
$$\cP(\varphi_0) = (\varphi_0|\varphi_1|\varphi_0| \ldots |\varphi_1|\varphi_0),~~
    \cP(\varphi_1) = (\varphi_1|\varphi_0|\varphi_1| \ldots |\varphi_0|\varphi_1)$$
are self-dual and uniserial of length $p$, and with the composition factors $\varphi_0$
and $\varphi_1$ alternating.

\subsection{The case \texorpdfstring{$H = \AAA_p$}{H = A\_p} with \texorpdfstring{$p \geq 7$}{p >= 7}}\label{pim-an}
Consider the irreducible complex characters $\chi_{0,1,2}$ of $\SSS_p$ labeled by
$(p)$, $(p-1,1)$, and $(p-2,1^2)$.
By Peel's theorem \cite{P},
\begin{equation}\label{dec-an}
  \chi_0^\circ = \varphi_0,~~\chi_1^\circ = \varphi_0+\varphi_1,~~
    \chi_2^\circ = \varphi_1 + \varphi_2,
\end{equation}
where $\varphi_{0,1,2} \in \IBr_p(\SSS_p)$, $\varphi_0(1) = 1$,
$\varphi_1(1) = p-2$, $\varphi_2(1) = (p-2)(p-3)/2$.  It is well known that
$\chi_{0,1,2}$ and $\varphi_{0,1}$ are all irreducible over $H$. Restricting to
$\SSS_{p-1}$, it is easy to see that $\varphi_2$ is irreducible over $H$ as well.
We will use the same notation for the restrictions of these characters to $H$. Since
$\chi_{0,1,2}$ are rational, they are non-exceptional in $B_0(H)$. Hence,
(\ref{dec-an}) and  \cite[Theorem 4.12.1]{LP} imply that
$\cP(\varphi_0) = (\varphi_0|\varphi_1|\varphi_0)$ is uniserial and that
$\cP(\varphi_1)$ has socle series $(\varphi_1|\varphi_0 \oplus \varphi_2|\varphi_1)$.
In particular, there is no $kH$-module of the form $(\varphi_1|\varphi_0|\varphi_1)$.


\section{Indecomposable modules of dimension less than $2p-2$}  \label{sec:gcr}
First we record a simple observation

\begin{lem}\label{single}
Let $b \in \bbN$ and let $V$ be a $kG$-module of dimension $\leq b$ with a $\GP$-composition factor $U$. Suppose that any quotient of length $2$ of $\cP(U)$ or
$\cP(U^*)$ has dimension $>b$. Then $U$ is a direct summand of the
$\GP$-module $V$. If
moreover $U$ has multiplicity $1$, then the $G$-module $V$ is either irreducible
of dimension $\dim U$, or decomposable.
\end{lem}

\begin{proof}
Suppose that $W$ is an indecomposable subquotient of length $2$ of the
$\GP$-module $V$ with $U$ as a composition factor. Replacing $W$ by $W^*$ if
necessary, we may assume that $\hd(W) \cong U$ and so $W$ is a quotient of
$\cP(U)$. But then by the hypothesis, $\dim W > b \geq \dim V$, a contradiction. So
$U$ is a direct summand of the $\GP$-module $V$ by Lemma \ref{semi1}(i):
$V_\GP = U_1 \oplus M$ with $U_1 \cong U$. The last claim now follows from
Lemma \ref{semi2}(ii).
\end{proof}

Given a nontrivial $U \in \IBr_p(X)$, we call any $kX$-module $V$ {\it $U$-special}, if
$V$ has $U$ or $U^*$ as composition factors of {\it total} multiplicity $\leq 1$, and moreover all other composition factors of $V$ are trivial.

\begin{lemma}\label{big-cf1}
Let $N = \bfO^p(N)$, $b \in \bbN$, and let $U \in \IBr_p(N)$ be a nontrivial module. Suppose that the only $U$-special quotients of $\cP(k)$, $\cP(U)$, and
$\cP(U^*)$ of dimension $\leq b$ are uniserial modules in the list
$$\cX := \{k,~Y,~(k|Y),~(Y|k),~(k|Y|k) \mid Y = U \mbox{ or }U^*\}.$$
Let $V$ be any $U$-special $kN$-module of dimension $\leq b$. Then $V \cong \oplus^m_{i=1}X_i$
for some $X_i \in \cX$.
\end{lemma}

\begin{proof}
We induct on the length of $V$. Suppose $V$ has length $\geq 2$. If all composition
factors of $V$ are trivial, then $N$ acts trivially on $V$ since $N = \bfO^p(N)$, and
we are done. Replacing $V$ by $V^*$ if necessary, we may assume that $V$ has $U$ as a composition factor of multiplicity $1$, and all other composition factors of $V$ are $k$.


Suppose that $U$ embeds in $\hd(V) := V/R$, where $R := \rad(V)$. Then all composition factors of
$R$ are trivial, and so $N$ acts trivially on $R$. Assume in addition that
$V/R$ is not simple. Then $V/R = M/R \oplus Y/R$ for some submodules
$M,Y$ with $Y/R \cong k$. Again, $N$ acts trivially on $Y$, so we can write
$Y = R \oplus Z$ for some submodule $Z \cong k$. It follows that $V = M \oplus Z$,
and we are done by induction. Assume now that $V/R \cong U$. Then
the surjection $\cP(U) \to V/R$ lifts to a surjection $\cP(U) \to V$. Since
$\dim V \leq b$, we must then have that $V \in \cX$.

The case $U \hookrightarrow \soc(V)$ now follows from the previous case by
duality.

Now we may assume that $U$ embeds neither in $\hd(V)$ nor in $\soc(V)$. Letting
$W := [N,V]$ and $T := \rad(W)$, we see that $W$ has no trivial quotient.
But $W/T$ is semisimple, so
$W/T \cong U$. Applying the induction hypothesis to $V/T$ and noting
that $U$ is in the socle but not in the head of $V/T$, we see that
$V/T \cong L/T \oplus Y/T$, where $L/T \cong (U|k)$ and $N$ acts trivially on $Y/T$.
In this case, $N$ acts trivially on $Y$ as well. If moreover $Y \neq T$, then
we can decompose $Y = T \oplus Z$ for some submodule $Z \neq 0$, whence
$V = L \oplus Z$ and we are done by induction. Thus we may assume
$V/T \cong (U|k)$. Consider any maximal submodule $M$ of $V$. Since
$U \not\hookrightarrow\hd(V)$, $V/M \cong k$ and so $M \supseteq W$. It
follows that $R \supseteq T$ and $R/T = \rad(V/T) \cong U$. Hence
$V/R \cong k$. In this final case, the surjection $\cP(k) \to V/R$ lifts to a surjection
$\cP(k) \to V$. Since $\dim V \leq b$, we must again have that $V \in \cX$.
\end{proof}

\begin{cor}\label{big-cf2}
Let $\GP$ be perfect, $b \in \bbN$, and let $U \in \IBr_p(\GP)$ be a
nontrivial module. Suppose that the only $U$-special quotients of
dimension $\leq b$ of $\cP(k)$, $\cP(U)$, and $\cP(U^*)$ are uniserial modules in the list
$$\cX := \{k,~Y,~(k|Y),~(Y|k),~(k|Y|k) \mid Y = U \mbox{ or }U^*\}.$$
Let $V$ be any indecomposable $kG$-module of dimension $\leq b$ such that $V_\GP$ is
$U$-special. Then $V_\GP$ is also indecomposable and belongs to $\cX$.
\end{cor}

\begin{proof}
We may assume that exactly one indecomposable direct summand $A$ of
$V_\GP$ has $U$ as a composition factor, and so $V_\GP = A \oplus B$ with
$\GP$ acting trivially on $B$.
Hence $B = 0$ by Lemma \ref{semi2}(ii).
\end{proof}

\begin{thm}\label{indec-qs}
Let $G$ be a finite group, $k$ an algebraically closed field of characteristic $p$,
$\bfO_p(G) = 1$,
and let $V$ be a faithful, indecomposable $kG$-module of dimension
less than $2p-2$. Assume in addition that $\GP$ is quasisimple but not of Lie-type
in characteristic $p$. Then one of the following statements holds, where
$U, W \in \IBr_p(\GP)$.

\begin{enumerate}[\rm(i)]

\item $V$ is irreducible.

\item $(\GP,p,\dim U) = (\SL_2(q),q-1,p+1)$, $(\AAA_p,p,p-2)$, $(\SL_n(q),(q^n-1)/(q-1),p-2)$,
$(M_{11},11,9)$, $(M_{23},23,21)$. Furthermore, $V_\GP$ is uniserial of the form $(k|U)$, $(U|k)$, or $(k|U|k)$, and $U \cong U^*$.

\item $(\GP,p,\dim U) = (\SL_2(q),q+1,p-2)$, $U \cong U^*$, and
$V_\GP$ is indecomposable of the form $(U|U)$, $(U|k \oplus U)$, or $(k \oplus U|U)$.

\item $(\GP,p,\dim U) = (2\AAA_7,7,4)$, $V_\GP = (U|U)$ is uniserial, and $U \cong U^*$.

\item $(\GP,p) = (M_{11},11)$ and $V_\GP = (U|W)$ is uniserial,
$\{\dim U,\dim W\} = \{9,10\}$.

\item $(\GP,p,\dim U) = (3\AAA_6,5,3)$, $V_\GP = (U|U)$ is uniserial, and
$U \not\cong U^*$.

\item $(\GP,p,\dim U) = (\tw2 \btwo(8),13,14)$, $V_\GP$ is uniserial of the form $(k|U)$ or
$(U^*|k)$ for a fixed $U \not\cong U^*$.
\end{enumerate}
\end{thm}

\begin{proof}
(a) Note that the statement is vacuous for $p = 2$.
Throughout the proof, we assume that $p > 2$, $V$ is reducible, and let $U$ be a
composition factor of the $\GP$-module $V$ of largest dimension. Also set
$b := 2p-3$ whenever we apply Lemma \ref{big-cf1} and Corollary \ref{big-cf2}.
Note that $V_\GP$ is  (reducible) indecomposable by Corollary \ref{indec-gp}. Next, $\GP$ must
act irreducibly and non-trivially on some subquotient $X$ of $V_\GP$. Applying Theorems
\ref{bz} and \ref{thm:2p} to the action of $\GP$ on $X$, we see that
${\mathrm {Mult}}(\GP/\bfZ(\GP))$, and so
$\bfZ(\GP)$, has $p'$-order. The indecomposability of $V_\GP$ then implies that
$\bfZ(\GP)$ acts via scalars on $V$ and that $\GP$ acts faithfully on $U$ (and so we may
identify $\GP$ with its image in $\GL(U)$).
In particular, if $k$ is a composition factor of
$V_\GP$, then $\GP$ is simple. This must be the case if $\dl(\GP) \geq p-1$.
Also, if $\bfZ(\GP) \neq 1$, then $\GP$ acts faithfully on every composition factor of $V_\GP$.

\smallskip
(b)  Assume first that $(\GP,p) = (J_1,11)$.  According to \cite{JLPW}, the only $\varphi \in \IBr_p(\GP)$
of degree $< 2p$
are $\varphi_{1,7,14}$. Here we use the notation $\varphi_{j}$ to denote the unique
$\varphi \in \IBr_p(\GP)$ of degree $j$. Moreover, using \cite{ModAt} we see that
\begin{equation}\label{pim-j1}
  \cP(\varphi_1) = (\varphi_1|\varphi_{119}|\varphi_1),~
    \cP(\varphi_7) = (\varphi_7|\varphi_{49} \oplus \varphi_{69}|\varphi_{7}),~
    \cP(\varphi_{14}) = (\varphi_{14}|\varphi_{106}\oplus\varphi_{119}|\varphi_{14}).
\end{equation}
Since $\dim V < 2p$, each composition factor $X$ of the $\GP$-module $V$
must afford the Brauer character $\varphi_i$ for some $i \in \{1,7,14\}$. Now
(\ref{pim-j1}) shows that $\Ext^1_\GP(X,Y) = 0$ for any two such composition
factors $X$ and $Y$. Hence the $\GP$-module $V$ is semisimple
by Lemma \ref{semi1}, a contradiction.

From now one we may assume that $\GP \not\cong J_1$ and so $\dl(\GP) \geq p-3$ by
Theorem \ref{bz}.
In particular, $\dim U \geq p-3$ and Corollary \ref{block3} applies.

\smallskip
(c) Here we consider the case where $\dim U > p$. Since $\dim V \leq 2p-3$ and
$V_\GP$ is reducible, it follows that $k$ is a composition factor of $V_\GP$, and so
$\GP$ is simple as noted in (a).  Also, all composition factors of $V_\GP$ other than $U$
are trivial.
Now we apply Theorem \ref{thm:2p} to $(\GP,U)$.

Suppose that $\GP = \AAA_n$ with $n \geq p$ as in the first row of Table I. Since
$p+1 \leq \dim U \leq 2p-3$, we have that $5 \leq p \nmid n$ and $p+2 \leq n \leq 2p-2$.
By \cite[Lemma 6.10]{NT2}, $H^1(\AAA_n,U) = 0$, whence
$\Ext^1_\GP(k,U) = \Ext^1_\GP(U,k) = 0$.  Also,
$\Ext^1_\GP(k,k) = \Ext^1_\GP(U,U) = 0$ by Lemma \ref{cyclic}.
It follows by Lemma \ref{semi1} that
$V_\GP$ is semisimple, a contradiction.

Next suppose that $(\GP,p,\dim U) = (\SL_2(q),q-1,p+1)$ as in Table IIa.  Then,
as shown in \S\ref{pim-sl2}, $(\GP,U)$ satisfies the hypothesis of Corollary \ref{big-cf2},
and so we arrive at (ii).

In the cases where $(\GP,p,\dim U) = (\AAA_7,7,10)$,  $(\SL_3(3),13,16)$,
$(\SU_4(2),5,6)$, $(\Sp_4(4),17,18)$, $(\gtwo(3),13,14)$, $(J_1,11,14)$,
$(J_1,19, 22 \mbox{ or }34)$, $(M_{12},11,16)$, or $(M_{11},11,16)$,
using the information on decomposition numbers given in \cite{ModAt}, one can check
that $U$ satisfies the hypothesis of Lemma \ref{single}, and so $V$ is decomposable,
a contradiction.

Assume that $(\GP,p,\dim U) = (\tw2 \btwo(8), 13,14)$. Then $V$ is $U$-special, and
using \cite{ModAt} one can check that the only quotients of dimension $\leq 23$ of
$\cP(k)$, $\cP(U)$, and $\cP(U^*)$ are $k$, $Y$, $(k|Y)$, or $(Y|k)$, with
$Y = U$ or $U^*$. Applying Corollary \ref{big-cf2}, we arrive at (vii).

\smallskip
(d) Next we consider the case $\dim U = p$, and apply Theorem \ref{thm:2p} to
$(\GP,U)$.
By Lemma \ref{p-div}, $U$ is projective, and so it is a direct summand of $V_\GP$,
a contradiction.

\smallskip
(e) Assume now that $\dim U = p-1$, and apply Theorem \ref{bz} to $(\GP,U)$.
Note that $U$ has multiplicity $1$ as $\dim V < 2p-2$. First
we consider the case $(\GP,p) = (\SU_n(q),(q^n+1)/(q+1))$. In this case,
$\GP$ is simple,
$\dl(\GP) = p-1$ and so all other composition factors of the $\GP$-module
$V$ are trivial.  As shown in \S\ref{pim-su},
$$\Ext^1_\GP(k,U) = \Ext^1_\GP(U,k) = \Ext^1_\GP(k,k) = \Ext^1_\GP(U,U) = 0.$$
It follows by Lemma \ref{semi1} that $V_\GP$ is semisimple, a contradiction.

Suppose now that $(\GP,p) = (\Sp_{2n}(q),(q^n+1)/2)$, $(2Ru,29,28)$,  $(3J_3,19,18)$,
$(2\AAA_7,5,4)$, $(3\AAA_7,7,6)$, $(6\AAA_7,7,6)$, $(2J_2,7,6)$,
$(6_1 \cdot \PSU_4(3),7,6)$, $(6\cdot \PSL_3(4),7, 6)$, $(2M_{12},11,10)$,
$(2M_{22},11,10)$, $(6Suz,13,12)$, or $(2\gtwo(4),13,12)$. Since $\bfZ(\GP) \neq 1$,
$\GP$ acts faithfully on every composition factor $X$ of $V_\GP$ as noted in (a), whence
$\dim X \geq p-1 > (\dim V)/2$ by Theorem \ref{bz}. It follows that $V_\GP$ is irreducible, a
contradiction.

Assume that $(\GP,p,\dim U) = (M_{11},11,10)$. The Brauer
tree of $B_0(\GP)$ is given in \cite[Example 4.12.11]{LP}. Using this information,
we see that the only quotient of length $2$ of dimension $\leq 19$ of $\cP(U)$ is
of form $(W|U)$, where $W \in \IBr_p(M_{11})$ has dimension $9$. Arguing as in
the proof of Lemma \ref{single}, we arrive at (v).

\smallskip
(f) Next we consider the case $\dim U = p-2$ and apply Theorem \ref{bz} to $(\GP,U)$.
We can exclude the subcase $(\GP,p) = (\SL_3(2) \cong \PSL_2(7),7)$.

\smallskip
(f1) First we assume that $(\GP,p) = (\SL_n(q),(q^n-1)/(q-1))$ or $(\AAA_p,p)$, and
moreover $U$ is a composition factor of $V$ of multiplicity $2$. Since
$\dim V \leq 2p-3$, we have two cases.

$\bullet$ $\dim V = 2p-3$ and so $k$ is also a composition factor of the
$\GP$-module $V$. Suppose that $\hd(V_\GP)$ is not simple. Then
$V_\GP$ contains two maximal submodules $A, B$ of length $2$ and
$A \cap B \subseteq \soc(V_\GP)$. On the other hand, the indecomposability of
$V_\GP$ implies that $\soc(V_\GP) \subseteq \rad(V_\GP) \subseteq A \cap B$, whence
$\soc(V_\GP) = A \cap B$ is simple. So up to duality, we may assume that
$\hd(V_\GP)$ is simple. It follows that
$V_\GP$ is a quotient of $\cP(U)$ or $\cP(k)$. The structure of PIMs described in \S\S\ref{pim-sl}, \ref{pim-an} shows that (iii) holds.

$\bullet$ $\dim V = 2p-4$ and so $V_\GP$ has exactly two composition factors,
both isomorphic to $U$. As $V_\GP$ is indecomposable,
it is a quotient of $\cP(U)$. Using the results of \S\S\ref{pim-sl}, \ref{pim-an},
we again arrive at (iii).

\smallskip
(f2) Now we assume that $(\GP,p) = (\SL_n(q),(q^n-1)/(q-1))$ or $(\AAA_p,p)$, and
moreover $U$ is a composition factor of $V$ of multiplicity $1$. By Theorem
\ref{bz}, $U$ is the only nontrivial irreducible $k\GP$-module of dimension
$\leq p-1$. Since $\dim V \leq 2p-3$, it follows that all other composition factors
of $V_\GP$ are trivial, i.e. $V_\GP$ is $U$-special.
The structure of PIMs described in
\S\S\ref{pim-sl}, \ref{pim-an} shows that the only $U$-special quotients of dimension
at most $b = 2p-3$ of $\cP(U)$ and $\cP(k)$ all belong to
$\{ k, U, (U|k), (k|U), (k|U|k) \}$. Also, $U \cong U^*$.  Hence we arrive at (ii) by
Corollary \ref{big-cf2}.

\smallskip
(f3) Assume that $(\GP,p,\dim U) = (M_{23},23,21)$. Then $U \cong U^*$,
$\cP(k) = (k|U|k)$, and the only quotient of length $2$ of dimension $\leq 43$ of $\cP(U)$ is
$(k|U)$.  Arguing as in the case of $\AAA_p$ in (e1) and (e2), we arrive at (ii).

Consider the case $(\GP,p,\dim U) = (M_{11},11,9)$. Then
$U \cong U^*$ and $\cP(k) = (k|U|k)$.
Using \cite[Example 4.12.11]{LP} as above, we see that $\cP(U)$ has only two non-simple quotients of dimension $\leq 19$, namely $(k|U)$ and $(W|U)$ with $W \in \IBr_p(\GP)$ of dimension $10$.
Arguing as above we arrive at (ii).

Suppose now that $(\GP,p,\dim U) = (3\AAA_7,5,3)$. Recall that $U$ is a composition factor of
largest dimension of $V_\GP$ and $\bfZ(\GP)$ acts via scalars on $V$. Using \cite{JLPW} one can
then check that all composition factors of $V_\GP$ are isomorphic to $U$. But $\Ext^1_\GP(U,U) = 0$
by Lemma \ref{cyclic}. Hence $V_\GP$ is semisimple by Lemma \ref{semi1}(ii), a contradiction.

Suppose that $(\GP,p,\dim U) = (3\AAA_6,5,3)$. As in the case of $(3\AAA_7,5,3)$,
we see that all composition factors of $V_\GP$ are isomorphic to $U$. But $\dim V \leq 7$ and
$V_\GP$ is indecomposable, so $\hd(V_\GP) \cong U$. Inspecting the structure of
$\cP(U)$ using \cite{ModAt}, we conclude that $V_\GP \cong (U|U)$ is uniserial, i.e (vi) holds.

\smallskip
(g) Finally, let $\dim U = p-3$. By Theorem \ref{bz}, we have that
$(\GP,p) = (2\AAA_7,7)$. As in the case of $(3\AAA_7,5,3)$,
we see that all composition factors of $V_\GP$ are isomorphic to $U$. But $\dim V \leq 11$ and
$V_\GP$ is indecomposable, so $\hd(V_\GP) \cong U$. Note that $\cP(U) = (U|U \oplus W|U)$, where
$W \in \IBr_p(\GP)$ has dimension $16$ (as one can see using \cite{ModAt}). It follows that
$V_\GP \cong (U|U)$, the unique quotient of dimension $8$ of $\cP(U)$, and
we arrive at (iv).
\end{proof}

\begin{lem}\label{indec3}
Suppose that $p = 3$ and $V$ is a reducible, faithful, indecomposable
$kG$-module of dimension $\leq 2p-3$. Then $\bfO_p(G) \neq 1$.
 \end{lem}

\begin{proof}
Suppose first that every $\GP$-composition factor of $V$ is of dimension $1$ and
so $\cong k$ (as $\GP = \bfO^{p'}(\GP)$). By faithfulness, $\GP$ is a $p$-group; moreover $\GP \neq 1$ as otherwise $G$ is a $p'$-group. Thus
$1 \neq \GP = \bfO_p(G)$.

Since $V_\GP$ is reducible, it remains to consider the case where $V_\GP$ has
exactly two composition factors, $U$ of dimension $2$ and $W$ of dimension $1$,
and moreover $\bfO_p(G) = 1$.
Let $K$ denote the kernel of the action of $\GP$ on $U$. Again,
$\GP = \bfO^{p'}(\GP)$ acts trivially on $W$. It follows by faithfulness of $G$ on $V$
that $K \leq \bfO_p(\GP) \leq \bfO_p(G) = 1$.  Next, since
$\GP = \bfO^{p'}(\GP)$, the image of $\GP$ in $\GL(U)$ is contained in
$\SL(U)$. Now if $|\GP|$ is odd, then $\GP$ is solvable and so by the Fong-Swan theorem  cannot act irreducibly on $U$ of dimension $2$. So $\GP$ contains an element of order $2$ which must then act as $-1_U$ and belong to $\bfZ(\GP)$. Thus $U$ and
$W$ have different central characters and so $V_\GP$ is semisimple, contradicting
Lemma \ref{semi2}.
\end{proof}

Next we will prove some criteria to decide the type of a self-dual indecomposable module.

\begin{lem}\label{type1}
Let $k = \overline{k}$ be of characteristic $p > 2$ and let $V$ be a self-dual indecomposable
$kG$-module with $\dim \End_{kG}(V) \leq 2$. Then $V$ supports a non-degenerate
$G$-invariant bilinear form that is either symmetric or alternating.
Furthermore, all such forms have the same, symmetric or alternating, type.
\end{lem}

\begin{proof}
Let $\Phi$ denote the matrix representation of $G$ on $V$ relative to a fixed basis
$(e_1, \ldots ,e_n)$ of $V$. Since $V \cong V^*$ as $G$-modules, we can find $b \in \GL_n(k)$
such that $\tw t \Phi(g)^{-1} = b\Phi(g)b^{-1}$, and so $b$ yields a non-degenerate $G$-invariant
bilinear form on $V$. Note that the map $\pi:X \mapsto bX$ yields a $k$-space isomorphism
between $\End_{kG}(V)$ and the space $B$ of $G$-invariant bilinear forms on $V$.
In particular, $\dim B \leq 2$, and, since $p > 2$, it is a direct sum $S \oplus A$ of symmetric and alternating $G$-invariant forms. Hence the claims follow if $\dim \End_{kG}(V) = 1$.
Assume $\dim \End_{kG}(V) = 2$. Since $V$ is indecomposable, $\End_{kG}(V)$ is a local
algebra, cf. \cite[Corollary 1.6.5]{LP}, and its unique maximal ideal $J$, which then has dimension
$1$, consists of (nilpotent) non-units. Thus $\pi(J)$ is contained in the subset $D$ of degenerate
$G$-invariant bilinear forms on $V$. But $\pi^{-1}(D)$ is obviously contained in $J$. It follows
that $D = \pi(J)$ is a subspace and $\dim D = 1$. Hence we are also done if $S$ or $A$ is zero. Assume $S, A \neq 0$, whence both of them are $1$-dimensional. Now if $Y \in D$, then
$\tw t Y \in B$ and it is degenerate. As $p > 2$ and $\dim D = 1$, it follows that
$\tw t Y = \pm Y$. Thus $D$ is either $S$ or $A$, and so the nonzero forms in the other subspace are precisely the non-degenerate $G$-invariant forms on $V$ that are either symmetric or alternating.
\end{proof}

\begin{lem}\label{type2}
{Suppose that $G$ is a finite group with a Sylow $p$-subgroup $P$ of order $p > 2$ such that
$\bfN_G(P)/P$ is abelian. Let $V$ be a reducible self-dual indecomposable $G$-module over
$k = \overline{k}$ of characteristic $p$, of even dimension $d < 2p$. Then $V$ is not
orthogonal if $d < p$ and $V$ is not symplectic if $d > p$.}
\end{lem}

\begin{proof}
For $1 \le i \le p$, let $X_i$ denote the unique indecomposable
$kP$-module of dimension $i$ (so $X_p$ is projective). By the Green correspondence
(see e.g. \cite[Theorem 4.9.2]{LP}),
$V_N = X \oplus Y$ for $N := \bfN_G(P)$, where $X$ is non-projective indecomposable
and $Y$ is projective (if nonzero). Let $M$ denote any indecomposable
$kN$-module. According to \cite[p. 42]{Alp}, $M$ is uniserial. Also, Lemma 8 of \cite[\S5]{Alp} says that the $P$-radical filtration agrees with the $N$-radical filtration on $M$; in particular,
$\rad(M) = \rad(M_P)$. As $N/P$ is abelian, any irreducible $kN$-module remains irreducible as
over $P$. It follows that $M_P$ is indecomposable. Applying this to $X$ and $Y$, we see that $V_P = X_d$ if $d < p$, respectively $V_P = X_{d-p} \oplus X_p$ if $d > p$.
Now suppose that $V$ is equipped with a non-degenerate $G$-invariant bilinear form of a
fixed parity.
The claim then follows by using the description of Jordan forms of unipotent elements in
classical groups, see e.g. \cite[Theorem 3.1]{LS}.
\end{proof}

{\bf Proof of Theorem \ref{thm: indecomposable}.}
There is nothing to prove for $p = 2$; furthermore $p \neq 3$ by Lemma
\ref{indec3}. So we may assume $p > 3$. By Proposition
\ref{indec-str}, the self-duality of $V$ implies that $\GP$ is quasisimple. If
furthermore $\GP$ is not a Lie-type group in characteristic $p$, then by Theorem
\ref{indec-qs} we arrive at (i) and (ii). Assume that $\GP$ is of Lie-type in characteristic
$p$. By Lemma \ref{ext-defi}(i), $\GP \cong \SL_2(q)$ or $\PSL_2(q)$
for some $q = p^a$. By Corollary \ref{indec-gp}, $V_\GP$ is
indecomposable of length $\geq 2$. Applying Proposition \ref{prop1}, we arrive at
(i) and (ii).

Note that in each of the listed cases, there is a unique (up to isomorphism) reducible indecomposable $\GP$-module $V$ of the indicated shape (indeed, if $W := \hd(V_\GP)$ then
there is a unique quotient of $\cP(W)$ of this shape). Since
$W^* \cong W \cong \soc(V_\GP)$, it follows that $V_\GP$ is self-dual. Thus all the
listed cases give rise to examples of reducible indecomposable self-dual modules (at least for
$\GP$).

It remains to determine the type of each indecomposable module. Note that in all cases
$\dim \End_{k\GP}(V) = 2$, whence $\dim \End_{kG}(V) \leq 2$ and Lemma \ref{type1} applies to both $G$ and $\GP$. Thus $V$ supports a non-degenerate $G$-invariant form that is either symmetric or alternating. If $\dim V$ is odd, then all such forms must be symmetric. Consider the case $\dim V$ is even. Note that in all cases $|P| = p$ and $\bfN_\GP(P)/P$ is abelian for
$P \in \Syl_p(\GP)$. So by Lemma \ref{type2}, all such forms are symmetric when $\dim V > p$,
respectively alternating when $\dim V < p$.
\hfill $\Box$

Recall from \cite{Serre2} that for $\cG$ a connected reductive group over an algebraically closed field $k$ and for $G \le \cG$ a subgroup 
we say that $G$ is \emph{$\cG$-cr} if whenever $G \le \cP$ for a parabolic subgroup $\cP$ of $\cG$, then $G$ is
contained in a Levi subgroup of $\cP$. If $\cG = \Sp(V)$ or $\SO(V)$ for some finite-dimensional vector space equipped with
a non-degenerate alternating or symmetric bilinear form, then this is equivalent to saying that for any $G$-stable isotropic
subspace $W \subset V$ there exists a $G$-stable isotropic
subspace $W' \subset V$ with $W+W'$ non-degenerate. For these $\cG$ and provided $p > 2$,
a subgroup $G \le \cG$ is $\cG$-cr if and only if the $kG$-module $V$ is
completely reducible \cite[\S3.2.2]{Serre2}.

We can extend Serre's notion to the disconnected group $\cG = \OO(V)$ by saying that a subgroup $G$ is {\it $\OO(V)$-cr} if for any
$G$-stable isotropic subspace $W \subset V$ there exists a $G$-stable isotropic subspace $W' \subset V$ with $W+W'$
non-degenerate. We then see using the same argument as in \cite[\S3.2.2]{Serre2} as well as Lemma~\ref{semi2}(i), that for $G \le \OO(V)$ and $p > 2$ the following are equivalent:
\begin{enumerate}
\item $G$ is $\OO(V)$-cr,
\item $G \cap \SO(V)$ is $\SO(V)$-cr,
\item the $kG$-module $V$ is completely reducible.
\end{enumerate}

The next result shows that for $\cG = \Sp(V)$ or $\OO(V)$, the finite non-$\cG$-cr subgroups
of $\cG$ are made up from the groups with a nontrivial unipotent normal subgroup and
the groups described in Theorem \ref{thm: indecomposable}. 

\begin{prop}\label{serre}
Let $k = \overline{k}$ be of characteristic $p > 0$ and let $\cG$ be either $\Sp(V)$ or
$\OO(V)$ with $\dim_k V \leq 2p-3$. Suppose that $G < \cG$ is a finite subgroup such that
the $G$-module $V$ is not completely reducible. Then there is a $G$-invariant decomposition
$V = V_1 \oplus V_2 \oplus V_3$ of $V$ into an orthogonal direct sum of three subspaces, where $V_i$ is either zero or non-degenerate, at least one of the $V_i$'s is zero and at least one of $V_1$ and $V_2$ is nonzero, and the following conditions hold for the images $G_i$ of $G$ in $\GL(V_i)$.
\begin{enumerate}[\rm(i)]
\item If $V_1 \neq 0$, then $\bfO_p(G_1) = 1$, the $kG$-module $V_1$ is
reducible indecomposable, and $(G_1,V_1)$ is as described in Theorem
\ref{thm: indecomposable}.

\item If $V_2 \neq 0$, then $\bfO_p(G_2) \neq 1$.

\item If $V_3 \neq 0$, then $V_3$ is an orthogonal direct sum of
non-degenerate subspaces, each being an irreducible $G$-module.
\end{enumerate}
\end{prop}

\begin{proof}
(a) First note that $p > 2$.
Setting $V_2 = V$ when $\bfO_p(G) \neq 1$, we may assume $\bfO_p(G) = 1$. Setting
$V_1 = V$ when $V_G$ is indecomposable, we may assume that $V_G$ is decomposable.

First we consider the case where no composition factor of $G$ has order $p$.
Choose a decomposition $V_G = A \oplus B$ with $A,B \neq 0$ being $G$-invariant and $A$ of
smallest possible dimension. Then $\dim A \leq p-2$ and the image $X$ of $G$ in $\GL(A)$ has
$\bfO_p(X) = 1$.  By \cite{Gcr}, the $X$-module $A$ is completely reducible, whence it is
irreducible by its choice. If $A$ is non-degenerate, then $V_G = A \oplus A^\perp$.
Consider the case $A \cap A^\perp \neq 0$. By the irreducibility of $A$, $A \subseteq A^\perp$ and
so $A^\perp = A \oplus C$ for $C := B \cap A^\perp$. It is easy to see that $C \cap C^\perp = 0$
and so $V_G = C \oplus C^\perp$. Note that $C^\perp \neq 0$. Also, $C \neq 0$ as otherwise
$A^\perp = A$, $B \cong V/A = V/A^\perp \cong A^*$ is an irreducible $G$-module and so
$V_G$ is semisimple, a contradiction. Thus in either case
$V$ is an orthogonal direct sum of nonzero non-degenerate $G$-invariant subspaces.
Repeating this process for the summands, we obtain an orthogonal direct sum
$V = \oplus^n_{i=1}U_i$, where each $U_i$ is non-degenerate and indecomposable as
a $kG$-module, and $n \geq 2$. Since $V_G$ is not semisimple, we may assume that $U_1$ is
reducible. Again the image $Y_i$ of $G$ in $\GL(U_i)$ has $\bfO_p(Y_i) = 1$. By Theorem
\ref{thm: indecomposable}, $\dim U_1 \geq p-1$, whence
all $U_i$ with $i \geq 2$ must be irreducible over $G$. Setting $V_1 = U_1$ and
$V_3 = \oplus^n_{i=2}U_i$, we are done. Note that in this case $\dim V > \dim U_1 \geq p-1$.

\smallskip
(b) Let $W_1, \ldots ,W_m$ denote all the composition factors of $V_\GP$ (with counting multiplicities)
and let $J := \bfO_{p'}(\GP)$.

Consider the case $p = 3$. If $\dim W_i = 1$ for all $i$, then the first paragraph of the proof
of Lemma \ref{indec3} shows that $\bfO_p(G) \neq 1$, contrary to our hypotheses. As $V_G$
is decomposable of dimension $\leq 3$, it follows that $V_\GP = W_1 \oplus W_2$ with
$\{\dim W_1,\dim W_2\} = \{1,2\}$ and this decomposition is $G$-invariant. Thus $V_G$ is
completely reducible, again a contradiction. So we must have that $p > 3$.

Suppose that $J$ acts by scalars on each of the $W_i$'s. Then, in a suitable basis of $V$,
$[J,\GP]$ is represented by unitriangular matrices, and so it is a $p$-subgroup. But
$\bfO_p(G) = 1$, so $J \leq \bfZ(\GP)$. Applying Lemma \ref{eg} to $\GP$, we see that
$\GP$, and so $G$ as well, has no composition factors of order $p$. Thus we are done by (a).

So we may now assume that $J$ does not act by scalars on $W_1$. It follows that the image of
$\GP$ in $\GL(W_1)$ contains a non-scalar normal $p'$-subgroup. Applying Theorem \ref{bz}, we see that $\dim W_1 \geq p-1$.
Since $\dim V \leq 2p-3$, it follows
that $J$ acts by scalars on each $W_i$ with $i > 1$, and $m \geq 2$ as $V_G$ is reducible.
Since $J$ is a $p'$-group, $V_J \cong W_1 \oplus (\oplus^m_{i=2}W_i)$, and this decomposition
is $G$-invariant. It follows
that $G$ fixes a decomposition $V = W_1 \oplus U$ where $U_\GP$ has composition
factors $W_i$, $2 \leq i \leq m$. Since $J$ acts by scalars on each $W_i$ with $i > 1$ but not on
$W_1$, it also follows that
$U = W_1^\perp$, whence $W_1$ is non-degenerate. If $\bfO_p(Y) = 1$ for
the image $Y$ of $G$ in $\GL(U)$, then $U_G$ is semisimple by \cite{Gcr} (as $\dim U \leq p-2$),
a contradiction. So we can now set $V_2 = U$ and $V_3 = W_1$. Note that in this case
$\dim V > \dim W_1 \geq p-1$.
\end{proof}

{\bf Proof of Corollary \ref{clas}.} Suppose that the $kG$-module $V$ is not completely reducible.
If $V_G$ is indecomposable, then we are done by Theorem \ref{thm: indecomposable}. Otherwise,
the proof of Proposition \ref{serre} shows that $\dim V \geq p$.
\hfill $\Box$

\section{Adequacy for \texorpdfstring{$\SL_2(q)$}{SL\_2(q)}}\label{sec: sl2q}
The aim of this section is to prove the following statement which extends the results of 
\cite[\S3]{GHT}:
 
\begin{prop}\label{sl2q}
Any nontrivial irreducible representation $V$ of $G := \SL_2(p^r)$ over $\overline\bbF_p$ is weakly 
adequate, except when $q :=p^r \leq 3$. 
\end{prop}

By the Steinberg tensor product theorem we can write 
$$V = L(a) := \otimes^{r-1}_{i=0}L(a_i)^{(i)}$$
for some $a = \sum^{r-1}_{i=0}a_ip^i$, $0 \leq a_i \leq p-1$, where 
$L(1)$ is the natural $2$-dimensional $\overline\bbF_pG$-representation, $L(b) = \Sym^{b}(L(1))$, and 
$^{(i)}$ denote the $i$th Frobenius twist.  Also, $\cG \cong \SL_2$ denotes the underlying algebraic group for $G$.

\begin{lem}\label{head}
We have that
$$\hd_\cG(\End(V)) \cong \bigoplus_{b_0, \ldots ,b_{r-1}\ :\ 0 \leq b_i \leq \min(\frac{p-1}{2},a_i)}
   \otimes^{r-1}_{i=0}L(2b_i)^{(i)}.$$
Moreover, if $a < q-1$ then 
$$\hd_G(\End(V)) = \hd_\cG(\End(V)),$$ 
whereas if $a = q-1$, then 
$$\hd_G(\End(V)) = \hd_\cG(\End(V)) \oplus L(q-1).$$     
\end{lem}

\begin{proof}
As $\End(V)$ is self-dual, we may replace ``head'' by ''socle''. By \cite[Lemmas 1.1 and 1.3]{DH}, 
we have for $0 \leq b \leq (p-1)/2$ 
$$L(b) \otimes L(b) \cong \oplus^b_{i=0}T(2i),$$
and for $(p-1)/2 \leq b \leq p-1$,
$$L(b) \otimes L(b) \cong \oplus^{p-2-b}_{i=0}T(2i) \oplus \oplus ^{\lfloor (p-1)/2\rfloor}_{i=p-1-b}T(2p-2-2i),$$
where $T(\la)$ denotes the tilting module of $\cG$ with highest weight $\la \geq 0$.  Recall that
$T(\la) = L(\la)$ if $\la \leq p-1$, and that, when $0 \leq \la \leq p-2$, $T(2p-2-\la)$ is uniserial
of shape $(L(\la)|L(2p-2-\la)|L(\la))$ and $T(2p-2-\la) \cong Q_1(\la)$ in the notation of 
\cite[\S3]{AJL}. 

The statement will follow if we can show for any $0 \leq b_i \leq 2p-2$ that (i) 
$\soc_\cG(\otimes^{r-1}_{i=0}T(b_i)^{(i)})$ is simple, and (ii) 
$\soc_G(\otimes^{r-1}_{i=0}T(b_i)^{(i)})$ is simple if $b_i < 2p-2$ for at least one $i$ and isomorphic to 
$L(0) \oplus L(q-1)$ otherwise. Let $c_i:= \min(b_i,2p-2-b_i) \leq p-1$ and $c:= \sum^{r-1}_{i=0}c_ip^i$.
Then
$$\otimes^{r-1}_{i=0}T(b_i)^{(i)} \hookrightarrow \otimes^{r-1}_{i=0}Q_1(c_i)^{(i)}  = Q_r(c)$$
in the notation of \cite[\S3]{AJL}. By \cite[Theorem 3.7]{AJL}, $\soc_\cG(Q_r(c)) = L(c)$. 
Furthermore, by \cite[Lemma 4.1]{AJL}, $\soc_G(Q_r(c)) = L(c)$ if $c \neq 0$ and 
$\soc_G(Q_r(c)) = L(0) \oplus L(q-1)$ if $c = 0$ (note that ``$\otimes$'' should be 
``$\oplus$'' in \cite[Lemma 4.1(b)]{AJL}). Finally, if $c=0$ then it is easy to check that 
$L(q-1)$ does not occur in $\otimes^{r-1}_{i=0}T(b_i)^{(i)}$, unless $b_i = 2p-2$ for all $i$.
\end{proof}

\begin{proof}[Proof of Proposition \ref{sl2q}]
By \cite[Proposition 3.1]{GHT} we may assume that $r > 1$. We will follow the same strategy of 
proof. It suffices to show $M = \hd_G(\End(V))$, where $M$ denotes the span of the images of
all $p'$-elements of $G$ in $\hd_G(\End(V))$. 

Suppose that $k = \sum^{r-1}_{i=0}k_ip^i$ with $0 \leq k_i \leq \min((p-1)/2,a_i)$. By 
\cite[Lemma 3.5]{GHT}, the $\cG$-subrepresentation $L(2k)$ in $\End(V)$ 
is generated by the weight
$0$ element $\Delta_k := \otimes^{r-1}_{i=0}\Delta_{k_i}^{(i)}$, where $\Delta_{k_i} \in \End(L(a_i))$ is 
defined in \cite[Lemma 3.5]{GHT}. Let $\delta_k := \tr(- \circ \Delta_k) \in (\End(V))^*$. For 
$\ell = \sum^{r-1}_{i=0}\ell_ip^i$ with $0 \leq \ell_i \leq a_i$, let 
$\pi_\ell := \otimes^{r-1}_{i=0}\pi_{\ell_i}^{(i)} \in \End(V)$, where $\pi_{\ell_i} \in \End(L(a_i))$ is the projection
$X^jY^{a_i-j} \mapsto \delta_{j \ell_i}X^jY^{a_i-j}$. For any other $\ell$ let $\pi_\ell := 0$. Also let
$p_k(\ell) := \delta_k(\pi_\ell) \in \overline\bbF_p$. Then $p_k(\ell) = \prod^{r-1}_{i=0}p_{k_i}(\ell_i)$, 
where $p_{k_i}(\ell_i)$ agrees with a polynomial of degree $k_i$ for $0 \leq \ell_i \leq a_i$. In particular,
as $k_i \leq a_i$, there exist $0 \leq \ell_i \leq a_i$ such that $p_{k_i}(\ell_i) \neq 0$ for all $i$. Thus
$p_k(\ell) \neq 0$ for some $\ell$.

\smallskip
(a) Suppose $a < q-1$ and $p > 2$. Also, suppose that there exists a $k = \sum^{r-1}_{i=0}k_ip^i$ with
$0 \leq k_i \leq \min((p-1)/2,a_i)$ such that $M$ does not contain $L(2k)$. Then $\delta_k(M) = 0$,
so the action of the split Cartan subgroup gives 
$$\sum_{\ell \equiv \ell' \!\!\!\!\!\pmod {(q-1)/2}}p_k(\ell) = 0, \ \forall \ell'.$$ 
As in \cite[\S3]{GHT}, the action of a non-split Cartan subgroup similarly gives  
$$\sum_{\ell \equiv \ell' \!\!\!\!\!\pmod {(q+1)/2}}p_k(\ell) = 0, \ \forall \ell'.$$ 
Therefore, if $0 \leq \ell < (q-1)/2$, $p_k(\ell) = -p_k(\ell+(q-1)/2) = p_k(\ell-1)$, and so by induction
$p_k(\ell) = p_k(\ell-1) = \ldots = p_k(-1) = 0$. Similarly, $p_k(\ell) = 0$ for 
$\ell > a-(q-1)/2$, and so $p_k(\ell) = 0$ for all $\ell$, a contradiction. 

\smallskip
(b) Now we consider the case $a=q-1$ and $p > 2$.

\smallskip
(b1) Suppose that $M$ does not contain $L(2k)$ for some $k < (q-1)/2$. The same
argument as in (a) shows that 
$$p_k(0) = p_k(1) = \ldots = p_k((q-3)/2) = -p_k((q+1)/2)  = \ldots = -p_k(q-1)$$
and $p_k((q-1)/2)= 0$. Hence $p_{k_i}((p-1)/2) = 0$ for some $i$. As $r > 1$ we deduce that
$p_k(\ell) = 0$ for some $0 \leq \ell < (q-1)/2$ (e.g. $\ell = p^i(p-1)/2$), so $p_k(\ell) = 0$ for all $\ell$,
again a contradiction.

\smallskip
(b2) Suppose that $M$ does not contain $L(q-1)^{\oplus2}$. By \cite[\S3]{GHT},
the $\cG$-representation generated by 
$v:= \otimes^{r-1}_{i=0}\left[\left(X\frac{\partial}{\partial Y}\right)^{(p-1)/2}\right]^{(i)}$ 
is the unique $\cG$-subrepresentation $L(q-1)$ in $\End(V)$. Note that the upper-triangular
Borel subgroup $B := (\begin{smallmatrix} *&* \\ &*\end{smallmatrix})\subset G$ fixes 
$v^2 = \otimes^{r-1}_{i=0}\left[\left(X\frac{\partial}{\partial Y}\right)^{p-1}\right]^{(i)}$ and that
$v$ and $v^2$ are linearly independent. As $v^2 \notin (\End(V))^G = \overline\bbF_p$,  
the $G$-representation generated by $v^2$ is isomorphic to $L(q-1)$ or to $L(q-1) \oplus L(0) \cong \Ind^G_B(\bbone)$.
In particular, for some $c \in \overline\bbF_p$, $v^2+c$ generates the second copy of $L(q-1)$ in 
$\End(V)$. A calculation as in \cite[\S3]{GHT} shows that $c = (-1)^r$. Now we can deduce
that $p_k(\ell)=0$ for all $\ell$ exactly as in part (b2) of the proof of \cite[Proposition 3.1]{GHT}.

\smallskip
(c) Suppose now that $p = 2$. Note that $\hd_G(\End(V))$ is multiplicity-free. If $M$ does not 
contain $L(0)$, then the argument in (a) (but using only a non-split Cartan 
subgroup) shows that $p_0(\ell) = 0$ for all $\ell$ (as $q+1 > a$), a contradiction. (In fact, we could 
alternatively use only a split Cartan subgroup, even when $a=q=1$.) Suppose that $a=q-1$ and that
$M$ does not contain $L(q-1)$. By (b2), 
$\otimes^{r-1}_{i=0}\left(X\frac{\partial}{\partial Y}\right)^{(i)}+1$ generates the unique 
$G$-subrepresentation $L(q-1)$ of $\End(V)$. However, as in (b2) of the proof of 
\cite[Proposition 3.1]{GHT}, 
$$\tr\left(\begin{pmatrix}\alpha & \\ & \alpha^{-1} \end{pmatrix} \circ  
     \otimes^{r-1}_{i=0}(X\frac{\partial}{\partial Y})^{(i)} \right) \neq 0$$
for any $\alpha \in \bbF_q^\times \setminus \{1\} \neq \varnothing$, and this gives a final contradiction.       
\end{proof}

\begin{remark}
The results of \cite{AJL} play a key role in our analysis of $\SL_2(q)$-representations. We should also
point out some minor inaccuracies in \cite[\S4]{AJL}. The first line of the displayed formula right 
before \cite[Corollary 4.5]{AJL} should have the extra condition $\lambda, \mu \neq p-1$. Furthermore,
in the case $n = 2$ of  \cite[Corollary 4.5(b)]{AJL}, there are four (not just two as stated) cases when 
$\dim\Ext^1 = 2$, namely when $\la_0, \la_1 \in \{(p-3)/2, (p-1)/2\}$ and $\mu_i = p-2-\la_i$ for all 
$i = 0,1$. (Also, the $k$ and $i$ in \cite[Corollary 4.5(a)]{AJL} satisfy $0 \leq i,k \leq n-1$.)
\end{remark}

\begin{cor}\label{sl2q-adequate}
Let $V$ be nontrivial absolutely irreducible representation of $G = \SL_2(p^r)$ in characteristic $p$. Then either $V$ is adequate, or one of the following holds:
\begin{enumerate}[\rm(i)]
\item $r=1$, $1 < \dim(V) = (p \pm 1)/2$, and $\dim\Ext^1_G(V,V) = 1$.
\item $p^r=2,3,4$ and $\dim V = p^r$.
\item $p^r = 9$ and $\dim V = 3,6,9$.
\end{enumerate}
\end{cor}

\begin{proof}
The case $r = 1$ is already treated by \cite[Corollary 1.4]{GHT}, so we will assume $r > 1$. In this
case, $\Ext^1_G(V,V) = 0$ by \cite[Corollary 4.5(a)]{AJL}. Suppose that $p^r \neq 4,9$. Then 
$H^2(G,k) = 0$, and furthermore $H^1(G,k) = 0$ as $G$ is perfect. It follows that $V$ is adequate.
The same conclusion holds if $p \nmid \dim(V)$. 

\smallskip
Now we consider the case $p^r = 4,9$ and keep the notation of the proof of Proposition
\ref{sl2q}. The proof of Lemma \ref{head} shows that the one-dimensional subspace $\End(V)^G$ is contained in the direct summand 
$W:=T(b_0) \otimes T(b_1)^{(1)}$, where $b_i =  0$ if $a_i < p-1$ and $b_i=2p-2$ if $a_i = p-1$.
As $H^1(G,\End(V)) = 0$, we deduce that $H^1(G,\End(V)/k) = H^1(G,W/k)$.

\smallskip
(a) If $a_0=a_1= p-1$, then $W = Q_2(0) \cong \cP(\bbone) \oplus L(p^2-1)$ by 
\cite[Lemma 4.1]{AJL}. Hence $H^1(G,\End(V)/k) \cong  H^1(G,\cP(\bbone)/k) \cong H^2(G,k)$, which
is one-dimensional.

(b) Suppose that precisely one of $a_0,a_1$ is $p-1$. Without loss we may assume that 
$a_0  =p-1 > a_1$. Then $W \cong T(2p-2)$. Note that $T(2p-2)$ has composition factors 
$L(0) = \bbone$ (twice) and $L(p-2) \oplus L(1)^{(1)}$. As $T(2p-2)$ is self-dual and injects into
$Q_2(0)$, we deduce that it is uniserial with trivial socle and head. Thus the sequence
$$0 \to k \to H^1(G, L(p-2) \otimes L(1)^{(1)}) \to H^1(G,\End(V)/k) \to H^1(G,k) = 0$$
is exact, whence
$$\dim H^1(G,\End(V)/k) = \dim \Ext^1_G(k,L(p-2) \otimes L(1)^{(1)}) -1 
    = \left\{ \begin{array}{ll}1, & \mbox{if }p=3,\\0, & \mbox{if }p=2,\end{array} \right.$$
by \cite[Corollary 4.5]{AJL}.
\end{proof}

If one replaces $\SL_2(q)$ by $\GL_2(q)$, then in fact there are no exceptions to
adequacy for $q > 3$ odd (if $q=3$ and $\dim V=3$, then weak adequacy fails). 

\begin{cor}\label{gl2q-adequate}  Let $G$ be a finite group and 
 $V$ be a faithful absolutely irreducible representation of 
$G$ in odd characteristic $p$.  If the image of $G$ in $\PGL(V)$ is $\PGL_2(p^a)$
with $p^a > 3$,
then $(G,V)$ is adequate.
\end{cor}

\begin{proof}   
Without loss we may assume that $V$ is an irreducible $kG$-module with $k = \overline{k}$.
 Let $H$ be the inverse image of  $\PSL_2(p^a)$ under the projection from $G$ onto 
 $\PGL_2(p^a) < \PGL(V)$. Then $H/\bfZ(H) \cong \PSL_2(p^a)$ and 
 $\bfZ(H) = \bfZ(G)$ is a $p'$-group. Since the universal $p'$-cover of 
 $\PSL_2(p^a)$ is $\SL_2(p^a)$, it follows that $H = \bfZ(H)L$, where 
 $L:=[H,H]$ is the quotient of $\SL_2(p^a)$ by a central subgroup (of order $1$ or $2$).
 Moreover, $G$ normalizes $L$ and centralizes $\bfZ(H)$,
 and in fact $G$ induces the full subgroup of inner-diagonal automorphisms of $L$.
 It is well known that any irreducible $kL$-representation is invariant under 
 any inner-diagonal automorphism. It follows that, if $W$ is an irreducible $kH$-summand of
 $V|_H$, then $G$ preserves the isomorphism class of $W$. But $G/H \cong C_2$, hence
 $W$ extends to a $kG$-module $\tilde{W}$ (see e.g. \cite[Theorem 8.12]{N}), and so 
 by Frobenius' reciprocity, $V \cong \tilde{W} \otimes_k A$ for some one-dimensional 
 $k(G/H)$-module $A$. We have shown that $V_H$ is irreducible. By Proposition \ref{sl2q},
$(H,V)$ is weakly adequate, so is $(G,V)$. Also, as  
$\bfO^p(H) = H$ and $G/H \cong C_2$, we see that $\bfO^p(G) = G$ and so 
$H^1(G,k) = 0$. Moreover, since $p \ne 2$, the inflation-restriction sequence in cohomology implies that adequacy of $(H,V)$ yields the same for $(G,V)$. 
So by Corollary \ref{sl2q-adequate}, it suffices to consider the following  cases:
\begin{enumerate}[\rm(a)]
\item $a=1$ and $\dim V = (p \pm 1)/2$; 
\item  $p^a=9$ and $\dim V=3,6,9$.
\end{enumerate}
In the first case,  $G$ has a cyclic Sylow $p$-subgroup $P$ of order $p$. It follows that
$H^2(G,k) = 0$ and so it it suffices to show that $\Ext^1_G(V,V) = 0$. Note that $V$ has no lifts to characteristic $0$ (since $\bfN_G(P)$ acts transitively on the $p-1$ nontrivial elements of 
$P$, an element $g \in P$ of order $p$ would have at least $p-1$ distinct
eigenvalues in any characteristic $0$ lift). Thus, $\ext^1_G(V,V)=0$ by Lemma \ref{cyclic} and $(G,V)$ is adequate.  

In the second case, note that $H^2(\PGL_2(9),k)=0$ \cite{Atlas}. 
Since $G/\bfZ(G) \cong \PGL_2(9)$ and $\bfZ(G)$ is a $p'$-group,
it follows that $H^2(G,k) = 0$. So again it suffices to show that $\ext_G^1(V,V)=0$. 
Note that the $p'$-group $\bfZ(H)$ acts trivially on $V \otimes_k V^*$ and 
$H^1(L, V \otimes_k V^*)=\Ext^1_L(V,V)= 0$ by \cite[Lemma 8.1]{GHT}. Hence, 
$H^1(H, V \otimes_k V^*)=0$. Since $G/H \cong C_2$, it follows that 
$H^1(G,V \otimes_k V^*) = 0$, and so we are done.
 \end{proof} 
 
 \section{Adequacy for \texorpdfstring{$\SL_n(q)$}{SL\_n(q)}}  \label{SLn adequacy}
 
 In this section, we give another family of examples of modules that
 are adequate.  The result follows from \cite[Lemma 2.5.6]{CHT} if $p > n$.  Also, 
 recall that the case $n=2$ was considered in the previous section.
 
 \begin{thm}  Let $p$ be a prime and $q$ a power of $p$.  Let $k$ be 
 an algebraically closed field of characteristic $p$ and $V = k^n$ with $n > 2$.  Suppose
 that $G < \GL(V)$ is a finite group that contains a normal subgroup $S \cong \SL_n(q)$.
 Then $(G,V)$ is adequate.  
 \end{thm}
 
 \begin{proof}   
 By \cite[Proposition 5.4.11]{KL}, any nontrivial irreducible $kS$-representations of dimension 
 $\leq n$ is quasi-equivalent to the natural $n$-dimensional $kS$-module $U$. It follows that 
 the $kS$-module $V$ is irreducible and quasi-equivalent to $U$. Next, the only automorphisms of
 $S$ that preserve the isomorphism class of $U$ (hence of $V$) are the inner-diagonal automorphisms. 
 Therefore, $G$ induces only inner-diagonal automorphisms of $S$, and so 
$p \nmid [G:S]$. This implies that it is enough to prove the statement for $S$ with $V$ being the 
standard representation. 
 
 Note that $V \otimes V^* = W \oplus k$ with $W$ irreducible if $p \nmid n$ and
 that $V \otimes V^*$ is uniserial (of length three) with trivial head and socle if $p\mid n$.
 Let $W$ denote the unique nontrivial irreducible composition factor of $V \otimes V^*$.
 Since there are semisimple elements in $S$ with nonzero trace on $V$ and
 since not all semisimple elements of $S$ are scalars, it follows that  $\End(V)$ is spanned
 by the images of the semisimple elements of $S$. 
 
 By the table in \cite{JP} or by \cite[Theorem 9]{TZ}, it follows that $\ext_S^1(V,V)=0$
 whence the result holds as long as $p \nmid n$.  If $p \mid n$, then
using the fact that $H^1(S,k)=0$ and  $H^0(S,W)=0$ and the long exact
sequence for cohomology we see that $H^1(S,  V \otimes V^*/k) =0$ if and only 
if $ \dim  H^1(S,W) = 1$.
By \cite{JP},   this is the case and so the result follows (one can give
an alternate proof using \cite{TZ} as well).  
 \end{proof}
 
 A slight modification of the proof shows that if $\gcd(n,q)=1$, then $(G,V)$ is in fact big.   
 Indeed, we only need observe the obvious fact that there exists
 a semisimple regular element $g \in \SL_n(q)$ with nonzero trace.

\section{Asymptotic adequacy}\label{sec: asymp}

In this section, we extend \cite[Theorem 1.2]{Gapp} to include disconnected groups as well as to allow
the possibility that $p$ divides $\dim V$. First we prove some statements relating 
{\it discrete} cohomology (i.e. of abstract groups) and {\it rational} cohomology (i.e. in the 
category of rational modules), 
of linear algebraic groups on the one side, and cohomology of finite groups of Lie type on the other side. We use subscripts $_\abs$ and $_\rat$ to make
distinction between these two types of cohomology groups. 

First we record the following result (which essentially is a special case of a result of
van der Kallen, cf. \cite[p. 239]{Par}):

\begin{lemma}\label{closure}
Let $k = \overline{\bbF}_p$ and let $\cG$ be a linear algebraic group defined 
over $\bbF_q \subset k$. Let $V$ be a finite-dimensional rational $k\cG(k)$-module. If $H^1(\cG(\bbF_{q^f}),V)=0$ for large enough $f$, then $H^1_\abs(\cG(k),V)=0$.
\end{lemma}

\begin{proof}
First we note that if $U$ is any finite-dimensional $k\cG(k)$-module, then $\cG(\bbF_{q^f})$ and 
$\cG(k)$ have the same subspace of fixed points on $U$ when $f$ is divisible by some 
integer $N = N(U)$. Indeed, let $U_j$ denote the fixed point subspace for $\cG(\bbF_{q^{j!}})$ on $U$. Then $U_1 \supseteq U_2 \supseteq \ldots$, and so $U_j$ stabilizes when $j \geq j_0$ for some $j_0$. But each element of $\cG(k)$ is contained in 
$\cG(\bbF_{q^{j!}})$ for some $j \geq j_0$. It follows that $\bfC_U(\cG(k)) = U_{j_0} = \bfC_U(\cG(\bbF_{q^{j!}}))$ for all
$j \geq j_0$. In particular, we can choose $N = j_0!$.

Consider any exact sequence $0 \to V \to W \to k \to 0$ of $k\cG(k)$-modules. By 
assumption, it is split over $\cG(\bbF_{q^f})$ for $f$ large enough. Hence
$\bfC_W(\cG(\bbF_{q^f}))$ has dimension equal to $\dim_k \bfC_V(\cG(\bbF_{q^f})) + 1$, which by our claim 
is equal to $\dim_k\bfC_V(\cG(k)) +1$ when $N(U)|f$. Again by our claim,
$\bfC_W(\cG(\bbF_{q^f})) = \bfC_W(\cG(k))$ for $N(W)|f$. It follows that 
$\dim_k \bfC_W(\cG(k)) = \dim_k \bfC_V(\cG(k)) + 1$, whence $W$ is split over $\cG(k)$. Hence
$H^1_\abs(\cG(k),V)=0$.
\end{proof}

Next we observe that the results of \cite[Prop. II.2.14]{jantzen} and \cite[Theorem 6.6]{CPS}  hold in more generality than they were stated. 

\begin{prop}\label{disc}
Let $k$ be an algebraically closed field of characteristic $p$ and let $\cG$ be a (not 
necessarily connected) reductive algebraic group defined over $k$. Let $V$ be a 
finite-dimensional rational $k\cG$-module.
\begin{enumerate}[\rm (i)]
\item Suppose that $p \nmid [\cG:\cG^0]$ and $V$ is irreducible. Then $\Ext^1_\cG(V,V)_\rat = 0$.
\item Suppose $\cG$ is connected and defined over $\bbF_q \subset k$. 
Then for $e$ and $f$ large enough (depending on $V$ and $n$), 
$$H^n_\rat(\cG,V(e)) \cong H^n(\cG(\bbF_{q^f}),V(e)) \cong H^n(\cG(\bbF_{q^f}),V),$$
where $V(e)$ is the $e$-th Frobenius twist of $V$.
\end{enumerate}
\end{prop}

\begin{proof}
(i) Since $p \nmid [\cG:\cG^0]$,
it suffices to show that $\Ext_{\cG^0}^1(X,Y)_\rat=0$ for any two irreducible $\cG^0$-submodules 
$X$, $Y$ of $V$. Assume the contrary: $\Ext_{\cG^0}^1(X,Y)_\rat \neq 0$ for some such $X$ and
$Y$. By Clifford's theorem, $Y \cong g(X)$ for some $g \in \cG$. Given a pair 
$(\cT,\cB)$ of a maximal torus $\cT$ and a Borel subgroup $\cB$ containing $\cT$ of $\cG^0$, we have that $g^{-1}(\cT,\cB)g = h^{-1}(\cT,\cB)h$ for a suitable $h \in \cG^0$.   
Replacing $g$ by $gh^{-1}$, we get that $g$ normalizes both $\cT$ and $\cB$, and $Y \cong g(X)$. 
Suppose that $X = L(\lambda)$ for some dominant weight $\lambda$ with respect to $\cT$.
Then $\tau(\lambda)$ is dominant and $Y= L(\tau(\lambda))$, where $\tau$ is the outer
automorphism (possibly trivial) of $\cG^0$ induced by $g$. 
By \cite[Proposition II.2.14]{jantzen}, $\Ext^1_{\cG^0}(X,Y)_\rat \neq 0$ implies that 
$\lambda \neq \tau(\lambda)$ but $\lambda$ and $\tau(\lambda)$ are 
comparable, say 
$\lambda > \tau(\lambda)$. Since $\tau$ fixes $(\cT,\cB)$, it fixes the set of 
positive roots with respect to $\cT$, whence 
$\tau^i(\lambda) > \tau^{i+1}(\lambda)$ for all $i \geq 0$. Also, note that 
the action of $\tau$ on the weight lattice $X(\cT)$ has finite order $N$. 
Thus we arrive at the chain
$$\lambda > \tau(\lambda) > \tau^2(\lambda) > \ldots > \tau^N(\lambda) = \lambda,$$
a contradiction.

\smallskip
(ii) Consider the action of the central torus $\cZ := \bfZ(\cG)^0$ 
on $V$ and decompose $V = V' \oplus [\cZ,V]$ with $V' := \bfC_V(\cZ)$. 
Then $H^n_\rat(\cG,[\cZ,V]) = 0$, and so $H^n_\rat(\cG,V) \cong H^n_\rat(\cG,V')$.
The same holds for Frobenius twists $V(e)$ of $V$; moreover,
$V'(e) \cong \bfC_{V(e)}(\cZ)$. Applying \cite[Theorem 6.6]{CPS} to the 
semisimple group $\cH := \cG/\cZ$ (and recalling that $\cZ$ is a torus), we get for 
$e$ and $f$ large enough that 
$$H^n_\rat(\cG,V'(e)) \cong H^n_\rat(\cH,V'(e)) \cong H^n(\cH(\bbF_{q^f}),V'(e)) 
    \cong  H^n(\cH(\bbF_{q^f}),V').$$
On the other hand, $\cH(\bbF_{q^f})$ is isomorphic to 
$\cG(\bbF_{q^f})/\cZ(\bbF_{q^f})$ (by the Lang-Steinberg theorem). Moreover, for 
$f$ large enough, $\cZ$ and $\cZ(\bbF_{q^f})$ have the same eigenspaces on $V$. It follows 
that $[\cZ,V] = [\cZ(\bbF_{q^f}),V]$ and $V' = \bfC_V(\cZ(\bbF_{q^f}))$, whence 
$$H^n(\cG(\bbF_{q^f}),V)=H^n(\cG(\bbF_{q^f}),V') = H^n(\cH(\bbF_{q^f}),V')$$
(as $\cZ(\bbF_{q^f})$ is a $p'$-group). The same holds for $V(e)$, and so the statement follows.
\end{proof}

\begin{lemma}\label{stable}
Let $p$ be a prime and let $k$ be an algebraically closed field of characteristic $p$.
Let $\cG$ be a connected reductive algebraic group over $k$ and $V$ be 
a rational $\cG$-module. If $H^1_\rat(\cG,V(e))=0$ for all Frobenius twists $V(e)$ of
$V$ with $e$ large enough, then $H^1_\abs(\cG(k),V)=0$.
\end{lemma}

\begin{proof}
If the result fails, then there exists a (possibly non-rational) $k\cG(k)$-module $W$ and
a non-split extension $0 \to V \to W \to k \to 0$.

Let $K$ be the algebraic closure of $\bbF_p$ in $k$. 
Note that $\cG$ can be defined over $\bbF_q \subset k$ for $q$ sufficiently large
(as it can be defined over $K$ by the isomorphism theorem for reductive groups).
Also, let $\cT(k)$ be a maximal torus of $\cG(k)$
containing a maximal torus $\cT(K)$ of $\cG(K)$.
For $e$ and $f$ large enough, we 
have by assumption and by Proposition \ref{disc}(ii) that 
$H^1(\cG(\bbF_{q^f}),V) =   H^1_\rat(\cG,V(e)) = 0$. It follows by Lemma \ref{closure} that $W$ is
split over $\cG(K)$, whence 
\begin{equation}\label{fixed-K}
  \dim_k \bfC_W(\cG(K)) = \dim_k \bfC_V(\cG(K)) +1, 
  ~\dim_k \bfC_W(\cT(K)) = \dim_k \bfC_V(\cT(K)) +1.
\end{equation}  
We claim that $\bfC_W(\cT(K)) = \bfC_W(\cT(k))$. Clearly, the fixed point subspace 
$U := \bfC_W(\cT(K))$ is 
$\cT(k)$-invariant. Also, since $V$ is a rational $k\cG(k)$-module, $\cT(k)$ acts trivially on 
$U \cap V = \bfC_V(\cT(K))$, which has codimension $1$ in $U$ by (\ref{fixed-K}). Thus, $\cT(k)$
maps into a unipotent  subgroup of $\GL(U)$. Note that $\cT(k)$ is $p$-divisible,
and so is any homomorphic image of it. It follows that $\cT(k)$ acts trivially on
$U$, as stated.

Thus, the fixed point subspace of $\langle \cT(k), \cG(K) \rangle$ on $W$ is the fixed point 
subspace of $\langle \cT(K),\cG(K) \rangle = \cG(K)$ on $W$. Observe that 
$\cG(k) = \langle \cT(k), \cG(K) \rangle$. (Indeed, if $U_\alpha(k) \supseteq U_\alpha(K)$ are root subgroups corresponding to a root $\alpha$ with respect to $\cT$, then $\cT(k)$ acts transitively 
on $U_\alpha(k) \setminus \{1\}$. Since $\cG(K) \supset U_\alpha(K)$ and 
$\cG(k)$ is generated by $\cT(k)$ and all the root subgroups $U_\alpha(k)$, the claim follows.)
Hence, $\bfC_W(\cG(k)) = \bfC_W(\cG(K))$ and so
$$\dim_k \bfC_W(\cG(k)) \geq \dim_k\bfC_V(\cG(k)) +1$$
by (\ref{fixed-K}). Hence $W$ is split as a $\cG(k)$-module, a contradiction.
\end{proof}

\begin{cor}\label{ext-abs}
Let $k$ be an algebraically closed field of characteristic $p$ and let $\cG$ be a reductive algebraic group over $k$. Let $V$ be an irreducible rational $k\cG(k)$-module.
Assume that $p \nmid [\cG:\cG^0]$. Then 
$$H^1(\cH(k),k) = \Ext_{\cH(k)}^1(V,V) =H^1(\cH(k),(V^* \otimes V)/k) = 0,$$ 
both as rational and discrete cohomology groups, and for both $\cH = \cG$, $\cG^0$.
\end{cor}

\begin{proof}
Since $p \nmid [\cG:\cG^0]$, it suffices to prove the statement for $\cG^0$.
As shown in Proposition \ref{disc}(i), $\Ext^1_{\cG^0}(V,V)_\rat = 0$.  
Also, $H^i_\rat(\cG^0,k) = 0$ for $i > 0$ by \cite[Corollary II.4.11]{jantzen}.
We have therefore shown that 
$H^1_\rat(\cG^0, k) = H^1_\rat(\cG^0, (V^* \otimes V)/k) = 0$.
The same applies to Frobenius twists. 
Hence $\Ext^1_{\cG^0(k)}(V,V)_\abs = 0$ and 
$H^1_\abs(\cG^0(k),k) = H^1_\abs(\cG^0(k),(V^* \otimes V)/k) = 0$ by Lemma \ref{stable}. 
\end{proof}

We finally show that adequacy holds over a sufficiently large field and also for (not necessarily 
connected) reductive algebraic groups (whether one uses rational cohomology or discrete cohomology in the definition).   Note
that if $p$ does divide $[\cG:\cG^0]$, then adequacy may fail (the spanning may fail as well
as the cohomological conditions even assuming that $p \nmid (\dim V)$ -- one can construct examples
precisely as in \cite{G2}).  

\begin{thm}\label{thm:asymp}
Let $k$ be an algebraically closed field $k$ of characteristic $p$. Let $\cG$ be a reductive algebraic group defined over $\bbF_q \subset k$  such that $p \nmid [\cG:\cG^0]$, and let $V$ be a 
finite-dimensional faithful irreducible rational $k\cG$-module. Then the following statements hold. 
\begin{enumerate}[\rm(i)]
\item  $(\cG,V)$ is adequate.  
\item Assume that every coset of $\cG^0$ in $\cG$ is defined over $\bbF_q$.
Then $(\cG(\bbF_{q^f}),V)$ is adequate for $f$ sufficiently large
(with $f$ possibly depending upon $V$). 
\end{enumerate}
\end{thm}

\begin{proof}
(a) Arguing precisely as in \cite{GHTT}, we see that the set of semisimple
elements in any coset of $\cG^0$ is Zariski dense in $\cG$.  It follows that
the linear span of semisimple elements of $\cG$ is Zariski dense in
the linear span of $\cG$ in $\End(V)$.   Thus,  the span of the semisimple
elements in $\cG$ is all of $\End(V)$.
  
Let $\cT \subset \cB$ be a maximal torus and a Borel subgroup of $\cG^0$ that are defined over 
$\bbF_{q}$. Choose $f$ large enough so that $\cT(\bbF_{q^f})$ has exactly the same weight spaces on 
$\End(V)$ and $V$ as does $\cT$.   Let $\cN := \bfN_\cG(\cT,\cB)$ denote the simultaneous normalizer of $(\cT,\cB)$ in $\cG$. 
Then $\cN \cap \cG^0 = \cT$,
hence every element of $\cN$ is semisimple (as 
$p \nmid [\cG:\cG^0]$). Conversely, if $g \in \cG$ is semisimple,
then by  \cite[7.5]{St} $g$ normalizes some pair $(\cT',\cB')$ of 
a maximal torus $\cT'$ contained in a Borel subgroup $\cB'$. We
deduce that 
\begin{equation}\label{union}
  \cG_\rss = \cup_{x \in \cG}x\cN x^{-1},  
\end{equation}
where $\cG_\rss$ denotes the set of semisimple elements in $\cG$.

Let $W$ be the linear span of 
$\cG(\bbF_{q^f})_\rss:=\cG(\bbF_{q^f}) \cap \cG_\rss$
in $\End(V)$. Then $W$ is $\cG(\bbF_{q^f})$-stable and hence in
particular $\cT$-stable (as $\cT$ and $\cT(\bbF_{q^f})$ have the 
same eigenspaces on $\End(V)$). Arguing as in \cite{Gapp},  
we see that $\langle T,  \cG(\bbF_{q^f}) \rangle$ is Zariski dense in $\cG$. 
It follows that $W$ is $\cG$-stable. Since 
$\cN/\cT \hookrightarrow \cG/\cG^0$ and every coset of $\cG^0$ is 
defined over $\bbF_q$, we deduce by Lang's theorem that 
$\cN = \cN(\bbF_{q^f}) \cdot \cT$. Moreover, as $\cT(\bbF_{q^f})$ and 
$\cT$ have the same eigenspaces on $V$, $\cT(\bbF_{q^f})$ and
$\cT$ span the same subspace of $\End(V)$. Now $W$ contains 
the span of $\cN(\bbF_{q^f}) \subset \cG(\bbF_{q^f})_\rss$, hence
contains the span of $\cN$. Since $W$ is $\cG$-stable,
we deduce from (\ref{union}) that $W$ contains the span of
$\cG_\rss$. Thus for $f$ sufficiently large we have that $W = \End(V)$;
in particular,  $\cG(\bbF_{q^f})$ acts absolutely irreducibly on $V$.

\smallskip
(b) From Corollary \ref{ext-abs} we get that 
$H^1_\abs(\cG(k),k) = H^1_\abs(\cG(k), (V^* \otimes V)/k) = 0$.
Together with (a), this implies (i).

\smallskip
We also have that $H^1_\rat(\cG^0,k) = H^1_\rat(\cG^0, (V^* \otimes V)/k) = 0$, and the same holds for all Frobenius twists. 
Applying Proposition \ref{disc}(ii) we obtain (for $f$ large enough) that 
$H^1(\cG^0(\bbF_{q^f}),k) = H^1(\cG^0(\bbF_{q^f}), (V^* \otimes V)/k) = 0$, whence $H^1(\cG(\bbF_{q^f}),k) = H^1(\cG(\bbF_{q^f}), (V^* \otimes V)/k) = 0$ as 
well since $p \nmid [\cG:\cG^0]$.
Hence (ii) holds.
\end{proof}

\end{document}